\documentclass[leqno,11pt]{amsart}
\usepackage{amssymb, amsmath}

\usepackage{xcolor,ulem,textcomp,MnSymbol}
\usepackage{url}
\usepackage[mathscr]{euscript}

\numberwithin{equation}{section}
\usepackage{enumerate}
\usepackage{graphicx}
\usepackage{cancel}

\theoremstyle{definition}
\theoremstyle{plain}
\newtheorem{thm}{Theorem}[section]
\newtheorem{Prop}[thm]{Proposition}

\newtheorem{Cor}[thm]{Corollary}
\newtheorem{defi}[thm]{Definition}

 \newtheorem{Thm}{Theorem}[section]
 \newtheorem{Rmk}[thm]{Remark}
 \newtheorem{Lem}[thm]{Lemma}
  \newtheorem{Def}[thm]{Definition}

\def\bbE {\mathbb {E}}
\def\bE {\mathbb {E}}
\def\bbP {\mathbb {P}}
 \def\N {\mathbb{N}}
 
  \def\KK {\mathbf{K}}

  \def\NN {\mathbf{N}}

  \def\SS {\mathbf{S}}

  \def\cJ {\mathcal{J}}
  
  \definecolor{darkgreen}{rgb}{0.0, 0.5, 0.0}

 \def\M {\mathbf{M}}
 \def\MM {\mathbf{M}}
\def\R {\mathbb{R}}

\def\T {\mathbb{T}}
\def\Z {\mathbb{Z}}

\def\bbV {\mathbb {V}}
 \def\LL {\mathbb{L}}

\def\cA {\mathcal{A}}
\def\cB {\mathcal{B}}
\def\cC {\mathcal{C}}
\def\cE {\mathcal{E}}

\def\cD {\mathcal{D}}

\def\cE {\mathcal{E}}
\def\EE {\mathcal{E}}
\def\cH {\mathcal{H}}
\def\GG {\mathcal{G}}
\def\cG {\mathcal{G}}
\def\cI  {\mathcal{I}}
\def\cL {\mathcal{L}}

\def\cM {\mathcal{M}}
\def\cN {\mathcal{N}}

\def\cP {\mathcal{P}}

\def\cZ {\mathcal{Z}}

\def\cR {\mathcal{R}}
\def\cT {\mathcal{T}}

\def\gs {{\sigma}}
\def\gp {{\varphi}}

\def\eps {{\varepsilon}}
\def\e {{\varepsilon}}

\def\indc {{\bf 1}}

\def\d {{\partial}}

\def\Otimes{\bigotimes}
\def\Otimes{\bigostar}

\newcommand{\sgn}{\operatorname{sign}}

\newcommand{\ba}{\begin{aligned}}
\newcommand{\ea}{\end{aligned}}

\newcommand{\be}{\begin{equation}}
\newcommand{\ee}{\end{equation}}

\numberwithin{equation}{section}

\begin{document}

 \title[Long-time derivation at equilibrium of the fluctuating Boltzmann equation] {Long-time derivation at equilibrium of the fluctuating Boltzmann equation} 
\author{Thierry Bodineau, Isabelle Gallagher, Laure Saint-Raymond, Sergio Simonella}
\begin{abstract} 
 
We study a hard sphere gas at equilibrium, and prove that  in the low density limit, the fluctuations converge to a Gaussian process  governed by the fluctuating Boltzmann equation.   This result holds for arbitrarily long times. The method of proof builds upon  the weak convergence method introduced in the companion paper~\cite{BGSS2} which is improved by considering clusters of pseudo-trajectories as in \cite{BGSS3}. \end{abstract} 

\maketitle

\section{Introduction}

In this paper, we   prove that dynamical fluctuations in the empirical measure of a hard sphere gas at equilibrium are governed,  in the low density limit (Boltzmann-Grad limit), by the fluctuating Boltzmann equation, and this for arbitrarily long  kinetic times.
In particular, we show that the limiting process is Gaussian. The fluctuating equation is a stochastic equation given by a linearized Boltzmann collision operator, forced by a Gaussian noise, white in space and time, whose structure can be predicted by a fluctuation-dissipation argument~\cite{S2}.

The convergence of the covariance of the fluctuations was proved for short times in \cite{BLLS80} and extended to non-equilibrium states in \cite{S81}.  Moreover, the Gaussian character of the limiting field and the fluctuating equation were conjectured in \cite{EC81,S81,S83}.
This conjecture was reconsidered and proved to be true in \cite{BGSS1,BGSS3}, where the convergence of the full fluctuation process was obtained by using cumulant techniques, away from equilibrium, together with (much stronger) large deviation bounds.

All these results are severely limited to a small interval of time, exactly in the same way as for the validity of the nonlinear Boltzmann equation, as proved in \cite{La75}.
This short time limitation was removed first for the covariance of the fluctuation field in \cite{BGSR2}, in the case of a two-dimensional gas of hard disks at equilibrium. The limiting covariance is governed by the Boltzmann equation, linearized around the Maxwellian distribution.
In the companion paper~\cite{BGSS2}, a more robust weak convergence method was introduced: taking advantage of the invariant measure, we discarded atypical dynamics (preventing the convergence) by localizing pathological behaviors and using a time decoupling. This allows to extend  the previous result to arbitrary dimensions.
The present paper elaborates upon the same strategy:  we study the higher order moments of the fluctuation field and prove that they asymptotically factorize according to Gaussian rules. 

\subsection{The model}
We   consider here exactly the same setting as in \cite{BGSS2}, of which we recall the notations. The microscopic model consists of  identical hard spheres of unit mass and of diameter~$\eps$.
The  motion of $N$ such hard spheres is ruled by a system of ordinary differential equations, which are set in~$ ( \T ^d\times\R^d)^{N }$ where~$\mathbb T^d$ is  the unit~$d$-dimensional periodic box with~$d \geq 2$: writing~${\bf x}^{\e}_i \in  \T ^d$ for the position of the center of the particle labeled by~$i$ and~${\bf v}^{\e}_i \in  \R ^d$ for its velocity, one has
\begin{equation}
\label{hardspheres}
{d{\bf x}^{\e}_i\over dt} =  {\bf v}^{\e}_i\,,\quad {d{\bf v}^{\e}_i\over dt} =0 \quad \hbox{ as long as \ } |{\bf x}^{\e}_i(t)-{\bf x}^{\e}_j(t)|>\eps  
\quad \hbox{for \ } 1 \leq i \neq j \leq N
\, ,
\end{equation}
with specular reflection at collisions: 
\begin{equation}
\label{defZ'nij}
\begin{aligned}
\left. \begin{aligned}
 \left({\bf v}^{\e}_i\right)'& := {\bf v}^{\e}_i - \frac1{\eps^2} ({\bf v}^{\e}_i-{\bf v}^{\e}_j)\cdot ({\bf x}^{\e}_i-{\bf x}^{\e}_j) \, ({\bf x}^{\e}_i-{\bf x}^{\e}_j)   \\
\left({\bf v}^{\e}_j\right)'& := {\bf v}^{\e}_j + \frac1{\eps^2} ({\bf v}^{\e}_i-{\bf v}^{\e}_j)\cdot ({\bf x}^{\e}_i-{\bf x}^{\e}_j) \, ({\bf x}^{\e}_i-{\bf x}^{\e}_j)  
\end{aligned}\right\} 
\quad  \hbox{ if } |{\bf x}^{\e}_i(t)-{\bf x}^{\e}_j(t)|=\eps\,.
\end{aligned}
\end{equation}
 This flow does not cover all possible situations, as multiple collisions are excluded.
But one can show (see \cite{Ale75}) that for almost every admissible initial configuration~$({\bf x}^{\e 0}_i, {\bf v}^{\e 0}_i)_{1\leq i \leq N}$, there are neither multiple collisions,
nor accumulation of collision times, so that the dynamics is globally well defined. 

We will not be interested here in one specific realization of the dynamics,  but rather in a statistical description. This is achieved by introducing a measure at time 0, on the phase space we now specify.
The collections of~$N$ positions and velocities are denoted respectively by~$X_N := (x_1,\dots,x_N) $ in~$ \T^{dN}$ and~$V_N := (v_1,\dots,v_N) $ in~$ \R^{d N}$,  and we set~$Z_N:= (X_N,V_N) $, with~$Z_N =(z_1,\dots,z_N)$, $z_i = (x_i,v_i)$. A set of~$N$ particles is characterized by a random variable ${\mathbf Z}^{\eps 0}_N =  ({\mathbf z}^{\eps 0}_1,\dots , {\mathbf z}^{\eps 0}_N)$, ${\mathbf z}^{\eps 0}_i = \left({\mathbf x}^{\eps 0}_i,{\mathbf v}^{\eps 0}_i\right)$ specifying the time-zero configuration in the phase space 
\begin{equation}
\label{D-def}
{\mathcal D}^{\eps}_{N} := \big\{Z_N \in (\T ^d\times\R^d)^N \, / \,  \forall i \neq j \, ,\quad |x_i - x_j| > \eps \big\} \, ,
\end{equation}
and an evolution according to the deterministic flow \eqref{hardspheres}-\eqref{defZ'nij} (well defined with probability~1)
$$ t \mapsto {\mathbf Z}^{\eps}_N(t) = \big({\mathbf z}^{\eps}_1(t),\dots , {\mathbf z}^{\eps}_N(t)\big)\;,\qquad t>0\;,$$
with~${\mathbf z}^{\eps}_i\left(t\right) = \left({\mathbf x}^{\eps}_i\left(t\right),{\mathbf v}^{\eps}_i\left(t\right)\right)$.

To avoid spurious correlations due to a given total number of particles, we   actually consider a grand canonical state (as in \cite{Ki75,BLLS80}), living on the phase space 
$$
{\mathcal D}^{\eps} := \bigcup_{N \geq 0}{\mathcal D}^{\eps}_{N}
$$
(notice that ${\mathcal D}^{\eps}_{N} = \emptyset$ for $N$ large). This means that the total number of particles is also a random variable,  which we shall denote by $\cN$.
In the low density regime, referred to as  the Boltzmann-Grad scaling, the density (average $\cN$) is tuned by the parameter $$\mu_\e := \e^{-(d-1)}\;,$$ ensuring that the mean free path between collisions is of order one \cite{Gr49}. 

More precisely, at equilibrium the probability density of finding $N$ particles in $Z_N$ is given by
\begin{equation}
\label{eq: initial measure}
\frac{1}{N!}W^{\eps, {\rm eq} }_{N}(Z_N) 
:= \frac{1}{\cZ^ \eps} \,\frac{\mu_\eps^N}{N!} \, \indc_{
{\mathcal D}^{\eps}_{N}}(Z_N) 
 \, \cM^{\otimes N}  (V_N)\, , \qquad \hbox{ for } N= 0, 1,2,\dots \end{equation} 
 with~
 \begin{equation*}
\label{eq: max}
  \cM(v) := \frac1{(2\pi)^\frac d2} \exp\big( {-\frac{ |v|^2}2}\big)\,,
\qquad \cM^{\otimes N}(V_N)= \prod_{i=1}^N \cM(v_i)\, , 
\end{equation*}
and the partition function is given by
\begin{equation}
\label{eq: partition function}
\cZ^\eps :=  1 + \sum_{N\geq 1}\frac{\mu_\eps^N}{N!}  
\int_{{\mathcal D}^{\eps}_{N}
}  \cM^{\otimes N}(V_N)dX_N\,dV_N = 1 + \sum_{N\geq 1}\frac{\mu_\eps^N}{N!}   \int_{\T^{dN}}  \Big( \prod_{i\neq j}\indc_{ |x_i - x_j| > \eps} \Big)dX_N\,.
\end{equation}
Here and below, $\indc_A$ will be the characteristic function of the set $A$, and we will also use the symbol $\indc_*$
for the characteristic function of the set defined by condition $*$. Notice that, for notational convenience, we work with functions extended to zero outside ${\mathcal D}^{\eps}_{N}$.

In the following, the probability of an event~$A$ with respect to the equilibrium measure~(\ref{eq: initial measure}) will be denoted~${\mathbb P}^{\rm eq}_\eps(A)$, and~${\mathbb E}^{\rm eq}_\eps$ will be the expected value.
Definition \eqref{eq: initial measure} ensures that $$\mu_\e^{-1}{\mathbb E}^{\rm eq}_\eps\left(\cN\right)\to 1$$ as $\mu_\e\to \infty$, as required.

\subsection{State of the art}

Consider the  empirical density of the hard-sphere model:
\begin{equation}
\label{eq: empirical}
\pi^{\eps}_t :=\frac{1}{\mu_\e}\sum_{i=1}^\cN \delta_{{\bf z}^\e_i(t)} \;.
\end{equation} 
Under the initial grand canonical measure 
\begin{equation}
\label{initial data out of equilibrium}  
\frac{1}{\cZ^ \eps(f^0)} \,\frac{\mu_\eps^N}{N!} \, \indc_{
{\mathcal D}^{\eps}_{N}}(Z_N) 
 \, (f^0)^{\otimes N}  (Z_N)\, , \qquad \hbox{ for } N= 0, 1,2,\dots
 \end{equation}
where  $f^0$ is a smooth and fast (Gaussian) decaying density and ${\cZ^ \eps(f^0)} $ the corresponding partition function, it has been proved by Lanford in \cite{La75} that, in the Boltzmann-Grad limit $\mu_\eps \to \infty$,  $\pi^{\eps}_t$ concentrates on the 
 solution of the Boltzmann equation
\begin{equation}
\left\{ \begin{aligned}
& \d_t f +v \cdot \nabla _x f = \! \displaystyle\int_{\R^d}\int_{{\mathbb S}^{d-1}}   \! \Big(  f(t,x,w') f(t,x,v') - f(t,x,w) f(t,x,v)\Big) 
\big ((v-w)\cdot \omega\big)_+ \, d\omega \,dw \,  ,\\
&  f(0,x,v) = f^0(x,v)
 \end{aligned}
  \right. \label{eq:Beq}
  \end{equation}
   where the precollisional velocities $(v',w')$ are defined by the scattering law
\begin{equation}\label{scattlaw}
 v' := v- \big( (v-w) \cdot \omega\big)\,  \omega \,  ,\qquad
 w' :=w+\big((v-w) \cdot \omega\big)\,  \omega  \, .
\end{equation}
More precisely, there exists a short time $T_L>0$ depending only on $f^0$, such that for any test function $h: \T ^d\times\R^d\to\R$ and any $\delta>0$, $t \in [0, T_L] $,
\begin{equation}\label{LLNf0}
\bbP_\eps \left( \Big|\pi^\eps_t(h) -  
   \int_{\T ^d\times\R^d}dz f(t,z)h(z)\Big| > \delta \right) \xrightarrow[\mu_\eps \to \infty]{} 0\;,
\end{equation}
which can be interpreted as a law of large numbers; see e.g.~\cite{IP89,S2,CIP94,GSRT, PS17, denlinger, BGSRS18,GG18,GG20}.

 \medskip
Using the invariance of the equilibrium measure (\ref{eq: initial measure}), it is not hard to see that in our setting $\pi^{\eps}_t$ concentrates on~$\cM$, which is a stationary solution of the Boltzmann equation:
for any test function $h: \T ^d\times\R^d\to\R$ and any $\delta>0$, $t \in \R$,
\begin{equation}\label{LLN}
\bbP^{\rm eq}_\eps \left( \Big|\pi^\eps_t(h) -  
   \int_{\T ^d\times\R^d} dz \cM (v)h(z)\Big| > \delta \right) \xrightarrow[\mu_\eps \to \infty]{} 0\;.
\end{equation}
Our purpose here is to study the fluctuations of the empirical density $\pi^{\eps}_t$ from its equilibrium value. In this regime,  the collision operator in Eq.\,\eqref{eq:Beq} is expected to reduce to the linearized operator (according to $f = \cM+g$)
\begin{equation} 
\label{eq:LBCO}
\begin{aligned}
   & \cL g := - v \cdot \nabla_x g \\
   & \hspace{0,8cm}+\int_{\R^d \times {\mathbb S}^{d-1}}\left[\cM(w') g(v') +  g(w')\cM(v') - \cM(w) g(v) -  g(w)\cM(v)\right]  \big( (v - w) \cdot \omega \big)_+  d w\,d \omega\;.
   \end{aligned}
 \end{equation}
Moreover, such fluctuations should be of size $1 / \sqrt\mu_\e$, which leads to define the fluctuation field $\zeta^{\eps,\rm eq}$ by 
\begin{equation}
\label{eq: fluctuation field}
\zeta^{\eps, \rm eq}_t \big(  h  \big) :=  { \sqrt{\mu_\eps }} \, 
\Big( \pi^\eps_t(h) -  {\mathbb E}^{\rm eq}_\eps\big( \pi^\eps_t(h) \big)     \Big)\;,
 \end{equation}
 for any test function $h$.

 Making the analysis of \eqref{LLN} slightly more quantitative one easily proves that, for any given~$t$, $\zeta^{\eps,\rm eq}_t$ converges in law towards a Gaussian white noise $\zeta$ with zero mean and covariance
\begin{equation} \label{eq:idf}
\bbE \Big( \zeta (h^{(1)}) \zeta (h^{(2)}) \Big) = \int_{\T ^d\times\R^d} h^{(1)}(z) h^{(2)}(z) \cM(v) dz\;.
\end{equation}
Much more interesting is the analysis of time-correlations, i.e.\,of time-dependent products such as ${\mathbb E}^{\rm eq}_\eps\left(\zeta^{\eps,\rm eq}_t \big( h^{(1)}  \big) \zeta^{\eps,\rm eq}_0 \big( h^{(2)}\big) \right)$, which involves an accurate understanding of the collisional processes in the hard-sphere dynamics. A local-in-time result for the covariance of the fluctuation field was obtained in \cite{BLLS80}, by a direct application of the method of \cite{La75} (as  discussed below, this was extended in~\cite{BGSR2} and~\cite{BGSS2}  to large times close to equilibrium). The techniques of dynamical clusters introduced in \cite{BGSS3} allowed then to extend this short time convergence   to moments of arbitrary order, as well as to establish  the tightness property of the fluctuation process  (see also \cite{BGSS1} for a less technical presentation of this method).

Recall that \eqref{eq:LBCO} is well-defined in~$L^2$. We also introduce
$$
L^2_\cM:=  \Big \{
g : \T^d\times \R^d \rightarrow \R \, , \qquad \|g\|_{L^2_\cM}:= \Big( \int_{ \T^d\times \R^d} |g|^2 \, \cM dxdv\Big)^\frac12 < \infty
\Big\} 
$$
and the Hilbert space   indexed by~$k \in \Z$  
$$
{\mathcal H}^k:=  \Big \{
g : \T^d\times \R^d \rightarrow \R \, , \qquad  \, \|g\|_{ {\mathcal H}^k}:= \big\|(\mbox{Id} - \Delta_v +|v|^2- \Delta_x)^kg\big\|_{L^2_\cM}<\infty
\Big\}  \, .
$$

\begin{Thm}  [{\bf Short time convergence of the fluctuating field}, \cite{BLLS80,BGSS3}]
\label{thmTCL}
Consider a system of hard spheres at equilibrium  in a~$d$-dimensional periodic box with~$d\geq 2$.  
There exists a time $T^\star>0$ such that, in the Boltzmann-Grad limit~$\mu_\eps \to \infty$, the following properties hold true.

(a) Let~$g^0$ and~$h$ be two functions in~$ L^2_\cM$.
The covariance of the fluctuation field~$\left(\zeta^{\eps,\rm eq}_t\right)_{t \in [0,T^\star]}$ converges:
\begin{equation} 
\label{eq:convcov}
\bbE^{\rm eq}_\eps \Big[ \zeta^{\eps,\rm eq}_0 (g^0) \; \zeta_t^{\eps,\rm eq} (h) \Big] \to \int \cM\, g (t)\, h\, dxdv\;,
\end{equation} 
where $\cM g$ is the solution of the linearized Boltzmann equation~$\partial_t \cM g  =\cL \cM g $, with~$  g_{| t = 0} = g^0$.

(b) There exists~$k>0$ such that  the family of processes~$\left(\zeta^{\eps,\rm eq}_t\right)_{t \in [0,T^\star]}$ is tight in the Skorokhod space $D\left([0,T^\star], \cH^{-k} \right)$. More precisely, 
\begin{equation}
\label{eq: tightness}
\begin{aligned}
&  \lim_{\delta \to 0^+} \lim_{\mu_\eps \to \infty} 
\bbP^{\rm eq}_\varepsilon \Big[ \sup_{ |s -t| \leq \delta \atop s,t  \in [0,T^\star] }
\big\| \zeta^{\varepsilon,\rm eq}_t  - \zeta^{\varepsilon,\rm eq}_s \big \|_{-k} \geq \delta'
\Big] = 0 \,, \qquad \forall \delta'>0 \,,  \\
& 
\lim_{A \to \infty} \lim_{\mu_\eps \to \infty} \bbP^{\rm eq}_\varepsilon \Big[ \sup_{ t  \in [0,T^\star]}
\big \| \zeta^{\varepsilon,\rm eq}_t \big \|_{-k} \geq A \Big] = 0\, .
\end{aligned}
\end{equation}

(c) The fluctuation field $\left(\zeta^{\eps,\rm eq}_t\right)_{t \in [0,T^\star]}$ converges in law to the (weak) solution of the fluctuating Boltzmann equation
\begin{equation}
\label{eq: OU}
d \zeta_t   = \cL \,\zeta_t\, dt + d\eta_t\,,
\end{equation}
with initial datum \eqref{eq:idf}, where $d \eta_t(x,v)$ is a stationary Gaussian noise.
 \end{Thm}

The noise in the above equation is explicitly characterized (see \cite{S83}). It has zero mean and covariance
(for all $T>0$)
\begin{equation}
\label{eq: noise covariance}
\begin{aligned}
& \bbE \left( \int_0^T   dt_1  \int dz_1   h^{(1)} (z_1)  \eta_{t_1} (z_1)     \int_0^T   dt_2 \int dz_2 \, h^{(2)} (z_2) \eta_{t_2} (z_2) \right)
 \\ &\qquad \qquad= \frac{1}{2} \int_0^T   dt  \int d\mu(z_1, z_2, \omega) 
 \cM(v_1)\, \cM(v_2) \Delta h^{(1)}  \, \Delta h^{(2)} 
\end{aligned}
\end{equation}
denoting
$$
d\mu (z_1, z_2, \omega): = \delta_{x_1 - x_2}  \,  \big( (v_1 - v_2) \cdot \omega \big)_+ d \omega\, d v_1\, d v_2 dx_1 dx_2
$$
and defining  for any $h$
$$  \Delta h (z_1, z_2, \omega) := h(z_1') +  h(z_2') -  h(z_1) -  h(z_2)\,, $$
where~$z_i':=(x_i,v_i')$ with notation~(\ref{scattlaw}) for the velocities obtained upon scattering. Note that this noise is white in time and space, but correlated in  
velocities.

 Weak solutions of Eq.\,\eqref{eq: OU} have been discussed in \cite{BGSS3}.
They are martingale solutions defined according to a classical procedure \cite{holley1978generalized}.

\begin{Rmk}  Theorem~{\rm\ref{thmTCL}} has been generalized out of equilibrium (see~\cite{S81} for part (a) and~\cite{BGSS3} for parts (b)-(c)) for initial measures of type
(\ref{initial data out of equilibrium}). The fluctuation field is still defined by  Eq.\,\eqref{eq: fluctuation field}. The fluctuating Boltzmann equation is the linearized Boltzmann equation around the solution $f(t)$ of the Boltzmann equation with initial datum~$f^0$, forced by a noise 
with a time-dependent covariance of the form~\eqref{eq: noise covariance}, with $\cM$ replaced by $f(t)$.

For more discussions on physical aspects of the fluctuation theory at low density, and on related mathematical results, we refer to    
\cite{ EC81, KL76, meleard, Rez, S81, S83} (see also \cite{BGSS3} and the references therein).

\end{Rmk}

Using the invariant measure, it was shown in \cite{BGSR2} in two space dimensions, and~\cite{BGSS2}  in higher dimensions, that the convergence result~\eqref{eq:convcov}  actually holds for all times.  \begin{Thm}  [{\bf Long time convergence of the fluctuating field},  \cite{BGSS2}
]
\label{thmTCLglob}
Consider a system of hard spheres at equilibrium  in a~$d$-dimensional periodic box with~$d\geq 3$.   Let~$g^0$ and~$h$ be two functions in~$ L^2_\cM$.
Then in the Boltzmann-Grad limit~$\mu_\eps \to \infty$,   the covariance of the fluctuation field~$\left(\zeta^{\eps,\rm eq}_t\right)_{t\geq 0}$ converges for all times
$$
\forall t \geq 0 \, , \quad \bbE^{\rm eq}_\eps \Big[ \zeta^{\eps,\rm eq}_0 (g^0) \; \zeta_t^{\eps,\rm eq} (h) \Big] \to \int \cM\, g (t)\, h\, dxdv\;,
$$
where $\cM g$ is the solution of the linearized Boltzmann equation~$\partial_t \cM g  =\cL \cM g $, with~$  g_{| t = 0} = g^0$.
 \end{Thm}

\subsection{Statement of the result}
Our goal  in this paper is  to build  upon the techniques introduced in~\cite{BGSS2}  to extend the validity of Theorem \ref{thmTCL} to arbitrarily large times.
To reach longer time scales, we devise a method of proof different from the one in \cite{BGSS3}: we   actually combine the cumulant technique of  \cite{BGSS3} (controlling locally the small correlations induced by the hard-sphere dynamics) with the weak convergence method introduced in \cite{BGSS2}, allowing to make an efficient use of the invariant measure  and thus providing   the long time convergence of the covariance of the fluctuation field.

We remark preliminarily, that at equilibrium, the tightness property \eqref{eq: tightness} can be readily generalized to arbitrary times.
Indeed splitting an arbitrary time interval $[0, \Theta]$ into 
subintervals of length $T^\star$, using a union bound and the time invariance of the equilibrium measure, we get
\begin{align}
\label{eq: long time tightness}
\bbP^{\rm eq}_\varepsilon \Big[ \sup_{ |s -t| \leq \delta \atop s,t  \in [0,\Theta] }
\big\| \zeta^{\varepsilon,\rm eq}_t  - \zeta^{\varepsilon,\rm eq}_s \big \|_{-k} \geq \delta'
\Big]
\leq 2 \frac{\Theta}{T^\star}
\bbP^{\rm eq}_\varepsilon \Big[ \sup_{ |s -t| \leq \delta \atop s,t  \in [0,T^\star] }
\big\| \zeta^{\varepsilon,\rm eq}_t  - \zeta^{\varepsilon,\rm eq}_s \big \|_{-k} \geq \delta'
\Big]\;.
\end{align}
Thus the short time fluctuations in  $[0, \Theta]$  can be controlled by  \eqref{eq: tightness}, and the same is true for the norm $\|_\cdot \|_{-k}$ of the field.

Our goal here is thus to extend the result $(c)$ of Theorem~\ref{thmTCL}, and this will be done by going  one step further in the weak convergence method of~\cite{BGSS2} looking  at clusters of trajectories to identify the fluctuation structure, and to combine it with a suitable iteration procedure, allowing to extend the convergence result \eqref{eq:convcov} to moments of the fluctuation field of order~$P>2$. We shall prove that such moments are vanishing for~$P$ odd, and converging to sums of products of covariances for $P$ even, in agreement with the Wick rule. This, together with the tightness property, will imply the convergence to the fluctuating Boltzmann equation.

\begin{Thm}  [{\bf  Long time convergence of the fluctuating field}]
\label{thmTCL'}
Consider a system of hard spheres at equilibrium  in a~$d$-dimensional periodic box with~$d\geq 3$.  
Then,
in the Boltzmann-Grad limit~$\mu_\eps \to \infty$, the fluctuation field $\left(\zeta^{\eps,\rm eq}_t\right)_{t \geq 0}$ converges in law for all times to the solution of the fluctuating Boltzmann equation 
\begin{equation}
\label{eq: OU'}
d \zeta_t   = \cL \,\zeta_t\, dt + d\eta_t\,,
\end{equation}
with initial datum \eqref{eq:idf}.
 \end{Thm}
  
\begin{Rmk} 
For simplicity we have chosen to state and prove the result in dimension~$d\geq 3$ only: the geometric arguments
needed to control pathological trajectories   require indeed a specific treatment in  the case of two space dimensions and we prefer to leave out these additional technicalities in this paper. 
\end{Rmk}
\begin{Rmk} The  moments of the fluctuation field remain actually under control on time intervals with size 
diverging slowly with~$\eps$, as~$O \big( \log \log| \log \eps| \big) $, as will be made clear by the quantitative convergence estimate  in Proposition {\rm\ref{prop: appariement}}. Note that in this regime the hydrodynamical limit holds true as shown in~\cite{BGSR2}.
The derivation of the fluctuating hydrodynamics of a Boltzmann gas is  the subject of the third companion paper \cite{BGSS4}
(see also Chapter~{\rm 7} of \cite {S2} for the general structure of the fluctuation theory).
\end{Rmk}

The proof of Theorem \ref{thmTCL'} requires several iterative steps, involving different time scales. 
A generalized fluctuation structure  will first be obtained on a very small time scale~$\delta$ in Section~\ref{sec:Duhamel}, and iterated to reach intermediate (small, of size~$\tau$) and macroscopic times in Sections
\ref{sec:tau} and \ref{fluctuation-sec}. This requires in particular 
an elimination of small remainder terms encoding unlikely events, namely recollisions at scale $\delta$ and superexponential growth at scale $\tau$. 
The main term of the iteration will be finally 
shown to converge to a Gaussian pairing in Section \ref{main-sec}. 

This intricated time sampling will be discussed first informally in Section \ref{sec:strategy}~: we will see that at each time scale,  remainders of very different nature are identified. 
It will  be very important that  all these remainders share a structure of products of fluctuation fields (as defined in (\ref{tensor-product})), for which we can establish general estimates on the expectation and covariance  (see Sections \ref{geometric-sec}-\ref{variance-sec}).


\section{ Elements of strategy}
\label{sec:strategy}

The aim of this section is to provide an informal description of the procedure necessary to extend a convergence  result of type \eqref{eq:convcov}, to arbitrary moments
\begin{equation} \label{eq:IPepseq}
I^{\eps, \rm eq}_P := \bbE^{\rm eq}_\eps \left[ \prod_{p=1} ^P \zeta^{\eps,\rm eq}_{\theta_p} ( h^{(p)}) \right]\,, \quad P\geq 3\;,
\end{equation}
where $ (\theta_1,\dots,\theta_P)$ is a collection of times with~$
0 = \theta_1 \leq \theta_2 \leq \dots \leq \theta_P =: \Theta \,,
$
and~$ (h^{(1)},\dots,h^{(P)})$ a collection of test functions. 
We shall restrict to $L^{\infty}$ test functions from now on.

\medskip

Our main result is  the following. Here and in what follows we use the notation~$A \ll 1$ to indicate that~$A$ goes to zero when~$\mu_\eps$ goes to infinity.
\begin{Prop}
\label{prop: appariement}
For any $P\geq 2$, denote by ${\mathfrak S}_P^{\rm pairs}$ the set of partitions of $\{1, \dots , P\}$ made only of pairs. 
Then asymptotically when $\mu_\eps \to \infty$, the moments are determined by the covariances, in the following sense: for~$\Theta\geq 1 $ and~$\tau>0 $ satisfying
\begin{equation}
\label{tau-conditions}
\Theta^{3(P-1)} \tau \ll1\quad \hbox{ and} \quad \Theta / \left(\tau \log|\log\eps|\right) \ll 1,
\end{equation}
there holds uniformly in~$\theta_1, \dots \theta_P \in [0, \Theta]$
\begin{equation} 
\label{eq: appariement} 
\begin{aligned}
&\left| \bbE^{\rm eq}_\eps \left[ \prod_{p=1} ^P \zeta^{\eps,\rm eq}_{\theta_p} ( h^{(p)}) \right]
-  \sum_{\eta \in {\mathfrak S}_P^{\rm pairs}}
\ \prod_{\{i,j\}  \in \eta} \bbE^{\rm eq}_\eps \left[  \zeta^{\eps,\rm eq}_{\theta_i} ( h^{(i)}) \; \zeta^{\eps,\rm eq}_{\theta_j} ( h^{(j)})  \right] \right|\\
& \qquad  \leq  \Big( \prod_{p = 1} ^P \| h^{(p)}\|_{L^\infty}  \Big) \Big( C_P\tau^{1/2}  \Theta^{(2P-1)/2}  +(C_P\Theta) ^{2^{\Theta /  \tau}} \eps ^{1/8d} \Big)
\end{aligned}
\end{equation}
for some constant $C_P$ depending only on $P$.
Notice that if $P$ is odd then ${\mathfrak S}_P^{\rm pairs}$ is empty and the product of moments is asymptotically 0.
\end{Prop} 
As the limit of the covariance has been computed in Theorem \ref{thmTCL} and extended to the time interval $[0,\Theta]$ in \cite{BGSS2},  
Proposition \ref{prop: appariement} fully determines the limiting moments which 
turn out to coincide with the Gaussian moments of the solution to the fluctuating Boltzmann equation~\eqref{eq: OU'}. By the `Moment Method' (see \cite{Billingsley 2} Section~30, Theorem 30.1), this fully characterizes the limiting distribution. Combined with the tightness  of the process (see \eqref{eq: long time tightness} and~\eqref{eq: tightness}), this completes the proof of the convergence to the fluctuating Boltzmann equation stated in Theorem \ref{thmTCL'}.

 \subsection{The global pairing scheme}\label{pairing scheme}
For simplicity, we will assume that all evaluation times are different 
$$
0 = \theta_1 < \theta_2 < \dots <\theta_P =: \Theta \, .
$$ 
The idea is then  to  design an iteration scheme decreasing the parameter $P$, where each elementary step uses the weak convergence method introduced in \cite{BGSS2}, and which realizes progressively the pairing (Wick rule), up to small remainder terms.

Let us focus on the moments $I^{\eps,\rm eq}_P$ defined in \eqref{eq:IPepseq} 
and consider the first step of our iteration procedure.
Our goal  is to reduce the number of evaluation times by transforming the fluctuation at time $ \theta_P$  into a sum of (more complicated) fluctuations at time~$\theta_{P-1}$
\begin{equation}\label{eq: first step eq} I^{\eps,\rm eq}_P  =  \sum_{m_P} \bbE_\eps^{\rm eq} \left[ \Big( \prod_{p=1} ^{P-1} \,  \zeta^{\e,\rm eq}_{\theta_p} ( h^{(p)}) \Big) \, 
   \zeta^{\e,\rm eq}_{m_P,\theta_{P-1} } \left(\phi_{\theta_P-\theta_{P-1}} ^{(P)} \right)  \right]\;,  
   \end{equation} 
 where $\zeta^{\e,\rm eq}_{m_P,\theta_{P-1} } \left(\phi_{\theta_P-\theta_{P-1}} ^{(P)} \right)$ is a fluctuation of an observable involving $m_P$ particles at time~$\theta_{P-1}$ (the notation is defined below in (\ref{eq: def zeta m eq})).
The function~$\phi_{\theta_P-\theta_{P-1}} ^{(P)} = \phi^{(P)}_{ \theta_P-\theta_{P-1}}[h^{(P)}](Z_{m_P})$ of~$m_P$ particles is given by the elementary test function $h^{(P)}$ pulled back,  along the flow induced by the Duhamel formula, during a time  $\theta_P-\theta_{P-1}$. We refer to Section \ref{sec:Duhamel} below for a detailed discussion, and focus now on the iteration procedure.
  
  \medskip

For this it is convenient to use the following extension of Definition~(\ref{eq: empirical}) of the empirical distribution: for any integer~$m\geq 1$, any time~$t $  and any test function~$H_m$ defined on~$(\T ^d\times\R^d)^m$, we set
\begin{equation*}
\label{eq: fluct field m}
    \pi^\e_{m,t}(   H_m):= \frac1{\mu_\eps^m} \sum_{(i_1,\dots,i_m)}H_m \big({\mathbf z}^{\eps}_{i_1}(t ),\dots , {\mathbf z}^{\eps}_{i_{m}}(t )\big)\, , 
\end{equation*}
where  the symbol $\sum_{(i_1,\dots,i_m)}$ indicates  a sum over $m$-tuples of labels  running from $1$ to $\cN$ which are all mutually different ($i_j \neq i_k$ for $k\neq j$).  Note that $ \pi^\e_{1,t} = \pi^{\eps}_t$ according to~(\ref{eq: empirical}), but $\pi^\eps_{m, t} \neq (\pi^\eps_t) ^{\otimes m}$ for $m>1$ since indices cannot be repeated. It will still be convenient to maintain a product notation by introducing a product symbol, the so-called  ``$\ostar$-product", which by definition takes into account  non-repeated indices, so that we can write
$$ \pi^\eps_{m,t} =  (\pi^\eps_t) ^ {\ostar m}\,.$$
Finally we introduce
 the shorthand notation
$$ \bbE_\eps^{\rm eq} [H_m]  :=   \bbE_\eps^{\rm eq} [    \pi^\eps _{m,t}(   H_m)] \,,$$
and extend  Definition~\eqref{eq: fluctuation field}  
of the fluctuation field:
\begin{equation}
\label{eq: def zeta m eq}
\zeta^{\e,\rm eq}_{m,t}(   H_m):=\sqrt{\mu_\eps } \Big(   \pi^\eps _{m,t}(   H_m) - \bbE^{\rm eq}_\eps [H_m]\Big)  \, .
\end{equation}

Returning to~(\ref{eq: first step eq}) let us  consider the product
$$
 \zeta^{\eps,\rm eq}_{\theta_{P-1}} ( h^{(P-1)})   \zeta^{\e,\rm eq}_{m_P,\theta_{P-1} }\left(\phi_{\theta_P-\theta_{P-1}} ^{(P)} \right)$$
and look at the repeated indices to decompose it into two contributions, with the above notation:
\begin{itemize}
\item a {\color{black}``$\ostar$-product"}, which by definition takes into account the non-repeated indices, 
\begin{equation*}\label{magic tensor product}
\begin{aligned}
 &\zeta^{\e,\rm eq}_{\theta_{P-1}} ( h^{(P-1)})  \ostar \zeta^{\e,\rm eq}_{m_P,\theta_{P-1}} \left(\phi_{\theta_P-\theta_{P-1}} ^{(P)} \right)\\
 &={\color{black}\mu_\eps \left( {1\over \mu_\eps }\sum h^{(P-1)}   - \bbE_\eps ^{\rm eq} [ h^{(P-1)} ] \right)\ostar  \left( {1\over \mu_\eps^{m_P}} \sum \phi_{\theta_P-\theta_{P-1}} ^{(P)}    - \bbE_\eps ^{\rm eq}[\phi_{\theta_P-\theta_{P-1}} ^{(P)}  ] \right)}\\
&:= \mu_\eps   \pi^\eps _{m_P+1,\theta_{P-1} } \big(  h^{(P-1)}  \otimes \phi_{\theta_P-\theta_{P-1}} ^{(P)} \big)  + \mu_\eps\bbE_\eps ^{\rm eq} [ h^{(P-1)} ] \bbE_\eps ^{\rm eq}[ \phi_{\theta_P-\theta_{P-1}} ^{(P)} ]  \\
& \qquad - \mu_\eps \bbE_\eps^{\rm eq}[ h^{(P-1)} ]    \pi^\eps _{m_P,\theta_{P-1}}\big( \phi_{\theta_P-\theta_{P-1}} ^{(P)}  \big)  - \mu_\eps \pi^\eps _{\theta_{P-1} }\big( h ^{(P-1)}  \big) \bbE_\eps ^{\rm eq} [ \phi_{\theta_P-\theta_{P-1}} ^{(P)} ]  \end{aligned}
\end{equation*}
which will be a new fluctuation to be analyzed at time $\theta_{P-1}$;
\item a ``contracted product"
$$     \pi^\eps _{m_P,\theta_{P-1} } \big(   \psi^{(P,P-1)} \big) \, , \quad \mbox{with} \quad   \psi^{(P,P-1)}:= \phi_{\theta_P-\theta_{P-1}} ^{(P)}(Z_{m_P})  \sum_{j = 1} ^{m_P}  h^{(P-1)}(z_j) \, , $$ 
 which will essentially decouple from the rest of the weight and sum up to give the covariance 
$$
 \bbE^{\rm eq}_\eps [ \zeta^{\eps, \rm eq}_{\theta_{P-1}} ( h^{(P-1)})   \zeta^{\eps, \rm eq}_{\theta_P} ( h^{(P)})  ] = \sum_{m_P} \bbE_\eps^{\rm eq} \Big[     \pi^\eps _{m_P,\theta_{P-1} } \big(   \psi^{(P,P-1)} \big) \Big ] + o(1) \,.
 $$ 
\end{itemize}
We  obtain 
$$ \begin{aligned}
I^{\e,\rm eq}_P = &\sum_{m_P} \bbE_\eps^{\rm eq} \left[ \Big(\prod_{p=1} ^{P-2}  \zeta^{\e,\rm eq}_{\theta_p} ( h^{(p)}) \Big)\,    \zeta^{\e,\rm eq}_{\theta_{P-1}} ( h^{(P-1)})  {\color{black}\ostar } \zeta^{\e,\rm eq}_{m_P,\theta_{P-1} } \big(\phi_{\theta_P-\theta_{P-1}} ^{(P)} \big)\right]\\
&+   \bbE^{\rm eq}_\eps [ \zeta^{\eps, \rm eq}_{\theta_{P-1}} ( h^{(P-1)})   \zeta^{\eps, \rm eq}_{\theta_P} ( h^{(P)})  ]   I^{\eps, \rm eq}_{P-2} + o(1) \, .
\end{aligned}
$$

 From the above discussion, one can guess the structure for the general term to be iterated at time $\theta_p$
\begin{equation}\label{Jepsvarsibma}
 \prod_{\{ j,j'\} \in \rho  } \bbE^{\rm eq}_\eps  \Big[  \zeta^{\eps, \rm eq}_{\theta_{j}} ( h^{(j)})   \zeta^{\eps, \rm eq}_{\theta_{j'}} ( h^{(j')})   \Big]  \times \sum_{\MM}\bbE_\eps^{\rm eq}\left[ \Big(\prod_{u<p}  \zeta^{\e,\rm eq}_{\theta_{u}} ( h^{(u)}) \Big) \,    \Otimes_{ i \in B } \zeta^{\e,\rm eq}_{ m_i, \theta_p} \big( \phi_{\theta_i-\theta_p} ^{(i)} \big) \right]
 \end{equation}
where~: 
\begin{itemize}
\item  $B$ is a subset of  $\{p,\dots P\}$,  $\rho$ is a partition of $\{p,\dots P\}\setminus B $ in pairs $\{j,j'\}$;
\item $\MM = (m_i)_{ i \in B}$ records the number of variables; $\phi^{(p)}_0 = h^{(p)}$ and the other dual functions~$\phi_{\theta_i-\theta_p} ^{(i)}$ have been pulled back from $h^{(i)}$~: each $\phi_{\theta_i-\theta_p} ^{(i)}$ is a function of $m_i$ variables containing the information on the backward transport  of $h^{(i)}$ on $[\theta_p, \theta_i]$;
\item   the {\color{black}$\ostar$-product} is defined as previously by avoiding repeated indices
\begin{equation}
\label{tensor-product-eq}
 \Otimes_{ i \in B }  \zeta^{\eps, \rm eq}_{ m_i, \theta_p} \big( \phi_{\theta_i-\theta_p} ^{(i)} \big):= \mu_\eps^{|B|/2} \sum_{A \subset B  }     \pi^\eps_{M_A,\theta_p} \Big( \bigotimes_{i \in A}\phi_{\theta_i-\theta_p} ^{(i)} \Big) \prod_{ j \in B \setminus  A} \bbE_\eps^{\rm eq}  [-\phi_{\theta_j-\theta_p} ^{(j)}]\,
 \end{equation}
 where $M_A := \sum_{i \in A} m_i$.
 \end{itemize}
As before, the fluctuation at time $\theta_p $  is first transformed into a sum of fluctuations at time~$\theta_{p-1}$, so that (\ref{Jepsvarsibma}) can be rewritten~: 
\begin{equation}
\label{duhamel-iteration}
\begin{aligned} &   \prod_{\{ j,j'\} \in \rho  }\bbE^{\rm eq}_\eps  \Big[  \zeta^{\eps, \rm eq}_{\theta_{j}} ( h^{(j)})   \zeta^{\eps, \rm eq}_{\theta_{j'}} ( h^{(j')})   \Big]   \times \sum_{\MM,\NN} \bbE_\eps ^{ \rm eq}\left[\Big( \prod_{u<p}  \zeta^{\eps, \rm eq}_{\theta_{u}} ( h^{(u)})\Big)\,   \Otimes_{ i \in B } \zeta^{\eps, \rm eq}_{m_i+n_i,\theta_{p-1} } \big(  \phi_{\theta_i-\theta_{p-1}} ^{(i )}   \big)  \right] 
 \end{aligned}
 \end{equation}
where $\NN = (n_i)_{ i \in B}$ records the number of particles added when pulling back further the test functions indexed by~$i \in B$ during the time interval~$[\theta_{p-1},\theta_p]$.
%
Then (by looking at repeated indices), either $p-1$ is added to $B$, or  an element of~$B$ is randomly chosen and paired with~$p-1$, adding a new pair in $\rho$.
%
%
Finally going back   down to~$p = 1$,  the dominant term turns out to correspond to~$B = \emptyset$ and~$\rho \in {\mathfrak S}_P^{\rm pairs}$, describing all  possible pairings of the fluctuation fields in \eqref{eq: appariement}.

\subsection{The pullback of test functions on a time~$\delta \ll 1$}\label{pullback-subsec}

 By definition, the {\it ``block"}  $\phi_{\theta_i-\theta_p} ^{(i)}$ will be obtained by pulling back  $h^{(i)}$ according to a Duhamel series on~$[\theta_p, \theta_i]$.  We shall see that a proper definition of $\phi_{\theta_i-\theta_p} ^{(i)}$ requires to describe all possible forward-in-time dynamics starting from a configuration~$Z_{ m_i}$ at time $\theta_p$, respecting suitable connection constraints. We encounter here a first difficulty due to an uncontrolled number of collisions in this forward dynamics. This issue appears already in the study of the covariance in~\cite{BGSS2} ($I^{\e,\rm eq}_2$ with the above notation), where we showed that the construction of dual functions is efficient if one performs a {\it conditioning} of  the invariant measure: this conditioning ensures that all    microscopic configurations have a  controlled dynamical behaviour on an elementary time step of size $\delta$, with~$\eps \ll \delta \ll 1$.  
 \medskip
 
 Given  an integer~$\gamma \in \N$, we call {\it microscopic cluster of size~$\gamma$} a set~$\cG $ of~$\gamma$ particle configurations in~$\T ^d\times\R^d$ such that~$(z,z') \in \cG\times\cG$ if and only if there are~$z_{1} = z,z_2,\dots,z_{\ell}= z'$ in~$\cG$ such that
\begin{equation*}
\label{eq: cut off delta}
 |x_{i}-x_{i+1} | \leq   3\sqrt{\gamma}\bbV \delta \, , \quad \forall1 \leq i \leq \ell-1\, , 
\end{equation*}where ${\mathbb V}\in \R^+$ is related to an energy truncation. To fix   ideas, we   choose from now on the microscopic  time scale~$\delta $, the intermediate time scale~$ \tau$, the macroscopic time~$\Theta$, the energy cut-off~$ {\mathbb V} $ and the size of the cluster~$\gamma$ as follows:
\begin{equation}
\label{eq: choix parametres}
\eps \ll \delta \ll \tau \ll 1 \ll \Theta = O\left( \log\log|\log\eps|\right)\;,\quad \quad \gamma = 4d\;, \quad  {\mathbb V} =|\log  \eps|\;, \, \quad \delta = \e^{1 - \frac{1}{2d}}\;.
\end{equation}
\begin{Def}\label{conditioning}
Given ~$\gamma \in \N$,
we define the set~$\Upsilon^\eps_N$ as the set of initial configurations~${\bf Z}^{\eps 0}_N $ in~$ {\mathcal D}^{\eps}_{N}$ such that for any~$p \in \{2,\dots,P\}$ and integers $k,r$ such that $1 \leq k \leq (\theta_p-\theta_{p-1}) / \tau$ and~$r \in [0, \tau/\delta]$, the configuration at time~$\theta_p - (k-1) \tau -r \delta$ satisfies
$$
\forall 1 \leq j \leq N \, , \quad |v_j |\leq \bbV \, , 
$$
and any microscopic cluster of particles is of size at most~$ \gamma$.
\end{Def}
 On each elementary step, the hard sphere system  will be shown to behave in essence as a collection of independent clusters of small size, and in particular the total number of collisions is under control. 
Note also that
with the previous choices~{\rm(\ref{eq: choix parametres})} of parameters, the conditioning by $\Upsilon^\eps_\cN $ is typical  in the sense that the complement satisfies
\begin{equation}
\label{conditioning-est}
\bbP^{\rm eq}_\eps \big(  ^c \Upsilon^\eps_\cN \big)
\leq \Theta \eps^d\;.
\end{equation}
We refer to \cite[Section 6.1]{BGSS2} for the proof of this result. 

\medskip

In the following, we shall mostly  restrict our estimates to the sets~$\Upsilon^\eps_N$, so it is useful to introduce the following notation: the probability of an event~$A$ with respect to the (unnormalized) conditioned measure is denoted by~${\mathbb P}_\eps(A)$
$$
\bbP_\eps\left( A \right) := 
\bbP^{\rm eq}_\eps\left(\Upsilon^\eps_N  \right)  \bbP^{\rm eq}_\eps\left(A \ |\ \Upsilon^\eps_\cN  \right)
:= \bbP^{\rm eq}_\eps\left(\Upsilon^\eps_N  \cap A \right)\;,
$$
and~${\mathbb E}_\eps$ is the corresponding expected value.
The  moments of the fluctuation field under such conditioning are written
$$
I^{\eps}_P := \bbE_\eps \left[ \prod_{p=1} ^P \zeta^\eps_{\theta_p} ( h^{(p)}) \right] $$
where $\zeta^\eps$ is the non-centered field defined by
\begin{equation*}
\label{eq: fluctuation field tilde}
\zeta^{\eps}_t \big(  h  \big) :=  { \sqrt{\mu_\eps }} \, 
\Big( \pi^\eps_t(h) -  {\mathbb E}_\eps\big( \pi^\eps_t(h) \big)     \Big)\;.
\end{equation*}
Furthermore we use the notation
$$ \bbE_\eps  [H_m]  :=   \bbE_\eps  [    \pi^\eps _{m,t}(   H_m)] \,,$$
and
  \begin{equation}
  \label{eq: fluctuation field m}
\zeta^{\e }_{m,t}(   H_m):=\sqrt{\mu_\eps } \Big(   \pi^\eps _{m,t}(   H_m) - \bbE _\eps [H_m]\Big)  \, .
\end{equation}

For future convenience we notice that
\begin{equation}
\label{eq: def zeta m diff}
\zeta^\e_{m,t}(   H_m) = \zeta^{\e,\rm eq}_{m,t}(   H_m) + \sqrt{\mu_\eps } \,\bbE^{\rm eq}_\eps [{\mathbf 1}_{^c \Upsilon^\eps_N}\,  \pi^\eps _{m,t}( H_m)]\;.
\end{equation}
The pairing mechanism described in Section~\ref{pairing scheme} will be achieved with~$I_P^\eps$ rather than~$I_P^{\eps,\rm eq}$, and the difference between $I_P^\eps$ and $I_P^{\eps,\rm eq}$ will be shown to be of subleading order in \eqref{eq: appariement}, thanks to the estimate \eqref{conditioning-est} (see Proposition \ref{Proposition - estimates on g0} below).

The conditioning ensures that the    pullback of the test functions can be performed  efficiently, on very small time intervals of size $\delta = O(\eps ^{1- 1/2d}) $ (see Section 3). 
Each pullback  will involve combinatorial factors, counting the number of   trajectories compatible with a given~$Z_{m_i}$, which cannot  be iterated $O(1/\delta)$ times without leading to strong divergences. The idea is therefore to iterate only the principal terms, removing   at each time step $O(\delta)$ all ``non minimal correlations", up to reaching an intermediate time scale $\tau$ such that $\delta \ll \tau \ll 1$ (see Section \ref{sec:tau}). 

\subsection{The factorization defect on a time~$\tau\gg \delta$}\label{factorization-subsec}

 The second difficulty is that the pullback does not preserve exactly the factorized structure 
$$
 \bbE_\eps \left[ \Big(\prod_{u<p}  \zeta^\eps_{\theta_{u}} ( h^{(u)}) \Big) \,   {\color{black} \Otimes_{ i \in B } }\zeta^\eps_{ m_i, \theta_p} \big( \phi_{\theta_i-\theta_p} ^{(i)} \big) \right]
  \neq \sum_{\NN}  \bbE_\eps \left[ \Big(\prod_{u<p}  \zeta^\eps_{\theta_{u}} ( h^{(u)}) \Big) \,   {\color{black} \Otimes_{ i \in B } }\zeta^\eps_{ m_i+n_i , \theta_{p-1} } \big( \phi_{\theta_i-\theta_{p-1} } ^{(i)} \big) \right]\,.
$$
Let us  start from an observable  which is a product of blocks
\begin{equation} 
\label{blocks}
\Phi_\MM (Z_\M):= \prod _{ i \in B} \phi^{(i)}\big( Z_{m_i}^{(j)}\big)\,,
\end{equation}
denoting $\MM = (m_i)_{i \in B}$ with $m_i \geq 1$. On an elementary time interval of size~$\delta$, these blocks~$ \phi^{(i)}$ are transported dynamically, which can lead to some dynamical correlations. Blocks are then connected into {\it ``packets"} according to the dynamical correlations. Then, in order to keep a fluctuation structure, all variables have to be centered. The trivial packets containing only one block have already the fluctuation structure so no additional term appears. The other packets are called {\it ``clustering"} since they contain at least two blocks, and their centering provides a contribution of the expectation.
 The goal of Section 3 is thus to establish the following identity, relating the fluctuation structure at time~$\theta$ to that at time~$\theta-\delta$ (with~$\theta$ chosen such that~$\theta$ and~$\theta-\delta $ belong to~$(\theta_{p-1},\theta_p)$):
 \begin{equation}
\label{Duhamel-identity-block}
\begin{aligned}
&\bbE_\eps   \Big[   \Big( \prod_{u=1} ^{p-1}  \zeta^\eps_{\theta_u} \big( h^{(u)}\big)\Big)  \Otimes_{ i \in B}
 \zeta^\eps_{m_i ,\theta} \big(   \phi ^{(i)} \big) \Big] \\
  &\qquad  =
\sum_{\bf \NN }\sum_{\eta_1 \cup \eta_2 \hbox{ \tiny \rm partition of }B \atop \eta_2 \hbox{ \tiny\rm  clustering} }   \prod _{q\leq |\eta_2| }\mu_\eps^{1-\frac{|\eta_{2,q}|}2} \bbE_\eps\Big[   \phi_\delta^{(\eta_{2,q})}  \Big]   \\
& \qquad  \qquad  \times 
\bbE_\eps   \Big[   \Big( \prod_{u=1} ^{p-1}  \zeta^\eps_{\theta_u} \big( h^{(u)}\big) \Big)    \Big( \Otimes_{q\leq |\eta_1| }   \mu_\eps^{\frac12 -\frac{|\eta_{1,q}|}2}  \zeta^\eps_{ M_{\eta_{1,q}}^\delta ,\theta-\delta} \Big(  \phi_\delta^{(\eta_{1,q})}  \Big) \Big)\Big] \,.
\end{aligned}
\end{equation}
where $M^\delta_{\eta_{1,q}}$ is the total number of particles in the packet $\eta_{1,q}$ at time $\theta-\delta$, and $\phi_\delta^{(\eta_{1,q})} $ is supported on configurations obtained by backward pseudo-trajectories forming a cluster $\eta_{1,q}$ (in a sense to be made precise in Section~\ref{sec:Duhamel}).

  It will be useful in the sequel to interpret~(\ref{Duhamel-identity-block}) by the following recipe 
   (see Figure~\ref{fig:fubini nested-section 2} page~\pageref{fig:fubini nested-section 2}): between time~$\theta$ and $\theta-\delta$
\begin{itemize}
\item some observables are grouped into clusters by dynamical constraints, which leads to the partition into packets~$\eta_1, \eta_2$;
\item  observables are pulled back according to the Duhamel pseudo-trajectories compatible with the previous clustering conditions;
\item some non trivial clusters (containing at least two packets)   encoded in $\eta_2$ are expelled from the fluctuation, their expectation remaining as an independent factor in the product.
\end{itemize}

 \bigskip
To iterate this procedure down to time~$\theta-\tau$ with~$\tau = R\delta \ll 1$, we need to extend formula~(\ref{Duhamel-identity-block}) starting from packets and not only from blocks. We will therefore construct iteratively on each time step~$[\theta - r \delta, \theta - (r-1) \delta]$ for~$r =1,\dots,R$ a sequence of  nested partitions
 $\eta_1^{r-1} \hookrightarrow \eta^r_1 \cup \eta^r_2 $ with $\eta^r_2$ corresponding to packets which are expelled from the main factorized structure, contributing only via their expectation, and $\eta_1^r$ corresponding to packets contributing to the factorized structure via their fluctuations  (see Figure \ref{fig:fubini nested-section 2}).   
 \begin{figure}[h] 
\includegraphics[width=5in]{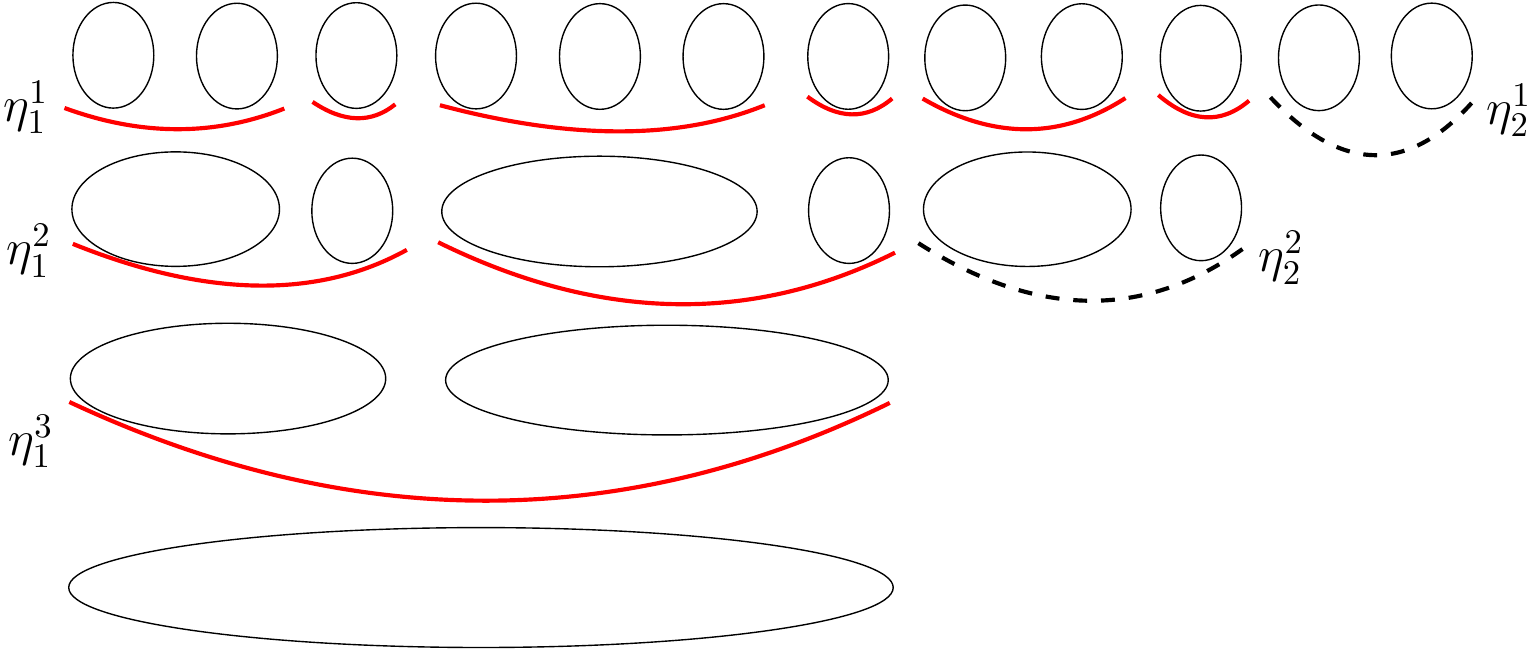}  
\caption{ At time $\theta$, the set $B$ contains 12 blocks. The nested partitions  are depicted and the dashed parts represent the expelled clustering cumulants.
After 3 iterations, blocks 8 to 12 (numbered from left to right)  have been expelled and
the other blocks  have merged into a single packet  connected by dynamical
constraints. 
}
 \label{fig:fubini nested-section 2}
\end{figure}

 \bigskip
As assumed in (\ref{duhamel-iteration}) and as becomes apparent in view of the powers of $\mu_\eps$ in (\ref{Duhamel-identity-block}), we expect that the leading order term corresponds to  the single block type factorized structure, i.e. to $  \cup _{r=1}^R \eta_1^r = \emptyset$ and $\eta_2^R$ is the trivial partition in singletons. 
However, discarding clustering terms at each time $\theta-r \delta$ ($r \leq R$) would generate diverging remainders since $1/\delta \gg 1$. We will therefore need to perform the full iteration (discarding only non minimal correlations as explained in Section~\ref{pullback-subsec}) on some intermediate time scale $ \tau\ll 1$, and then to combine all remainders due to clusterings on $[\theta- \tau, \theta]$ in  a rather subtle way (see Section  4) to recover the single block type factorized structure (see Section 5). We refer to Figure~\ref{fig:sampling-old} for a summary of the procedure.
 \begin{figure}[h] 
\includegraphics[width=5in]{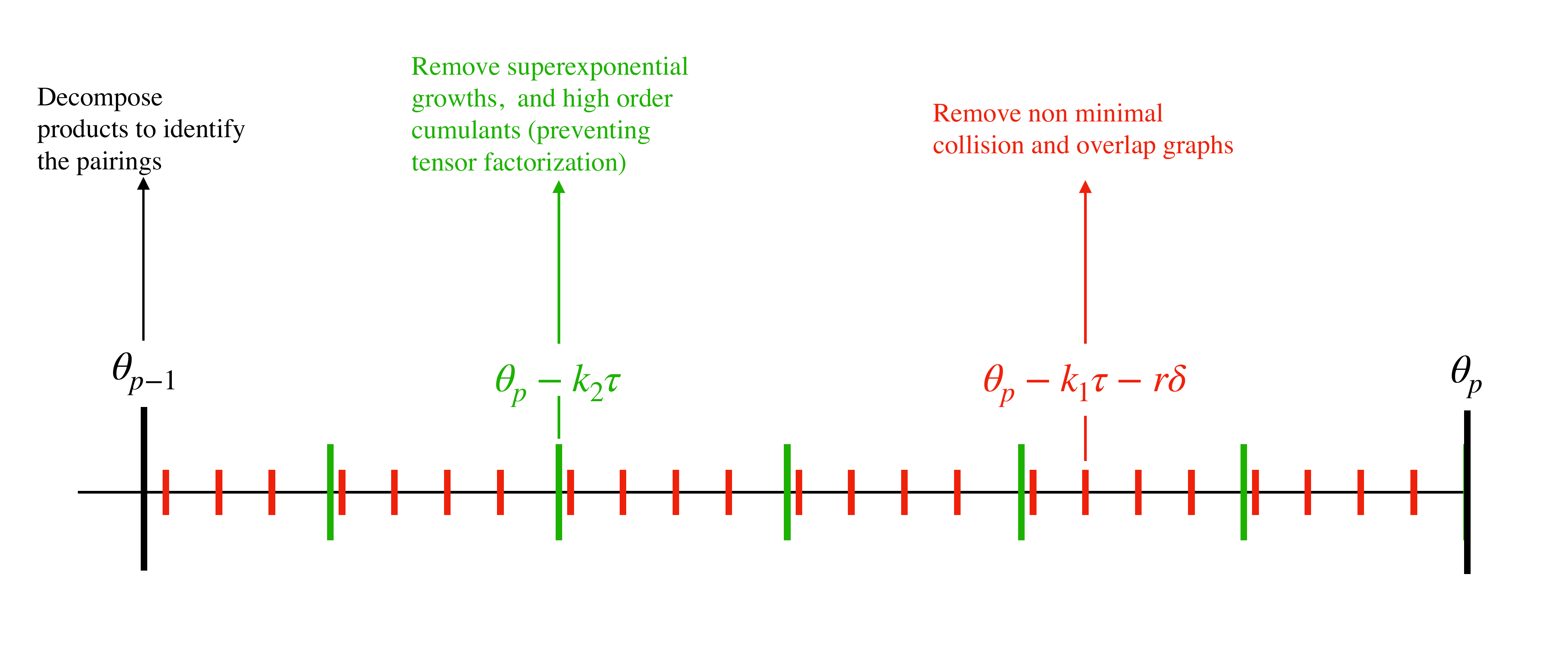}  
\caption{Double time-sampling of the   interval~$(\theta_{p-1}, \theta_{p})$ into pieces of size~$\tau$ (in green) and subpieces of size~$\delta$ (in red). At each time step of size~$\delta$, all non minimal correlations are discarded; at each time step of size~$\tau$, all other remainder terms are discarded (exponentially large collision trees, higher order clusters). Pairings are identified at each macroscopic time step~$\theta_p,\theta_{p-1}$, etc.}
 \label{fig:sampling-old}
\end{figure}

\subsection{Control of the remainder terms}

All the remainder terms coming from these two samplings (at scales $\delta$ and $\tau$) are controlled by decoupling the different times thanks to  H\"{o}lder's inequality.
This will rely on  the moment estimates on the fluctuation field
 stated in  the following proposition.
\begin{Prop}
\label{Proposition - estimates on g0}
Let~$h$ be a function in $L^\infty(\T ^d\times\R^d)$. 
Then for all~$1 \leq p < \infty$ and for~$\eps$ small enough, the moments of the fluctuation field (at equilibrium and under the conditioned measure) are bounded: 
\begin{eqnarray}
&&  \left|\bbE^{\rm eq}_\eps \Big( \big(\zeta^{\eps,\rm eq} (h)\big) ^p \Big)\right| \leq C_p   \| h \|_\infty^p\,, \label{eq: moment ordre 2,4 eq}\\
&&\left| \bbE_\eps \Big( \big(\zeta^\eps (h)\big) ^p \Big)\right| \leq C_p   \| h \|_\infty^p\,, \label{eq: moment ordre 2,4}
\end{eqnarray}
where the constant $C_p>0$ depends only on~$p$. 
Moreover under  the assumptions of Proposition~{\rm\ref{eq: appariement}}
\begin{equation} \label{eq:IPeq-IP}
\left| I_P^{\eps,\rm eq} - I_P^{\eps} \right| \leq \Big(
\prod_{p = 1} ^P \| h^{(p)}\|_{L^\infty} \Big)C_P \left(\Theta\e^{d}\right)^{1/2}
\end{equation}
uniformly in~$\theta_1, \dots \theta_p \in [0, \Theta]$.
\end{Prop}
The standard result  \eqref{eq: moment ordre 2,4 eq} can be found in Proposition A.1 from \cite{BGSS2}, from which~\eqref{eq: moment ordre 2,4} and~\eqref{eq:IPeq-IP} will be derived in Section \ref{variance-sec}.

The key argument  to implement the strategy is therefore to obtain
  estimates for the expectation and variance of $\otimes$ products defined by   (\ref{tensor-product-eq}),  proved in Sections~\ref{geometric-sec}
and~\ref{variance-sec}. To derive these estimates it is necessary to have a   precise description of the structure of the test functions~$\phi^{(i)}_{\theta_i-\theta_p}(Z_{m_i})$. As
   will be made precise in Section \ref{sec:tau}, 
  they  are supported on ``dynamical clusters''  of $m_i$  particles, called {\it forward clusters} below. This means that   there exists a graph with~$m_i$ vertices, constructed by adding one edge each time two particles find themselves at a distance (equal or) less than~$\eps$ during the time interval~$[\theta_{p}, \theta_i]$.  
     \begin{figure}[h] 
\centering
\includegraphics[width=2in]{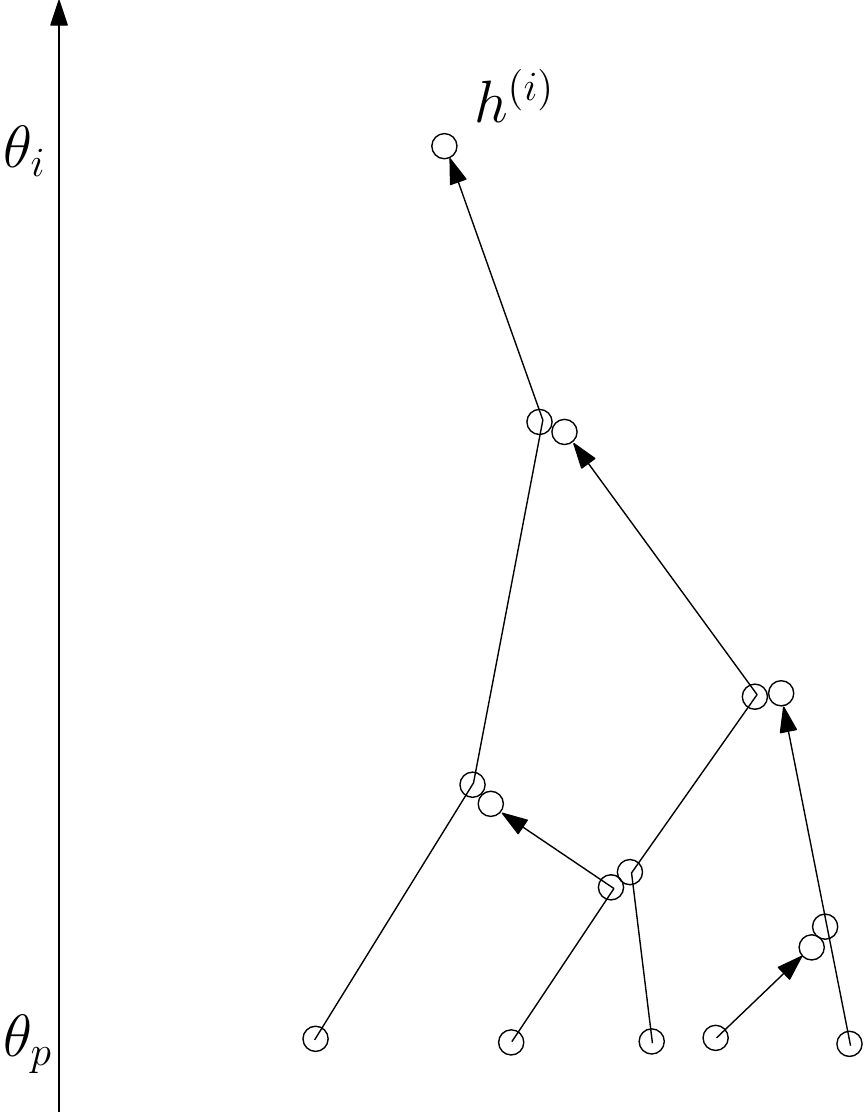} 
\caption{Forward cluster of $\phi_{\theta_i-\theta_{p}} ^{(i)}$ associated with the pullback of~$h^{(i)}$ during a time $\theta_i - \theta_{p}$, case of $m_i = 5$ hard spheres. The function $\phi^{(i)}$  is supported on a configuration $Z_{m_i}$ such that, starting at time $\theta_p$ and going forward in time, particles encounter, and are progressively removed from the dynamics, until only one particle is left at time $\theta_i$.
}
\label{fig:FC}
\end{figure}

   In particular 
   they will be shown to satisfy roughly an estimate of type
$$
\left|\phi^{(i)}_{\theta_i-\theta_p}\left(Z_{m_i}\right) \right| 
\leq  C {(C\mu_\eps)^{m_i- 1} \over m_i!}\sum \indc_{ Z_{m_i} \,\hbox{\tiny \rm forward cluster}}\;,$$
where we sum over all possible forward dynamics starting from $Z_{m_i}$ as in Figure \ref{fig:FC}.
 Note that the size of the typical volume spanned by one particle in a finite time is 
$\mu_\e^{-1}$, so that  the volume of a cluster is typically $\mu_\eps^{-m_i + 1}$. The $1/m_i!$ is due to the symmetrisation  by  permutation of the particle labels.
This estimate in turn   will imply that
$$
  \bbE _\eps  \left[   \left(\Otimes _{i \in B}     \zeta^{\eps} _{m_i} \big( \phi^{(i)} \big) \right)^2\right]    \leq  (C'\Theta)^{2M}   \, , \quad M := \sum_{i \in B} m_i\, . 
$$
Moreover using the fact that error terms are supported on clusters with additional constraints, we will get some additional smallness, providing the expected control on  the error terms at arbitrary times.

\section{Preserving the fluctuation structure on small times}
\label{sec:Duhamel}
 
  In this section we  detail one part of the discussion of the previous section, namely   how to transport the fluctuation structure between two time steps, thanks to  the Duhamel formula.   However in order to make sense of the Duhamel formula and its dual form uniformly in~$\varepsilon$, we will actually not connect directly time $\theta_p$ to time $\theta_{p-1}$, but rather introduce an iteration 
   on   infinitesimal time intervals (much smaller than Lanford's convergence time). This consists in   transforming a weight at a  time $\theta \in(\theta_{p-1},\theta_p)$ in a (more complicated) weight at a time~$\theta-\delta $ in~$(\theta_{p-1},\theta_p)$, for some very small~$\delta>0$ tuned in~(\ref{eq: choix parametres}). In what follows, we shall therefore focus only on the interval~$(\theta-\delta, \theta)$: the precise statement requires some notation and is given at the end of this section (see Proposition~\ref{prop-from theta to theta + delta} page~\pageref{prop-from theta to theta + delta}).

\medskip

  The procedure  relies   on three ingredients: we fix~$\theta= \theta_p - r\delta $ for some $r \in [0, (\theta_p-\theta_{p-1})/\delta]$.
 \begin{itemize}
 \item We first introduce  the family of {\it correlation functions} $\left( G^{\eps}_M\right)_{M \geq 1}$ at time 
 $t\in [\theta-\delta,\theta ]$, defined for any test function $H_M$  of $M$ variables by
\begin{equation}
\label{corr-function}
\begin{aligned}
 \int G^{\eps}_M (t,Z_M)H_M (Z_M) dZ_M := \bbE_\eps  \Big[\Big( \prod_{u=1} ^{p-1}  \zeta^\eps_{\theta_u} \big( h^{(u)}\big) \Big)\,    \pi^\eps _ {M,t} (H_M ) \Big]\, .
\end{aligned}
\end{equation}
These correlation functions satisfy a hierarchy of linear evolution equations, the so-called BBGKY hierarchy, so that $G^\eps_M (\theta)$ can be expressed   as a Duhamel sum involving the  correlation functions at time $\theta-\delta $ (see Section \ref{duhamel-subsec})
\begin{equation}
\label{eq:seriesexp}
 G^\eps_M(\theta ) = \sum_{N\geq 0} Q^\eps_{M, N} (\delta) G^\eps_{M+N} (\theta-\delta )\,,
 \end{equation}
 where the operator $Q^\eps_{M, N}$ encodes transport and collisions.
 
 \item We then  use a graphical representation of the elementary operator
 $Q^\eps_{M, N} (\delta)$ in terms of ``pseudo-trajectories'' to define an ``adjoint'' operator (see Section \ref{duality-subsec}). 
 
    \item We finally recombine the  contributions  of the different correlation functions to identify a fluctuation structure at time $\theta-\delta$ (see Section \ref{factorization-sec}). 
\end{itemize}
 Note that the last time interval   when $r = \lfloor \frac{\theta_p-\theta_{p-1}}{\delta} \rfloor$
 may be a little smaller than $\delta$, but the very same arguments can be applied. 

\subsection{The Duhamel iteration and its graphical representation} 
\label{duhamel-subsec}
 
In the grand canonical setting, \eqref{corr-function}  is equivalent to
  \begin{align*}
\label{eq: densities at t}
G^{\eps}_M (t,Z_M)
 = \frac1{\mu_\eps^M } \,\sum_{n=0}^{\infty} \,\frac{1}{n!} \,\int_{ (\T ^d\times\R^d)^n} dz_{M+1}\dots dz_{M+n} \,
W_{M+n}^\eps (t, Z_{M+n}) \, , \qquad t \in [0,\Theta] \, ,
\end{align*}
\medskip
where the (signed, non-normalized) measure $\left(W^\eps_N(t)\right)_{N \geq 1}$ is defined as follows. 

\noindent
On $[\theta_1, \theta_2]$,~$W^\eps_N$ solves the Liouville equation
\begin{equation}
\label{Liouville}
	\d_t W^{\eps}_N +V_N \cdot \nabla_{X_N} W^{\eps}_N =0  \,\,\,\,\,\,\,\,\, \hbox{on } \,\,\,{\mathcal D}^{\eps}_{N}\, ,
	\end{equation}
	with specular reflection on the boundary 
(and extending $W^\eps_N$ by zero outside~$\cD_N^\eps$) and with initial data (see \eqref{eq: initial measure})
$$W^\eps_N(\theta_1,Z_N) := {\mathbf 1}_{\Upsilon^\eps_N} 
  W^{\eps, {\rm eq} }_{N} (Z_N)  \;
 \frac1{\sqrt{\mu_\eps}}\left(\sum_{i=1}^N h^{(1)}(z_i) -   \mu_\e \bbE_\eps \big[ h^{(1)}\big]\right) \, .$$
 Inductively for $p>2$, one solves again Eq.\,\eqref{Liouville} on $[\theta_{p-1}, \theta_p]$,
with perturbed initial data 
 $${ W}^\eps_N(\theta_{p-1}^+,Z_N) :=
W^\eps_N(\theta_{p-1}^-,Z_N) \frac1{\sqrt{\mu_\eps}}  \left(\sum_{i=1}^N h^{(p-1)}(z_i) -  \mu_\e\bbE_\eps   \big[ h^{(p-1)}\big]\right)\, ,$$
 where $\pm$ indicate the limits from the future/past.

 \subsubsection{The Duhamel iteration.} 
   By integration of the Liouville equation~(\ref{Liouville}) for fixed $\eps$, one obtains formally that  for any integer~$M$, the $M$-particle correlation function~$G^{\eps}_M  $ satisfies
\begin{equation}
\label{BBGKYGC}
 \partial_t G^\eps_M + V_M \cdot \nabla_{X_M} G^\eps_M = C_{M,M+1}^\eps G^\eps_{M+1} \quad \mbox{on} \quad {\mathcal D}^\eps_{M  }\;,
\end{equation}
 with specular boundary reflection as in \eqref{Liouville}. This is the well-known  BBGKY hierarchy (see~\cite{Ce72}), which is the  elementary brick in the proof of   Lanford's theorem for short times~\cite{La75}. The operator~$C^\eps_{M, M+1}$ describes  the collision  between one ``fresh'' particle (labelled $M+1$) 
 and one given particle~$i\in \{1,\dots, M\}$:$$ C_{M,M+1} ^\eps G^\eps_{M+1}   := \sum_{i=1}^M C_{M,M+1} ^{\eps,i} G^\eps_{M+1}  $$
with
 $$
 \begin{aligned}
 (C_{M,M+1} ^{\eps,i} G^\eps_{M+1} )(Z_M)&:= 
  \int_{{\mathbb S}^{d-1} \times \R^d}   G^\eps_{M+1} (Z_M^{\langle i \rangle}, x_i,v_i',x_i+\e \omega ,u ') \big( (u - v_i)\cdot \omega  \big)_+ \, d \omega\, du\\
&\quad -  \int_{  {\mathbb S}^{d-1}\times \R^d}   G^\eps_{{M+1} } (Z_M, x_i+\e \omega,u) \big( (
u  - v_i
)\cdot \omega  \big)_- \, d \omega \,du\\
&= 
  \int_{{\mathbb S}^{d-1} \times \R^d}   G^\eps_{M+1} (Z_M^{\langle i \rangle}, x_i,v_i',x_i+\e \omega ,u ') \big( (u - v_i)\cdot \omega  \big)_+ \, d \omega\, du\\
&\quad -  \int_{  {\mathbb S}^{d-1}\times \R^d}   G^\eps_{{M+1} } (Z_M, x_i-\e \omega,u) \big( (
u  - v_i
)\cdot \omega  \big)_+ \, d \omega \,du \, ,
 \end{aligned}
 $$
 where~$(v'_i,u ')$ is recovered from~$(v_i,u )$ through the scattering law as in \eqref{scattlaw}, and with the notation
$$Z_M^{\langle i \rangle} := (z_1,\dots,z_{i-1},z_{i+1},\dots,z_M )\,.
$$ 
Now let us fix $\theta= \theta_p - r\delta $ for some $r \in [0, (\theta_p-\theta_{p-1})/\delta]$.
Denote by~$S^\eps_M$\label{Sn-def}  the group associated with  transport in $\cD^\eps_M$, with specular reflection on the boundary. 
By iteration of  Duhamel's formula
$$
G^\eps _M  (\theta ) = S^\eps_M(\delta) G_{M }^{\eps }(\theta-\delta )+ \int_{\theta-\delta }^{\theta   }S^\eps_M(\theta -t_1)  C_{M,M+1} ^\eps G^\eps_{M+1}(t_1)\, dt_1\,,
$$
  the solution~$G^\eps _M $ of the hierarchy~(\ref{BBGKYGC}) can formally be expressed as a sum of operators acting on the   data at time~$\theta-\delta $:
\begin{equation}\label{eq:seriesexpbis}
G^\eps _M   (\theta) =\sum_{N\geq0}    Q^\eps_{M, N}( \delta )  G_{M+N }^{\eps }(\theta-\delta )\, ,
\end{equation}
where we have defined  
$$\begin{aligned}
Q^\eps_{M, N }( \delta)   G_{M +N}^{\eps }(\theta -\delta  )&:= 
\int_{\theta -\delta}^{\theta } \int_{\theta -\delta}^{t_{1}}\dots  \int_{\theta-\delta }^{t_{N-1}}  S^\eps_M( \theta  -t_{ 1}) C^\eps_{M,M+1}  S^\eps_{M+1}(t_{1}-t_{2})    \\
&\qquad \qquad \qquad \dots  S^\eps_{M+N}(t_{N})    G_{M+N }^{\eps }(\theta-\delta ) \: dt_{N} \dots d t_{1} \, .
\end{aligned}
$$ 
We stress that formula \eqref{eq:seriesexpbis} is valid almost everywhere for a large class of measures~$\left(W^\eps_N(t)\right)_{N \geq 1}$, in spite of a collision operator being defined as a trace (see for instance \cite{GSRT, S}).

 \subsubsection{Backward pseudo-trajectories.} 
It is a standard procedure to translate the iterated Duhamel formula~(\ref{eq:seriesexpbis}) in terms of (backward) {\it pseudo-trajectories}. We first encode the combinatorics of collisions in  a graph $a = (a_j)_{1\leq j \leq N} $ where $a_j\in  \{1,\dots,M+j-1\}$ denotes the label of the particle colliding with particle $M+j$ at its creation time. Note that the set of graphs $a$ is a collection $\cA_{M,N}$ of $M$ binary trees with a total of $N$ branchings. We define~$\cA_{M,N}^\pm$ the set of such {\it collision trees}, where each 
$a_j$ is equipped with a sign $s_j \in \{-1,1\}$.
Then, given such an $a$, as well as a configuration $Z_M$ and collision parameters $(t_j, \omega_j, u_j)_{1\leq j \leq N}$ with~$t_{N+1} = \theta -\delta < t_N<\dots <t_1 <\theta =t_0$, we define iteratively the pseudo-trajectory 
$$\Psi^{\eps}_{M,N} = \Psi^{\eps}_{M,N} \Big(Z_{M}, a, ( t_{j}, \omega_{j}, u_j)_{j= 1,\dots, N}\Big)$$
 as follows
 (denoting by $Z^\e_{M+i}(t )$ the coordinates of  the  pseudo-particles at time~$ t \in ] t_{i+1},t_{i}]$):
\begin{itemize} \label{listPT}
\item starting from $Z_M$ at time $ \theta $,
\item transporting all existing particles backward on $(t_{j}, t_{j-1})$  (on ${\mathcal D}^\eps_{M +j-1}$ with specular reflection at collisions),
\item adding a new particle labeled~$M+j$ at time $t_{j} $  at position~$x^\e_{a_{j}}(t_{j}) +\eps s_{j} \omega_{j}$, and with velocity~$u_j$,
\item   applying the scattering rule  if $s_{j} >0$.
\end{itemize}

We discard non admissible parameters for which the above procedure is ill-defined; in particular we exclude values of $\omega_{j}$ corresponding to an overlap of particles  (two particles at mutual distance less than $\e$) as well as those such that
 $\omega_j \cdot \big( u_j-v^\e_{ a_{j}} (t_{j}^+)\big) \leq0$.
In the following we denote by~$\cG^\eps_N(a,Z_M )$ the set of admissible parameters   $(t_j, \omega_j, u_j)_{1\leq j \leq N}$
and  by~$Z^\e_{M+N}(\theta -\delta)$ the configuration at time~$\theta -\delta$.
With these notations, one gets the following geometric representation : 
\begin{equation*}
\begin{aligned}
 G^\eps _M (\theta, Z_M) &=\sum_{N \geq 0}
\sum_{a \in \cA^\pm_{M,N} }
\int_{\cG^\eps_N(a,Z_M )}    dT_{N}  d\Omega_{N}  dU_N\\
&\qquad \times \left(\prod_{j=  1}^{  N}  s_{j}\Big(\big( u_j -v^\e_{ a_{j}} (t_{j}^+)\big) \cdot \omega_{j} \Big)_+\right)  
G_{M+N}^{\eps  } \big (\theta -\delta, Z^\e_{M+N}(\theta -\delta)\big)\, , 
\end{aligned}
\end{equation*}
where $(T_{N}, \Omega_{N}, U_N) := (t_{j}, \omega_{j}, u_j)_{  1\leq j\leq   N}$.

\medskip

We recall the following classical notions of {\it collision} and {\it recollision} in a pseudo-trajectory.
  \begin{Def}
\label{def: clustering}
A \emph{collision} is the addition of a fresh particle at distance~$\eps$  from an existing particle (see the third item above), while a  \emph{recollision} involves two  particles transported by the backward flow~$S^\e$
(in between two collision times).
\end{Def}

\subsubsection{Blocks and packets.} 

As mentioned in Section \ref{factorization-subsec}, when pulling back the product of fluctuation fields,   various structures are involved.  We refer to Figure~\ref{fig:blocks and packets} for a schematic description of the following definition.

\begin{Def}\label{def:blocks and packets}
Given a particle labeled~$j\in \{1, \dots, P\}$ and the number~$m_j $ of particles in the collision tree of~$j$ at  some time, the associate
  \emph{block}  of particles  is  the set of all  particles  in the collision tree  at that time.  We   denote by $Z_{m_j}^{(j)}$ the corresponding configuration.

A \emph{packet}  of particles  is a union of different blocks which have been connected dynamically  at time~$t$. These packets   aggregate as the iteration described in Section~{\rm\ref{factorization-subsec}} progresses.  Given $B\subset \{1, \dots, P\},$  a partition $\varsigma$ of $B$  and $i \leq |\varsigma|$,  we     call~$C_i$ the   {\it packet} of~$M_{\varsigma_i} :=\displaystyle \sum_{j \in \varsigma_i} m_j$ particles with  configuration~$( Z_{m_j}^{(j)})_{j \in \varsigma_i} $, connecting  all the blocks~$j $ in~$ \varsigma_i$.
\end{Def} 

   \begin{figure}[h] 
\centering
\includegraphics[width=4in]{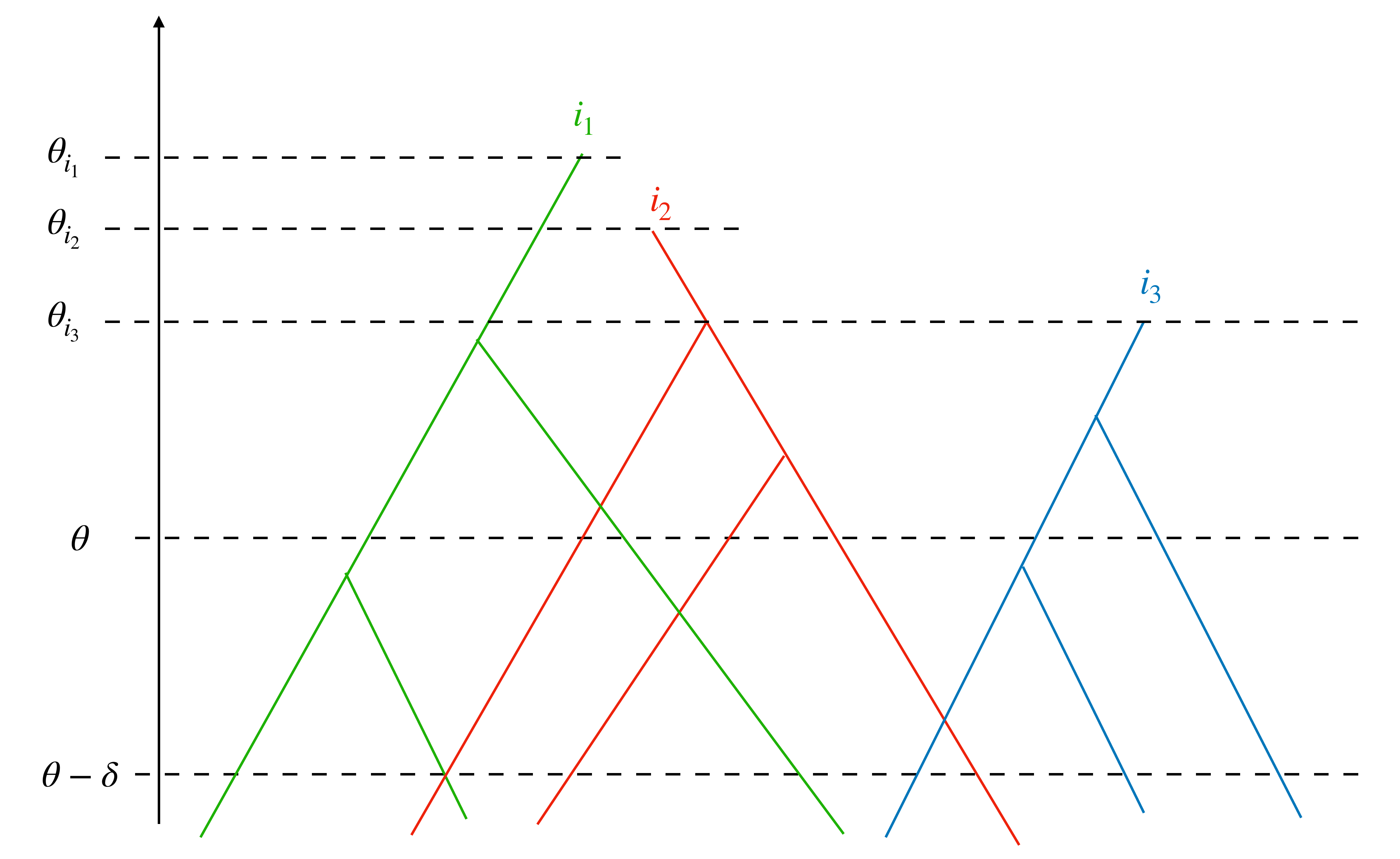} 
\caption{Three blocks indexed by~$i_1,i_2,i_3$ consist   of~$m_{i_1} = 2$, ~$m_{i_2} = 3$ and~$m_{i_3} = 2$ particles   at time~$\theta$, at which time blocks~$i_1$ and~$i_2$ have merged into one packet. The remaining block merges with that packet to build one packet at time~$\theta-\delta$.}
\label{fig:blocks and packets}
\end{figure}

For the elementary step of the iteration, we then have to consider functions of the form
\begin{equation} 
\label{eq:defPhi}
\Phi_\MM (Z_\M):= \prod _{ i \leq  |\varsigma| } \phi^{(\varsigma_i)}\big(\{ Z_{m_j}^{(j)}\} _{j \in \varsigma_i} \big) 
\end{equation}
where the  functions~$\phi^{(\varsigma_i)}$ will be assumed  symmetric within each block $j\in \varsigma_i$.  We have written~$\MM:= \left(m_j\right)_{j\in B} $.

  For a given $a \in \cA^\pm_{M,N}$ with $M = \sum_{j \in B} m_j$,
   we   denote by $n_j$ the number of branchings in the collision tree issued from the block $j$  on~$(\theta-\delta,\theta)$. Note that the   cardinal of the blocks and the packets vary with time. 
  We   also denote~$\NN:= \left(n_j\right)_{j\in B} $ so that  $\sum_{j\in B}  n_j = N$. 
Finally we   denote by~$a^{(j)}\in  \cA^\pm_{m_j,n_j}$ the restriction of~$a$ to the block $j$. Then starting from~(\ref{eq:defPhi}) we can define
  \begin{equation*}
\begin{aligned}
 G^\eps _M (\theta, Z_\M) &=\sum_{\NN}
 \sum_{a \in \cA^\pm_{\MM,\NN}} 
 \int_{\cG^\eps_\NN(a,Z_\M )}    dT_{\NN}  d\Omega_{\NN}  dU_\NN\\
&\qquad \times \left(\prod_{j \in B}\prod_{\ell=  1}^{ n_j}  s_\ell^{(j)}\Big(\big( u_\ell^{(j)} -v^\e_{ a_\ell^{(j)}} (t_\ell^{(j)+})\big) \cdot \omega_\ell^{(j)} \Big)_+\right)  
G_{M+N}^{\eps  } \big (\theta -\delta, Z^\e_{\M+\NN}(\theta -\delta)\big)\, , 
\end{aligned}
\end{equation*}
where  $(T_{\NN}, \Omega_{\NN}, U_\NN)
 = \left((t_\ell^{(j)}, \omega_\ell^{(j)}, u_\ell^{(j)})_{1\leq \ell \leq n_j}\right)_{j \in B}$ and each of the~$|B|$ sets~$(t_\ell^{(j)})_{1 \leq \ell \leq n_j}$ is ordered;  recalling~\eqref{corr-function} and \eqref{eq:seriesexpbis} and using the symmetry of correlation functions, one has 
\begin{equation}
\label{Duhamelformulabeforeduality} 
\begin{aligned}
\bbE_\eps  & \Big[ \Big( \prod_{u=1} ^{p-1}  \zeta^\eps_{\theta_u} \big( h^{(u)}\big) \Big)    \pi^\eps _ {M, \theta }( \Phi_\MM ) \Big] \\
&  = \int G^{\eps}_M (\theta ,Z_\MM)\Phi_\MM (Z_\MM) dZ_\MM 
  =
 \sum _{\NN } 
 \int  \Big( Q^{\eps} _{\M,  \NN}( \delta )    G^{\eps }_{M+N}(\theta-\delta) \Big)(Z_\MM)\; \Phi_\MM (Z_\MM) dZ_\MM \, .
\end{aligned}
\end{equation} 
\begin{Rmk}
By the Fubini identity, one may equivalently    prescribe an  order on all the collision times~$(t_j)_{1 \leq j \leq N}$ corresponding to   the trees~$a \in \cA^\pm_{M,N}$, or a partial order on all the times~$(t_\ell^{(j)})_{1 \leq \ell \leq n_j}$ for each~$j \in B$, corresponding to   the trees~$a \in \cA^\pm_{\M,\NN}$.
\end{Rmk}

\subsection{Pullback of observables}
\label{duality-subsec}
The idea now is to take advantage of  the geometric representation to construct the ``adjoint" of the operator $Q^\eps_{\M, \NN }( \delta ) $, rewriting~(\ref{Duhamelformulabeforeduality}) as
 $$
\begin{aligned}
\bbE_\eps   \Big[  \Big( \prod_{u=1} ^{p-1}  \zeta^\eps_{\theta_u} \big( h^{(u)}\big) \Big)    \pi^\eps _ {M, \theta }( \Phi_\MM ) \Big] 
&  ``= "\sum _{\NN } 
 \int     G^{\eps }_{M+N}(\theta-\delta) \Big( Q^{\eps * } _{\M,  \NN}( \delta )  \Phi_\MM \Big) dZ_{\MM+\NN} \\
 & ``= "\sum_\NN \bbE_\eps  \Big[  \Big( \prod_{u=1} ^{p-1}  \zeta^\eps_{\theta_u} \big( h^{(u)}\big) \Big)    \pi^\eps _ {M+N, \theta-\delta  }\Big(Q^{\eps * } _{\M,  \NN}( \delta )  \Phi_\MM \Big) \Big]  \, .
\end{aligned}
$$ 
In other words, this means that we would like to change variables
\begin{equation}
\label{change-variables}
 (Z_\M, T_{\NN}, \Omega_{\NN}, U_\NN) \mapsto Z^\e_{\M+\NN}(\theta-\delta)
 \end{equation}
 where $(T_{\NN}, \Omega_{\NN}, U_\NN)
 = \left((t_\ell^{(j)}, \omega_\ell^{(j)}, u_\ell^{(j)})_{1\leq \ell \leq n_j}\right)_{j \in B}$.
Unfortunately it is not true that the change of variables (\ref{change-variables}) is admissible in general, due to the presence of recollisions. Actually only recollisions involving particles of the same block are an issue, which leads to the following classification~:

\begin{Def}
 A recollision is said to be  \emph{internal}
 if it involves two particles of the same block. It is said  \emph{external} if it involves two particles of different blocks.
 \end{Def} 

\subsubsection{Duality in absence of internal recollisions.}

Let us first describe the duality argument in the case when  there is no internal recollision, which is simpler.  Setting~$\MM:= \left(m_j\right)_{j\in B}$   the number of particles in the block~$j$ at time~$\theta$ and~$\NN:= \left(n_j\right)_{j\in B} $ the number of particles added to the block~$j$ between times~$\theta$ and~$\theta-\delta$,  we denote by $Q^{\eps0}_{\M, \NN}$   the restriction of~$Q^\eps_{\M, \NN}$ to pseudo-trajectories  without internal recollisions. We set, recalling~(\ref{eq:defPhi}) and~(\ref{Duhamelformulabeforeduality}),
\begin{equation*}
I^0_{\M, \NN}  :=  
 \int \Big( Q^{\eps 0} _{\M,  \NN}( \delta )    G^{\eps }_{M+N}(\theta-\delta) \Big) (Z_\MM)\; \Phi_\MM (Z_\M) dZ_\MM \, .
 \end{equation*}
Given~$a \in \cA^\pm_{\M,\NN }$, consider the change of variables~:
\begin{equation}
\label{change of variables}
 (Z_\M, T_{\NN}, \Omega_{\NN}, U_\NN) \longmapsto Z^\e_{\M+\NN}(\theta-\delta ) \in \, \cR^0_{a} \,,
\end{equation}
where the configurations in $\cR^0_{a} $ have to be  compatible with 
 pseudo-trajectories satisfying the following constraints on $(\theta-\delta, \theta)$: 
 \begin{itemize}
 \item[(i)] there are $n_j$ particles added to the block $j$;
 \item[(ii)]  the addition of new particles in the block $j$  is prescribed by  $a^{(j)}$;
 \item[(iii)] the pseudo-trajectory has no internal recollision. 
\end{itemize}
Note that  pseudo-trajectories compatible with~$\cR^0_{a} $ may have external recollisions (recall Definition~\ref{def:blocks and packets}).
  This change of variables is injective since  the forward flow
underlying \eqref{change of variables} starting from~$Z_{\MM+\NN}\in \cR_a ^0$ at $\theta-\delta$  can be defined in a unique way on $[\theta-\delta, \theta]$ as explained below.
 \begin{Def}\label{def:encounter}
We   say that  two particles \emph{encounter}  when they  approach at a distance~$\e$ at some time, in the forward flow. 
\end{Def} 
If two particles encounter at time~$t^-$, their resulting configuration at   time~$t^+$   is  then obtained as follows:
\begin{itemize}
\item[(a)] if they belong to two different blocks, then their velocities are deflected according to the scattering law (\ref{defZ'nij}); \label{pointa}
\item[(b)]   if they are in the same block $j$, the collision  is prescribed by the tree~$a^{(j)}$, and one of the two particles disappears (it is removed) from the flow. 
 The velocity of the remaining particle is updated by scattering, or not,   according to the parameters $\SS_j = ( s_\ell^{(j)} )_{\ell \leq n_j}$ encoded also by the tree $a^{(j)}$.
\end{itemize}
One can prove recursively that the  jacobian  of the inverse map (\ref{change of variables})  is
$$
 \frac1{\mu_\eps ^{N}} \prod_{j\in B } \prod_{\ell = 1}^{n_j} \Big(\big(
u_{\ell}^{(j)} -v^\e_{a_\ell^{(j)} }(t_\ell^{(j)+} )\big)\cdot \omega_\ell^{(j)}
\Big)_+ \, .
$$

\medskip

Denoting by~$Z_\M^\e(\theta ,Z_{\M+\NN })$ the configuration of   the $M$ particles at time~$\theta $ starting from~$Z_{\M+\NN } $ in~$ \cR_{a} ^0$ at time~$\theta-\delta$, one  can therefore write  
\begin{equation} \begin{aligned}
\label{eq : I0 M N}
I^0_{\M, \NN}  
& =  \sum_{a \in \cA^\pm_{\MM,\NN}}  { \mu_\eps^{N }   }  \int_{\cR^0_{a} } dZ_{\M+\NN}
  G^{\eps}_{M+N} ( \theta-\delta,Z_{M+N })\Phi_\MM \big(Z_\M^\e(\theta ,Z_{\M+\NN })\big )   \prod _{j\in B}\sgn(\SS_j) \, ,
  \end{aligned}
  \end{equation} 
  where the factor~$ \mu_\eps^{N }$ comes from the jacobian, and~$\sgn(\SS_j)$ is the product of all scatterings signs attached to the block $j$
  $$ \sgn(\SS_j) = \prod_{\ell = 1}^{n_j} s_\ell^{(j)}\,.$$

Referring to (\ref{eq:defPhi}), since $\Phi_\MM$ is symmetric within blocks, we get a complete symmetrization within blocks  by performing a partial symmetrization on the added particles~: we will denote by ${\mathfrak S}^{n_j} _{m_j+n_j}$ the partitions of~$n_j+m_j$ particles in $n_j$ (ordered) singletons and a block of size~$m_j$. Then we can set
$$
 \Phi^{0 }_{\M,\NN }(Z_{\M+\NN}):=\mu_\eps^{N  } \sum_{\big( \sigma_j \in {\mathfrak S}^{n_j} _{n_j+m_j}\big)_{j\in B}}  
 \sum_{
a \in {\mathcal A}^\pm_{\M, \NN}
} 
\Phi_\MM \big(Z_\M^\e(\theta , Z_\sigma)\big ) \indc_{\{ Z_{\sigma} \in \cR_{a}  ^0\}    } \prod _{j\in B }\left( \frac  { m_j!  }{  (m_j+n_j) !} \sgn(\SS_j)\right) \, ,
$$
where we have denoted for simplicity~$Z_\M^\e(\theta , Z_\sigma)$ for the~$M$ remaining particles at time~$\theta$, and 
as a consequence \eqref{eq : I0 M N} can be rewritten  by duality, recalling~(\ref{corr-function}), 
\begin{align}\label{eq : Irec M N 0 dual}
I^0_{\M, \NN}  
=     
\; \bbE_\eps   \Big[  \Big( \prod_{u=1} ^{p-1}  \zeta^\eps_{\theta_u} \big( h^{(u)}\big) \Big)  
  \pi^\eps_{M + N,\theta-\delta} \big( \Phi^{0 }_{\M,\NN } \big) \Big] \, .
  \end{align}
Note that  on  each configuration $Z_{\M+\NN}$, there exist at most $4^{N}$  different $(\sigma, a)$ such that~$Z_{ \sigma }$ belongs to~$ \cR_{a}^0 $. Indeed at each encounter in the forward dynamics (recall Definition~\ref{def:encounter}), 
the particle which  disappears has to be chosen, as well 
as a possible   scattering. To fix these discrepancies, we introduce for each index~$j\in B $ two sets of signs~$\bar\SS_j:=(\bar s_\ell^{(j)} )_{1 \leq \ell \leq n_j}$ and~$\SS_j:=(s_\ell^{(j)})_{1 \leq \ell \leq n_j} $ which determine respectively which particle should be annihilated (say~$\bar s _\ell^{(j)} = +$ if the particle with largest index remains,~$\bar s _\ell^{(j)} = -$ if it disappears) and whether there is scattering ($s _\ell^{(j)}  = +$) or not ($s _\ell^{(j)} = - $).
Note that
  the signs~$(s_\ell^{(j)} )_{1 \leq j \leq n_i}$ are encoded  in
 the collision tree $a^{(j)}$ while~$(\bar s_\ell^{(j)} )_{1 \leq \ell \leq n_j}$  are known if~$\sigma_j$ is given. 
We stress the fact that if two particles in different blocks  encounter,  there is no ambiguity on the dynamics: it corresponds to  a recollision in the backward pseudo-trajectory hence there is always scattering  (see Case~(a) page~\pageref{pointa}).
  If we prescribe the sets~$({\mathbf S}, \bar \SS)   :=(  \SS_j,\bar\SS_j ) _{j\in B  }$, then the mapping
\begin{equation}\label{mapping'}(a, (\sigma_j)_{j\in B}, Z_\M, T_{\NN}, \Omega_{\NN}, U_\NN) \longmapsto Z^\e_{\M+\NN}(\theta-\delta )
\end{equation}
restricted to pseudo-trajectories compatible with $({\mathbf S}, \bar \SS) $, is  injective.
This leads to defining
\begin{equation}
\label{defPhi0MNS}
 \Phi^{0 }_{{\mathbf \Xi}^0}(Z_{\M+\NN}) := \mu_\eps^{N  } 
 \indc_{\{ Z_{\M+\NN} \in \cR^0_{{\mathbf S}, \bar \SS }  \}} 
 \; \Phi_\MM \big(Z_\M^\e(\theta , Z_{\M+\NN} )\big )     \prod _{j\in B }\left( \frac  { m_j!  }{  (m_j+n_j) !} \sgn(\SS_j) \right)   \, ,
\end{equation}
where $\cR_{{\mathbf S} , \bar \SS} ^0$ is the set of configurations such that a forward flow with $n_j$  annihilations in the block $j$ and compatible with $({\mathbf S}, \bar \SS)  $ exists, and where
\begin{equation}\label{defXi0}{\mathbf \Xi}^0 :=(\MM ,\NN  , {\mathbf S} , \bar \SS)\, .
\end{equation}
Our final result is then  \eqref{eq : Irec M N 0 dual} with
$$
\Phi^{0 }_{\M,\NN }=  \sum_{  \SS,\bar \SS}   \Phi^{0 }_{{\mathbf \Xi}^0}\;,
$$
where the sum over~$ \SS,\bar \SS$ runs in~$\{-1,1\}^{2n_j}, j \in B.$
\begin{Rmk}The symmetrisation over the labels of the particles, which was already an important argument   in~\cite{BGSS2}, is a key step of the procedure: it is not apparent when looking at the expectation since the sum over the partial permutations compensates exactly the combinatorial factor~$m_j!/(m_j+n_j) !$, on the other hand since the supports of the test functions are disjoint, it will be  a true gain when computing the variance.
\end{Rmk}
%
  
  \subsubsection{Duality: general case.}

In the case when  internal recollisions are allowed in backward pseudo-trajectories, the change of variables (\ref{change-variables})  is no longer  injective and, in order to apply our strategy, we need to control the number of   internal recollisions.
  The important fact is that   thanks to the conditioning~$\Upsilon_\cN^\eps$ introduced in Definition \ref{conditioning},   the configuration at time $\theta-\delta$ has no microscopic cluster of more than $\gamma$ particles, and the total energy of each microscopic cluster is at most $\gamma \bbV^2/2$ so that the variation of the  relative distance between  two particles from different clusters is at most~$2 \sqrt{\gamma} \,  \bbV \delta$, which prevents any collision during the time lapse~$\delta$.  Each cluster evolves therefore independently from the other clusters on the time interval~$[\theta-\delta, \theta] $.
 
 Furthermore the recollisions in each cluster cannot be due to periodicity since~$\bbV \delta \ll1$.
 Since the total number of collisions for  a system of $\gamma $ hard spheres in the whole space  is finite (see Theorem 1.3 in~\cite{Burago} or \cite{Illner89})
 say at most~$k_\gamma$, each particle in a pseudo-trajectory cannot have more than $K_\gamma = \sum_{\ell =2}^\gamma k_\ell$ recollisions during the short amount of time $\delta$. This crude upper bound on the number of recollisions takes into account the fact that the number of particles in a cluster may have varied on $[\theta-\delta, \theta]$  due the creation of new particles. 
We then associate with each particle $i$ an index $\kappa_i$ (less than $K_\gamma$) which is zero at time $\theta$ and  increased by one each time the particle undergoes a recollision in the backward pseudo-dynamics. We denote by~${\mathbf K}_{\M+\NN} $ the set of recollision indices~$(\kappa^{(j)}_\ell) _{\ell \leq m_j+ n_j }$ at time $\theta-\delta$. 
This new set of parameters enables us to recover the lost injectivity of~(\ref{mapping'}).
The
 construction of the forward dynamics starting from a configuration~$Z_{\M+\NN}$ is slightly more intricate. 
Fix a tree~$a \in \cA^\pm_{\M,\NN }$, a set of   indices~${\mathbf K}_{\M+\NN}  $ and the starting configuration~$Z_{\M+\NN} = (Z_{m_j+n_j}^{(j)} )_{j\in B }$. The forward flow
is uniquely defined  based on the following three possibilities, each time two particles encounter~:
\begin{itemize}
 \item[(a)] either the two  particles belong to two different blocks : in this case the particles are scattered and their indices are unchanged (this corresponds to an external recollision in the backward pseudo-trajectory);
\item[(b)] or the two particles belong to the same block  and have a positive index: in this case also the particles are scattered, and their indices are decreased by~1 (this corresponds to an internal recollision in the backward pseudo-trajectory);
\item[(c)] or the two particles belong to the same block and one particle has  zero index: one particle (with zero index) is annihilated and the other one is possibly scattered, as  prescribed by the collision tree $a$.  The indices are unchanged.
\end{itemize}

Finally we define, for each~$a \in \cA^\pm_{\M,\NN }$  and each~$ {\mathbf K} $  (we drop the index~$\M+\NN$  in~$ {\mathbf K}_{\M+\NN}$ for the sake of readability in the sequel),  the set~$\cR_{ {\mathbf K},a} $ of configurations compatible with backward pseudo-trajectories having the following constraints:
  \begin{itemize}
 \item[(i)] there are $n_j$ particles added to the block $j$;
 \item[(ii)]  the addition of new particles is prescribed by  $a^{(j)}$;  
 \item[(iii)]  internal recollisions  are compatible  with~$(\kappa^{(j)}_\ell) _{\ell \leq m_j+ n_j }$ (coded in $ {\mathbf K}$).
 \end{itemize}
Notice that, denoting by ${\mathbf K}={\mathbf 0}$ the set of all null recollision indices, $\cR_{ {\mathbf 0},a} = \cR_{  a}^0$.
The change of variables, as in \eqref{change of variables},
\begin{equation}
\label{change of variables'}
 (Z_\M, T_{\NN}, \Omega_{\NN}, U_\NN) \longmapsto Z^\e_{\M+\NN}(\theta-\delta ) \in \, \cR_{ {\mathbf K},a}
\end{equation}
is injective.
Denoting as previously by~$Z_\M^\e(\theta ,Z_{\M+\NN })$ the configuration of   the $M$ particles at time~$\theta $ starting from~$Z_{\M+\NN } \in \cR_{ {\mathbf K},a}$ at time~$\theta-\delta$, one  can therefore write  
  \begin{equation}\label{eq : Irec M N}
\begin{aligned}
I_{\M, \NN}  &:=  
 \int  \Big( Q^{\eps } _{\M,  \NN}( \delta )    G^{\eps }_{M+N}(\theta-\delta) \Big)(Z_\MM)\; \Phi_\MM (Z_\MM) dZ_\MM \\
& =  \sum_{a \in \cA^\pm_{\MM,\NN}}  \sum_{{\mathbf K}}{ \mu_\eps^{N }   }  \int_{\cR_{ {\mathbf K},a}} dZ_{\MM+\NN }
  G^{\eps}_{M+N} ( \theta-\delta,Z_{\MM+\NN } )\Phi_\MM \big(Z_\M^\e(\theta ,Z_{\M+\NN })\big )    \prod _{j\in B }\sgn(\SS_j) \, .  
  \end{aligned}
  \end{equation}
As    in~(\ref{defPhi0MNS}), we can use the exchangeability  of~$G^{\eps}_{M+N}$ and~$\phi^{(\varsigma_i)}$ to symmetrize partially the particles in each block~$j$  at time~$\theta-\delta$, by summing over~$ \sigma_j \in {\mathfrak S}^{n_j} _{m_j+n_j}$. The mapping   
$$ (a,  (\sigma_j)_{j\in B}, Z_\M, T_\NN, \Omega_\NN, U_\NN) \mapsto Z^\eps _{\M+\NN } (\theta- \delta)$$
  is injective for any fixed ${\mathbf K}$ and $(\SS, \bar \SS)$,  so   one can define the  $\delta$-pullback of test functions 
\begin{equation}
\label{eq: pushforward}
{\sharp}_\delta \Phi_\MM (Z_{\M+\NN}): = 
\Phi_\MM \big(Z_\M^\e(\theta , Z_{\M+\NN})\big ) 
\end{equation}
where the configuration $Z_\M^\e(\theta , Z_{\M+\NN})$ at time $\theta$ is obtained by the forward dynamics described above from the configuration $Z_{\M+\NN}$ at  time $\theta - \delta$.
Finally  we set
\begin{equation}
\label{defPhiNk}
  \Phi_{{\mathbf \Xi}}(Z_{\M+\NN}):= \mu_\eps^{N  } 
\; {\sharp}_\delta \Phi_\MM (Z_{\M+\NN}) \;
\indc_{\{ Z_{\M+\NN} \in \cR_{ {\mathbf K},\SS, \bar \SS} \}    } \prod _{j\in B }\left( \frac  { m_j!  }{  (m_j+n_j) !} \sgn(\SS_j)\right)  ,
\end{equation}
where $\cR_{ {\mathbf K},{\mathbf S}, \bar \SS }  $ is the set of configurations such that a forward flow with $n_j$ annihilations in the block $j$ and  compatible with $ ({\mathbf K},{\mathbf S}, \bar \SS )$ exists.
We have denoted as in~(\ref{defXi0})
\begin{equation}\label{defXi}
{\mathbf \Xi}  :=(\MM ,\NN  , {\mathbf K}, {\mathbf S}, \bar \SS )\, .
\end{equation}
For ${\mathbf K} = {\mathbf 0}$, note that~$\cR_{ {\mathbf 0},{\mathbf S}, \bar \SS} = \cR_{ {\mathbf S}, \bar \SS}^0$.   
 
 \medskip
 
Finally set
\begin{equation}\label{remove sum K S}
 \Phi_{\M,\NN } :=\sum_{ {\mathbf K},\SS, \bar \SS} \Phi_{{\mathbf \Xi}}  \, , 
\end{equation}
where the sum over~$ \SS,\bar \SS$ runs in~$\{-1,1\}^{2n_j}$, $ j \in B$ and the sum over~$ {\mathbf K}$ runs in $\{0, \dots , K_\gamma\}^{m_j+ n_j}$, with~$j\in B$.
Identity  \eqref{eq : Irec M N} can be rewritten  by duality
 \begin{align}
\label{eq : Irec M N dual}
\bbE_\eps   \Big[ \Big( \prod_{u=1} ^{p-1}  \zeta^\eps_{\theta_u} \big( h^{(u)}\big) \Big)  
  \pi^\eps_{M ,\theta} \big( \Phi_\M \big) \Big] 
=    \sum_{\NN   }  
\; \bbE_\eps   \Big[ \Big( \prod_{u=1} ^{p-1}  \zeta^\eps_{\theta_u} \big( h^{(u)}\big) \Big)  
  \pi^\eps_{M + N,\theta-\delta} \big( \Phi_{ \M,\NN } \big) \Big] \, .
  \end{align}
 

\subsection{Clustering structure} 
\label{factorization-sec}

Our aim is to study the transport of the factorization structure on an infinitesimal time interval~$[\theta-\delta, \theta]$. 
We are interested in 
\begin{equation}
\label{eq: I M}
\cI_\MM : = \bbE_\eps   \Big[  \Big( \prod_{u=1} ^{p-1}  \zeta^\eps_{\theta_u} \big( h^{(u)}\big) \Big)  \Otimes_{ i \leq  |\varsigma| } 
 \zeta^\eps_{M_{\varsigma_i},\theta} \big(   \phi ^{(\varsigma_i)} \big) \Big]
\end{equation}
for a partition~$\varsigma$ in packets of some $B\subset \{p, \dots, P\} $. As in \eqref{eq:defPhi}, the function~$ \phi ^{(\varsigma_i)}$ is evaluated at the configuration~$ ( Z_{m_j}^{(j)})_{j \in \varsigma_i} $ at time~$\theta$. We recall that~$M_{\varsigma_i}:=\displaystyle \sum_{j \in \varsigma_i} m_j$.  By analogy with  (\ref{tensor-product-eq}), we define the {\color{black}$\ostar$-product} for the conditioned fluctuation fields   by  discarding repeated indices
\begin{equation}
\label{tensor-product}
 \Otimes_{ i \leq  |\varsigma|  }   \zeta^\eps_{M_{\varsigma_i},\theta} \big(   \phi ^{(\varsigma_i)} \big):= \mu_\eps^{|\varsigma|/2} \sum_{\alpha \subset \{1, \dots,  |\varsigma|\}   }     \pi^\eps_{M_\alpha,\theta} \Big( \bigotimes_{i \in \alpha} \phi ^{(\varsigma_ i)} \Big) \prod_{ j \in \alpha^c} \bbE_\eps  [- \phi ^{(\varsigma_j)}]\,
\end{equation}
where $M_A = \sum_{i \in A} M_{\varsigma_i} $.

After using a pullback as in~\eqref{eq : Irec M N dual}, 
we would like to recover a factorized structure  with centered observables. This means that we need to  decompose the functions~$\Phi_{{\mathbf \Xi} } $  defined in~(\ref{defPhiNk}) into products, and take care of the counterterms corresponding to contributions of different  correlation functions.
The technical procedure implementing this program is a cumulant decomposition of trajectories  as devised in \cite{BGSS3} (which we apply here to dual functions).

%

\medskip
Let us first of all decompose the   product
\begin{equation}
\label{322'}\begin{aligned}
 \cI_\MM   
&= \mu_\eps^{   |\varsigma|/2} \sum_{\alpha  \subset   \{1, \dots ,|\varsigma|\}   } 
\left( \prod_{i\in   \alpha^c} \bbE_\eps  [-\phi ^{(\varsigma_i)}] \right) \times
\bbE_\eps   \Big[ \Big( \prod_{u=1} ^{p-1}  \zeta^\eps_{\theta_u} \big( h^{(u)}\big) \Big)  
    \pi^\eps_{M_\alpha,\theta }(\Phi_\alpha )
\Big]  \, ,
\end{aligned}
\end{equation}
denoting here   $\,\,\M_\alpha = \left(m_j\right)_{i\in \alpha, j\in \varsigma_i }$,  and 
$
\Phi_\alpha = \displaystyle  \bigotimes_{i\in \alpha} \phi^{(\varsigma_i)} \,.
$
To simplify notation, given the family~$\,\,\NN_\alpha = \left(n_j\right)_{i\in \alpha, j\in \varsigma_i }$ of added particles on $[\theta-\delta, \theta]$, we set~$M_\alpha^\delta := M_\alpha+N_\alpha$  for the number of particles in the sets~$(C_i)_{i\in \alpha}$ at time~$\theta-\delta$. 
As in~\eqref{eq : Irec M N dual},  we can write $$\begin{aligned}
& \bbE_\eps   \Big[ \Big( \prod_{u=1} ^{p-1}  \zeta^\eps_{\theta_u} \big( h^{(u)}\big) \Big)  
    \pi^\eps_{M_\alpha,\theta }(\Phi_\alpha )
\Big]\\
&\qquad\qquad\qquad\qquad\qquad =  \sum_{\NN_\alpha}\int   \Big(Q^{\eps } _{\M_\alpha,  \NN_\alpha}( \delta )    G^{\eps }_{M_\alpha+N_\alpha}(\theta-\delta) \Big)(Z_{\MM_\alpha}) \; \Phi_\alpha (Z_{\M_\alpha}) dZ_{\M_\alpha}
 \\
& \qquad\qquad\qquad\qquad\qquad =  \sum_{\NN_\alpha }  \; \bbE_\eps   \Big[ \Big( \prod_{u=1} ^{p-1}  \zeta^\eps_{\theta_u} \big( h^{(u)}\big) \Big)  
  \pi^\eps_{ M_\alpha^\delta ,\theta-\delta}\big(\Phi_{\MM_\alpha,\NN_\alpha}\big) \Big]   \, ,
\end{aligned}
$$
where as in~(\ref{remove sum K S})
$$
\Phi_{\MM_\alpha,\NN_\alpha}:= \sum_{{\mathbf K}_\alpha,{\mathbf S}_\alpha, \bar {\mathbf S}_\alpha}\Phi _{{\mathbf \Xi}_\alpha}\, .
$$
From now on to lighten further the notation we omit the dependence on the number of variables and set
$$
\Phi_{\alpha , \delta} :=\Phi_{\MM_\alpha,\NN_\alpha}\, , $$
so that  
\begin{equation}
\label{eq: 3.19}
   \bbE_\eps   \Big[ \Big( \prod_{u=1} ^{p-1}  \zeta^\eps_{\theta_u} \big( h^{(u)}\big) \Big)  
    \pi^\eps_{M_\alpha,\theta }(\Phi_\alpha )
\Big]\  =  \sum_{\NN_\alpha }  \; \bbE_\eps   \Big[  \Big( \prod_{u=1} ^{p-1}  \zeta^\eps_{\theta_u} \big( h^{(u)}\big) \Big)  
  \pi^\eps_{ M_\alpha^\delta ,\theta-\delta}\big(\Phi_{\alpha , \delta}\big) \Big]   \, .
\end{equation}
Note that in particular  there holds 
\begin{equation}
\label{eq: 1 site}
\bbE_\eps  [\phi ^{(\varsigma_i)}]
 = \bbE_\eps   \Big[  \pi^\eps_{M_{\varsigma_i},\theta}\big(  \phi ^{(\varsigma_i)} \big) \Big]
= \sum_{N_{\varsigma_i} } 
\bbE_\eps   \Big[    \pi^\eps_{{M_{\varsigma_i} ^\delta}, \theta-\delta }
\big(  \phi ^{(\varsigma_i)}_\delta \big)  \Big] = \sum_{\NN_{\varsigma_i} } 
\bbE_\eps   \Big[  \phi ^{(\varsigma_i)}_\delta   \Big]  \, ,
\end{equation}
 where as above~$ \phi ^{(\varsigma_i)}_\delta$ is defined after a summation over~$\KK_{\varsigma_i} ,\SS_{\varsigma_i}, \bar \SS_{\varsigma_i}$. 

\medskip

Let us analyse~$\Phi_{\alpha , \delta}$.  
Setting $\M_\alpha^\delta = \M_\alpha+\NN_\alpha$, one has  
\begin{equation} 
\label{eqPhiad}
\Phi_{\alpha , \delta}(Z_{\M_\alpha^\delta})= \sum_{{\mathbf K}_\alpha,{\mathbf S}_\alpha, \bar \SS_{\alpha}}\Phi _{{\mathbf \Xi}_\alpha}(Z_{\M_\alpha^\delta})\, ,
\end{equation}
and the function $\Phi_{\alpha , \delta}$ is   supported on configurations  at time~$\theta-\delta$ of the backward pseudo-trajectories corresponding to the packets~$  (C_i)_{i \in \alpha}$.   
We stress the fact that the variable decomposition among the blocks is still encoded  in $Z_{\M_\alpha^\delta}$.
If these pseudo-trajectories  were evolving independently, then each of them would  
lead to a dual function $\phi ^{(\varsigma_i) }_\delta$ and the product form would be exact.
 Even though this  product form is the main part, there are further contributions due to   dynamical correlations between the packets $(C_i)_{i\in \alpha}$ which we are going to analyze below. 

 We are going to group packets $(C_i)_{i\in \alpha}$ which are connected by 
 (external) recollisions. 
Denote by~$ \cP_\alpha$ the set of  partitions of~$\alpha$. Given~$\lambda \in \cP_\alpha$, we  restrict the change of variables~\eqref{change of variables'}
to  the pseudo-trajectories such that 
a chain of recollisions occurs in each set $(\lambda_\ell)_{\ell \leq |\lambda|}$,    meaning that the  
graph with packets $C_i$ as vertices and  recollisions as edges has connected components specified by $\lambda_\ell$. In particular,  the pseudo-trajectories from two different connected components $\lambda_{\ell_1}$ and $\lambda_{\ell_2}$ do not approach, which we will denote by $\indc_{\lambda_{\ell_1} \not\sim \lambda_{\ell_2}}$.
Each connected component $\lambda_\ell$ will be called a {\it forest}.  By extension, the blocks (and particles) of the associate packets  will be said to belong to~$\lambda_\ell$. Denoting by~$\cR^{ \lambda}_{{\mathbf K}_\alpha, {\mathbf S}_\alpha, \bar \SS_\alpha}$ the corresponding restriction of~$\cR_{{\mathbf K}_\alpha,  {\mathbf S}_\alpha, \bar \SS_\alpha}$, we get the change of variables
\begin{equation}
\label{eq: R lambda K}
 (Z_{\M_\alpha}, T_{\NN_\alpha}, \Omega_{\NN_\alpha}, U_{\NN_\alpha} ) \longmapsto 
 Z^\e_{\M_\alpha^\delta}(\theta-\delta ) \in \, \cR^{ \lambda}_{{\mathbf K}_\alpha,  {\mathbf S}_\alpha, \bar \SS_\alpha}\,   ,
\end{equation}
 which is injective: by construction when two particles meet in the forward flow, they have to belong to the same forest~$\lambda_\ell$ and the rule upon encounter (disappearance of a particle or not, scattering or not) is given by Definition~\ref{rules of the evolution} with parameters~${\mathbf K}_\alpha,  {\mathbf S}_\alpha, \bar \SS_\alpha$. In the following we shall denote by~$ \cR^{\lambda_\ell}_{{\mathbf K}_{\lambda_\ell},\SS_{\lambda_\ell}, \bar \SS_{\lambda_\ell}}$ the set~$\cR^{ \lambda}_{{\mathbf K}_\alpha,  {\mathbf S}_\alpha, \bar \SS_\alpha}$  restricted on a single forest.
 The definition~\eqref{defPhiNk} can be extended to dual functions 
with the constraint above:
 \begin{equation}
\label{forest}
\begin{aligned}\Phi_{{\mathbf \Xi}_\alpha}(Z_{\M_\alpha^\delta})   &  = 
\sum_{\lambda \in  \cP _\alpha}
  \mu_\eps^{N_\alpha } 
\Phi_{\alpha}\big(Z_{\M_\alpha}^\e(\theta , Z_{\M_\alpha^\delta })\big ) 
\indc_{\{ Z_{\M_\alpha^\delta } \in \cR^{\lambda}_{{\mathbf K}_\alpha,\SS_\alpha, \bar \SS_\alpha } \}}\prod _{i\in \alpha \atop j\in \varsigma_i }\left( \frac  { m_j!  }{  (m_j+n_j) !} \sgn(\SS_j) \right)\\
&  = \sum_{\lambda \in  \cP _\alpha}
 \prod_{\ell= 1}^{|\lambda|} \Big(
\Phi_{\lambda_\ell } \big(Z_{\M_{\lambda_\ell}}^\e(\theta , Z_{\M_{\lambda_\ell}+\NN_{\lambda_\ell}})\big )
\widetilde\varphi_{\lambda_\ell} ( Z_{\M_{\lambda_\ell} +\NN_{\lambda_\ell}}) \Big)
  \times \prod_{ \ell_1 \not = \ell_2 } {\mathbf 1}_{ \lambda_{\ell_1} \not \sim \lambda_{\ell_2} }(Z_{\M_\alpha+\NN_\alpha}) \\
  &  = \sum_{\lambda \in  \cP _\alpha}
 \prod_{\ell= 1}^{|\lambda|} \Big(
({\sharp}_\delta \Phi_{\lambda_\ell} )
\widetilde\varphi_{\lambda_\ell}\Big)
  \times \prod_{ \ell_1 \not = \ell_2 } {\mathbf 1}_{ \lambda_{\ell_1} \not \sim \lambda_{\ell_2} }(Z_{\M_\alpha+\NN_\alpha})  \, ,
\end{aligned}
\end{equation}
denoting by $({\sharp}_\delta \Phi_{\lambda_\ell} )$ the $\delta$-pullback of $ \Phi_{\lambda_\ell}$ by the dynamics as in \eqref{eq: pushforward}, and 
where the contribution of a forest is
\begin{equation}
\label{eq: tilde varphi}
 \widetilde\varphi_{\lambda_\ell}( Z_{\M_{\lambda_\ell} +\NN_{\lambda_\ell}}) 
 :=   \mu_\eps^{N_{\lambda_\ell} } \indc_{ \{Z_{\M_{\lambda_\ell}+ \NN_{\lambda_\ell}}    \in \cR^{\lambda_\ell}_{{\mathbf K}_{\lambda_\ell},\SS_{\lambda_\ell}, \bar \SS_{\lambda_\ell}} \} }  
\prod _{i\in \lambda_\ell \atop j\in \varsigma_i }\left( \frac  { m_j!  }{  (m_j+n_j) !} \sgn(\SS_j)\right) \,.
\end{equation}
We have used the fact that  the  pseudo-trajectories associated with different forests do not intersect in order to decouple the $\Phi_{\lambda_\ell }$.    
  The function $\widetilde \varphi_{\lambda_\ell}$ encodes in particular  the correlations due to  encounters between particles of different packets.  
    
    \medskip
    
A  correlation remains through the dynamical exclusion condition expressed by the constraint 
$$
\prod_{ \ell_1 \not = \ell_2 } {\mathbf 1}_{ \lambda_{\ell_1} \not \sim \lambda_{\ell_2} }(Z_{\M_\alpha+\NN_\alpha}) \;,
$$ encoding the fact that no  encounter should occur between the particles in different forests~$\lambda_{\ell_1}$ and~$ \lambda_{\ell_2}$. 
We will expand this exclusion condition writing
$ \indc_{ \lambda_{\ell_1} \not \sim \lambda_{\ell_2}} = 1- \indc_{ \lambda_{\ell_1}  \sim \lambda_{\ell_2}}\,,$ and 
defining the following notion.
  \begin{Def}
\label{def:overlap}
An  \underline{overlap} occurs between two forests $\lambda_{\ell_1}, \lambda_{\ell_2}$ (which is denoted by $\lambda_{\ell_1} \sim \lambda_{\ell_2}$) if two pseudo-particles from  $\lambda_{\ell_1}$ and  $\lambda_{\ell_2}$  find themselves at a distance   less than~$\eps$ one from the other at some time.
\end{Def} 
\noindent
Note that an overlap between two forests is   a mathematical artefact to analyze the dynamical correlations. In particular, it does not modify the dynamics in the forests.
 \begin{Def}[Extended encounter rules] \label{rules of the evolution}
Given a set of   $\kappa$-indices~${\mathbf K} $, a set of signs $(\SS, \bar \SS)$ and a partition $\lambda$ in forests, the forward flow  starting from some configuration~$Z_{\M+\NN} = (Z_{m_j+n_j}^{(j)} )_{j\in B }$  is reconstructed according to the following   rules each time two particles  encounter:
\begin{itemize}
 \item either the two  particles belong to two different forests~: they do not see each other. The $\kappa$-indices are unchanged;
  \item or   the two  particles belong to two different blocks but to the same forest~: they are scattered. The $\kappa$-indices are unchanged;
\item or the two particles belong to the same block and have a positive $\kappa$-index: they are scattered. Both indices are decreased by~1;
\item or the two particles belong to the same block and one particle has  zero $\kappa$-index: one particle (with zero index) is annihilated. The label of the particle which is annihilated, and the possible scattering of the other colliding particle are prescribed by the signs~$(\SS, \bar \SS)$. The other indices are unchanged.
\end{itemize}
\end{Def}

\bigskip

In order to identify all possible correlations, we introduce  now a cumulant expansion of the non overlapping constraint
\begin{equation}\label{def-phirho}
\begin{aligned}
\prod_{ \ell_1 \not =\ell_2 \atop 1 \leq\ell_1, \ell_2 \leq |\lambda|} 1_{ \lambda_{\ell_1} \not \sim \lambda_{\ell_2}}
&=  \sum_{G \in \GG_{|\lambda|}} \prod _{\{\ell_1, \ell_2\} \in E(G)} (-\indc _{ \lambda_{\ell_1}\sim \lambda_{\ell_2}} ) =\sum_{\rho \in \cP_{ |\lambda|}}  
\prod_{q=1}^{{ |\rho|}} \varphi_{\rho_q} \, ,
\end{aligned}
\end{equation}
where $\GG_{|\lambda|}$ is the set of graphs $G$ with $|\lambda|$ vertices,   $E(G)$ denotes the set of edges of a graph~$G$, and  the cumulants are defined on the connected components $\rho_q$ of $\{1, \dots, |\lambda|\}$ by
\begin{equation}
\label{defphiro}
\varphi_{\rho_q} =
\sum_{G'\in \cC_{ \rho_q}} \prod_{\{ \ell_1,\ell_2\}  \in E (G')} (-\indc_{\lambda_{\ell_1} \sim \lambda_{\ell_2}})\;,
\end{equation}
denoting  $\cC_{ \rho_q}$  the set of connected graphs with vertices $\rho_q$. In particular, the function
$\varphi_{\rho_q}$ is supported on clusters formed by overlapping   forests.

Combining (\ref{forest}) with (\ref{def-phirho}), we get
$$\Phi_{{\mathbf \Xi}_\alpha } =   \sum_{\lambda\in \cP _\alpha}   \sum_{\rho \in \cP_{|\lambda|} }  \prod_{\ell= 1}^{|\lambda|} \big(
{\sharp}_\delta\Phi_{\lambda_\ell } \widetilde\varphi_{\lambda_\ell}  \big)
 \prod_{q=1}^{|\rho|} \varphi_{\rho_q}\;,$$
 denoting by $({\sharp}_\delta \Phi_{\lambda_\ell} )$ the $\delta$-pullback of $ \Phi_{\lambda_\ell}$ by the dynamics as in \eqref{eq: pushforward}.
Exchanging the order of the sums,  we end up with the following (scaled) cumulant expansion  
\begin{equation}
\label{eq: cumulant decomposition}
\Phi_{{\mathbf \Xi}_\alpha}=    \sum_{\eta \in \cP _\alpha} 
\; \prod_{q = 1}^{|\eta|}  \mu_\eps^{1-|\eta_q|} \phi_{\delta,{\mathbf \Xi}_{\eta_q}} ^{(\eta_q)}\,   ,
\end{equation}
where the  (dual) cumulants are defined for any subset~$\eta_q$ of $\{1, \dots, |\varsigma|\}$ by
\begin{equation} 
\label{eq:defcum}
\begin{aligned}
 \phi_{\delta,{\mathbf \Xi}_{\eta_q}}^{(\eta_q)} :=  \mu_\eps^{|\eta_q |-1}   \sum_{\lambda \in \cP_{\eta_q}  }
 \Big( \prod_{\ell =1}^{|\lambda|} {\sharp}_\delta \Phi_{\lambda_\ell }
 \; \widetilde \varphi_{\lambda_\ell} ( Z_{\M_{\lambda_\ell} +\NN_{\lambda_\ell}})  \Big)  \varphi_{\{\lambda_1, \dots, \lambda_{|\lambda|}\}} \,.
 \end{aligned}
\end{equation}
Recall that $\widetilde \varphi$ encodes the external recollisions between packets  in each forest  and keeps track of~$\KK_{\lambda_\ell},\SS_{\lambda_\ell}, \bar \SS_{\lambda_\ell}$, while $\varphi$ encodes the overlaps between forests. Finally we set  
\begin{equation} 
\label{eq:defcum-sum}
 \phi_\delta^{(\eta_q)} (Z_{M_{\eta_q} ^\delta}) :=\sum_{\KK_{\eta_q} , \SS_{\eta_q }, \bar \SS_{\eta_q} }\phi_{\delta,{\mathbf \Xi}_{\eta_q}}^{(\eta_q)}
\end{equation}
so that, plugging (\ref{eq: tilde varphi}) in (\ref{eq:defcum}) and denoting $\sgn (\SS_{\eta_q})$ the product of all scattering signs~$\SS_{\eta_q}$, we obtain
\begin{equation}
\label{defphideltaetai}
\begin{aligned}
 \phi_\delta^{(\eta_q)} 
  =  \mu_\eps^{N_{\eta_q} + |\eta_q |-1} &\left( \prod _{i  \in \eta_q\atop j\in \varsigma_i}\frac  { m_j!  }{  (m_j+n_j) !} \right) \sum_{\KK_{\eta_q} , \SS_{\eta_q }, \bar \SS_{\eta_q} }
 \sum_{\lambda \in \cP_{\eta_q}  }
   \sgn (\SS_{\eta_q}  )   
  \varphi_{\{\lambda_1, \dots, \lambda_{|\lambda|}\}}\\
  &\qquad \qquad\times   \Big(  \prod_{\ell =1}^{|\lambda|}({\sharp}_\delta \Phi_{\lambda_\ell } ) \indc_{ \{Z_{\M_{\lambda_\ell}+ \NN_{\lambda_\ell}}    \in \cR^{\lambda_\ell}_{{\mathbf K}_{\lambda_\ell},\SS_{\lambda_\ell}, \bar \SS_{\lambda_\ell} } \} }  
  \Big) 
  \, .
  \end{aligned}
\end{equation}
As a consequence of \eqref{eqPhiad}, \eqref{eq: cumulant decomposition} and \eqref{eq:defcum-sum}, we finally obtain a cumulant decomposition for any  collection $\alpha$ of packets 
\begin{equation}
\label{eq: cumulant decomposition alpha}
\Phi_{\alpha , \delta}(Z_{\MM_\alpha^\delta}) =    \sum_{\eta \in \cP _\alpha} 
\; \prod_{q = 1}^{|\eta|}  \mu_\eps^{1-|\eta_q|} \phi_{\delta} ^{(\eta_q)} \,   ,
\end{equation}
where $\M_\alpha^\delta = \MM_\alpha + \NN_\alpha$.

 By definition,~$\phi_\delta^{(\eta_q)}$ corresponds to the contribution of packets 
$(C_i)_{i\in \eta_q}$ which are completely connected dynamically by encounters. 
From the definition \eqref{322'} of $ \cI_\MM$, identities \eqref{eq: 3.19}-\eqref{eq: 1 site}  and
 the cumulant decomposition \eqref{eq: cumulant decomposition alpha}, we arrive at \begin{align*}
 \cI_\MM & =\mu_\eps^{ |\varsigma|/2} 
\sum_{\bf N } \sum_{\alpha  \subset \{1,\dots, |  \varsigma|\}  } 
 \left(   \prod_{ i \in   \alpha^c} \bbE_\eps  [- \phi_\delta^{( \varsigma_i )}] \right)\\
& \qquad  \qquad \qquad  \qquad \qquad \times \sum_{\eta \in \cP _\alpha} 
\bbE_\eps   \Big[ \Big( \prod_{u=1} ^{p-1}  \zeta^\eps_{\theta_u} \big( h^{(u)}\big) \Big)   \pi^\eps_{ M_\alpha^\delta  ,\theta-\delta} \Big( \bigotimes_{q = 1} ^{|\eta|}   \mu_\eps^{1-|\eta_q|}  \phi_\delta^{(\eta_q)}  \Big)\Big] \, .
\end{align*}
Note that $\phi_\delta^{( \varsigma_i)}$ is indexed by a single set $\varsigma_i$ so that it is constructed without resorting to forests or overlaps.
Moreover we can decompose each cumulant in the sum of its expectation and its  fluctuation
\begin{equation}
\label{exp-fluct-dec}
\begin{aligned}
  \pi^\eps_{ M_{\eta_q}^\delta   ,\theta-\delta}   ( \phi_\delta^{(\eta_q)}  ) & = \mu_\eps^{-\frac12} \zeta^\eps_{ M_{\eta_q}^\delta  ,\theta-\delta} ( \phi_\delta^{(\eta_q)}  )+ {\mathbb E} _\eps[ \phi_\delta^{(\eta_q)}   ] \, .  \end{aligned}
\end{equation}
We obtain (cf.\,\eqref{tensor-product})
\begin{align*}
\cI_\MM & =\mu_\eps^{ |\varsigma|/2} 
\sum_{\bf N } \sum_{\alpha  \subset \{1,\dots, |  \varsigma|\}  } 
\sum_{\eta \in \cP _\alpha}   \left(   \prod_{ i \in   \alpha^c} \bbE_\eps  [- \phi_\delta^{( \varsigma_i )}] \right)\\
& \times \sum_{\eta \in \cP _\alpha}  \sum_{ I \subset \{1, \dots , |\eta|\} } \prod_{q \in I^c} \mu_\eps^{1-|\eta_q|} \bbE_\eps\Big[\phi_\delta^{(\eta_q)}  \Big] 
\;\bbE_\eps   \Big[ \Big( \prod_{u=1} ^{p-1}  \zeta^\eps_{\theta_u} \big( h^{(u)}\big) \Big)    \Big( \Otimes_{q \in I}   \mu_\eps^{\frac12 -|\eta_q|}  \zeta^\eps_{ M_{\eta_q}^\delta ,\theta-\delta} (\phi_\delta^{(\eta_q)}  ) \Big)\Big] \, .
\end{align*}
For a given $\alpha$, we denote by $I$ the set of observables contributing to  the fluctuation field, and by $I^c$ the set of observables contributing via their  expectation. We then split $I^c$ into two parts
$${\color{black} I^c_1:= \{ q \in I^c \,/\, |\eta_q | = 1\}\, , \quad I^c_2 :=\{ q\in  I^c  \,/\, |\eta_q | \geq 2\}\, .}$$
We also define
$$ \beta_- =\alpha^c  \, ,  \quad  \beta_+ =  \bigcup_{q \in  I^c_1 } \eta_q  \, ,  \quad \beta_1 = \bigcup_{q \in I} \eta_q  \, , \quad \beta_2= \bigcup_{q \in  I^c_2} \eta_q \,,$$
and we denote (abusively) $\eta_{1}$ and $\eta_{2}$ the restriction of $\eta$ to $\beta_1$ and $\beta_2$ respectively. By definition, $\eta_{2}$ has no singleton: recall that  as defined in Section~\ref{factorization-subsec}, such a partition is called  {\it clustering}. 
Then by Fubini, 
\begin{align}
 \label{eq:partnosing}
\cI_\MM & =\mu_\eps^{ |\varsigma|/2} 
\sum_{\bf N }\sum_{\beta_1, \beta_2 , \beta_-, \beta_+ \atop \hbox{\tiny \rm partition of } \{1, \dots, | \varsigma]\} } 
\sum_{\eta_{1} \in \cP_{\beta_1}} \sum_{\eta_{2}\in \cP_{\beta_2} \atop \eta_2 \hbox{ \tiny\rm  clustering}} 
\prod_{ \ell_- \in   \beta_- } \bbE_\eps  [-\phi _\delta^{( \ell_-)}]\\& 
\qquad \qquad\times\prod_{\ell_+ \in \beta_+} \bbE_\eps  [\phi _\delta^{(\ell_+)}]  
 \prod _{1\leq q\leq |\eta_2| } \mu_\eps^{1-|\eta_{2,q}|}\bbE_\eps[  \phi_\delta^{(\eta_{2,q})}  ]  
 \nonumber \\
& \qquad  \qquad  \times 
\bbE_\eps   \Big[  \Big( \prod_{u=1} ^{p-1}  \zeta^\eps_{\theta_u} \big( h^{(u)}\big) \Big)    \Big( \Otimes_{1\leq q\leq |\eta_1| }   \mu_\eps^{\frac12 -|\eta_{1,q}|}  \zeta^\eps_{ M_{\eta_{1,q}}^\delta ,\theta-\delta} (\phi_\delta^{(\eta_{1,q})} ) \Big)\Big] \, . \nonumber
\end{align}
Fixing $\beta_1$ and $\beta_2$, we see that  the sum over $\beta_-, \beta_+$ is zero as soon as $\beta_1\cup \beta_2 \neq  \{1, \dots, |\varsigma|\} $.
We find \begin{align*}
\cI_\MM & =\mu_\eps^{| \varsigma|/2} 
\sum_{\bf N } \sum_{\beta_1, \beta_2  \atop \hbox{\tiny \rm   partition of } \{1, \dots, |\varsigma|\}}
\sum_{\eta_{1} \in \cP_{\beta_1}} \sum_{\eta_{2}\in \cP_{\beta_2} \atop \eta_2 \hbox{ \tiny\rm  clustering} }   \prod _{1\leq q\leq |\eta_2| } \mu_\eps^{1-|\eta_{2,q}|} \bbE_\eps\Big[ \phi_\delta^{(\eta_{2,q})}  \Big]   \\
& \qquad  \qquad  \times 
\bbE_\eps   \Big[ \Big( \prod_{u=1} ^{p-1}  \zeta^\eps_{\theta_u} \big( h^{(u)}\big) \Big)    \Big( \Otimes_{1\leq q\leq |\eta_1| }   \mu_\eps^{\frac12 -|\eta_{1,q}|}  \zeta^\eps_{ M_{\eta_{1,q}}^\delta ,\theta-\delta} (\phi_\delta^{(\eta_{1,q})}  ) \Big)\Big]\, .
\end{align*}
  Finally given~$\eta_{1}$ and~$\eta_{2}$
we decompose 
$
| \varsigma| =   \sum_{q } |\eta_{1,q}| +   \sum_{q } |\eta_{2,q }|
$
and we arrive at the following  identity. 

\begin{Prop}
\label{prop-from theta to theta + delta}
Consider  a   partition $\varsigma$ of a set $B \subset \{p, \dots, P\}$, indexing the  test functions~$(\phi ^{(\varsigma_i)} )_{i \leq  |\varsigma| }$ as in~\eqref{eq:defPhi}.
Then  for any~$\theta= \theta_p - r\delta$ with ~$r \in [0, (\theta_p - \theta_{p-1})/\delta]$, there holds with notation~{\rm(\ref{defphideltaetai})}
\begin{equation}
\label{Duhamel-identite}
\begin{aligned}
&\bbE_\eps   \Big[ \Big( \prod_{u=1} ^{p-1}  \zeta^\eps_{\theta_u} \big( h^{(u)}\big)\Big)  \Otimes_{ i \leq  |\varsigma| }
 \zeta^\eps_{M_{\varsigma_i} ,\theta} \big(   \phi ^{(\varsigma_i)} \big) \Big] \\
  &\qquad  =
\sum_{\bf \NN }\sum_{\eta_1 \cup \eta_2 \hbox{ \tiny \rm partition of } \{1, \dots, |\varsigma|\}\atop \eta_2 \hbox{ \tiny\rm  clustering} }   \prod _{q= 1}^{|\eta_2|}\mu_\eps^{1-\frac{|\eta_{2,q}|}2} \bbE_\eps\Big[   \phi_\delta^{(\eta_{2,q})}  \Big]   \\
& \qquad  \qquad  \times 
\bbE_\eps   \Big[ \Big( \prod_{u=1} ^{p-1}  \zeta^\eps_{\theta_u} \big( h^{(u)}\big) \Big)    \Big( \Otimes_{q= 1}^{|\eta_1|}  \mu_\eps^{\frac12 -\frac{|\eta_{1,q}|}2}  \zeta^\eps_{ M_{\eta_{1,q}}^\delta ,\theta-\delta} \Big(  \phi_\delta^{(\eta_{1,q})}  \Big) \Big)\Big] \,.
\end{aligned}
\end{equation}
\end{Prop}

The algebraic identity  (\ref{Duhamel-identite}) extends formula (\ref{Duhamel-identity-block})  to take  into account the structure of packets at time $\theta$. The length $\delta$ of the time interval is limited only by the fact that we need to control the number of internal recollisions uniformly in $\e$ (so that the sums over~$\KK_{\eta_{i,q}}$ 
defining~$  \phi_\delta^{(\eta_{i,q})} $ 
are finite). Extending this time interval would require to modify the conditioning, but then $\indc_{(\Upsilon^\eps_N)^c}$ would not be a negligible correction.


\section{Extracting minimal cumulants} 
\label{sec:tau}

In this section, we aim at iterating  Proposition \ref{prop-from theta to theta + delta} to  pull back the fluctuation structure on an  intermediate  time scale~$\tau$ such that $\delta \ll \tau \ll 1$.  
For the sake of simplicity, we choose $\tau$ such that $R:=\tau / \delta$ is an integer.  
Let $\theta = \theta_p -k\tau $ for some integer~$k$ be such that~$[\theta-\tau, \theta]\subset[ \theta_{p-1}, \theta_p]$.

\medskip

\subsection{Backward iterated clustering}

Let $B= \{ b_\ell \,| \   \ell = 1, \dots, |B|\}$ be a subset of $\{p, \dots, P\} $
indexing the test functions $(\phi ^{(i)})_{i \in B}$ at time $\theta$.
As explained in Section \ref{factorization-subsec},  the strategy is to  iterate~$R= \tau/\delta $ times the formula  (\ref{Duhamel-identite})  down to time~$\theta-\tau$. We therefore construct iteratively on each time step~$[\theta - r \delta, \theta - (r-1) \delta]$ for~$r =1,\dots,R$       nested partitions
 $\eta_1^{r-1} \hookrightarrow \eta^r_1 \cup \eta^r_2 $ with $\eta^r_2$ corresponding to (non trivial) packets which are expelled from the main factorized structure, contributing only via their expectation, and $\eta_1^r$ corresponding to packets contributing to the factorized structure via their fluctuations.

Keeping  track of all the intermediate~$\eta^r$ for~$1 \leq r \leq R$ would imply rapidly growing combinatorics. Therefore, at time $\theta-r \delta$,  we are going to identify cumulants of the form $\phi_{r\delta}^{(A)}$, with $A \subset B$ by taking into account all possible clustering dynamics leading to a given cluster indexed by $A$, whatever the successive partitions 
$  \eta^1_1, \dots, \eta^r_1$ : the identification of packets of packets with the union of these packets is therefore essential to gather all contributions.
This identification will be key to control the combinatorics.

\begin{figure}[h] 
\includegraphics[width=5in]{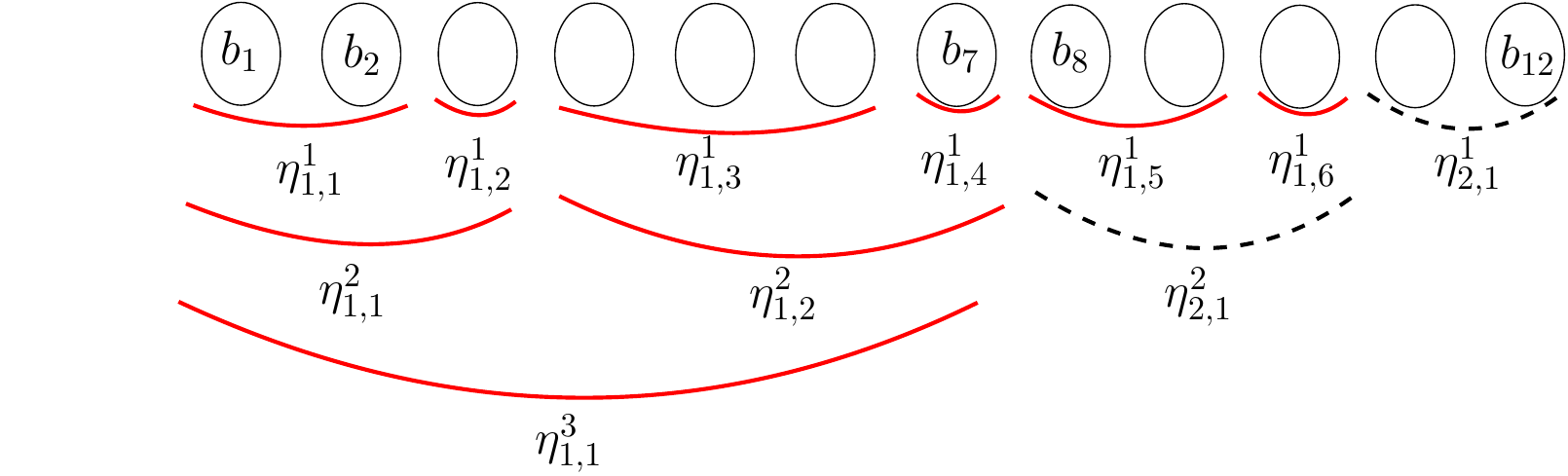}  
\caption{
The nested partitions  are depicted and the dashed parts represent the expelled clustering cumulants.
 By summing over all the possible intermediate decompositions of $\eta^3_{1,1}=\{ b_1, \dots , b_7 \}$, we will recover the cumulant~$\phi^{(\eta^3_{1,1})}_{3 \delta}$.
The expelled cluster will be collected in the set  $\rho^3 = \{ \eta^1_{2,1}, \eta^2_{2,1} \}$. 
}
 \label{fig:fubini nested}
\end{figure}


 \bigskip

On the other hand,  the  combinatorics  encoding all possible elementary pullbacks  on $[\theta-r\delta, \theta-(r-1)\delta]$   leading to a given configuration  at time $\theta-r\delta$,  is also quite bad and by iteration will be out of control. Indeed,  the number of possible $\KK $ describing the forward dynamics on $[\theta-r\delta, \theta-(r-1)\delta]$  is $ K_\gamma ^{M^{r\delta}_B} $  (recalling~$M_B^{r\delta }= M_B^{(r-1)\delta} +N^r_B$).  In particular, we do not expect that formula (\ref{Duhamel-identite}) can be iterated brutally~$O(\tau /\delta)$ times without having a strong divergence.  
   The strategy to avoid this   divergence will consist in retaining  at each time step~$[\theta-(r-1)\delta, \theta-r\delta]$,   only ``local minimal cumulants" defined by pulling back the observables along backward pseudo-trajectories such that 
   \begin{itemize}
\item[(i)]  internal recollisions (inside any block $j$) are forbidden, i.e.  ${\mathbf K}_j = 0$;
 \item[(ii)]  recollisions between blocks in any given packet are forbidden;
\item[(iii)]  the graph encoding the recollisions between packets of any given forest has to be minimally connected;
\item[(iv)]  the graph encoding the overlaps between the different forests has to be minimally connected.
\end{itemize}
We then define
   \begin{equation}
\label{minimal}
\begin{aligned}
\bar  \phi_\delta^{(\eta_q)} 
 : =  \mu_\eps^{N_{\eta_q} + |\eta_q |-1} &\left( \prod _{i  \in \eta_q\atop j\in \varsigma_i}\frac  { m_j!  }{  (m_j+n_j) !} \right) \sum_{ \SS_{\eta_q }, \bar \SS_{\eta_q} }
 \sum_{\lambda \in \cP_{\eta_q}  }
   \sgn (\SS_{\eta_q}  )   
 \bar  \varphi_{\{\lambda_1, \dots, \lambda_{|\lambda|}\}}\\
  &\qquad \qquad\times   \Big(  \prod_{\ell =1}^{|\lambda|}({\sharp}_\delta \Phi_{\lambda_\ell } ) \indc_{ \{Z_{\M_{\lambda_\ell}+ \NN_{\lambda_\ell}}    \in \bar \cR^{\lambda_\ell}_{\SS_{\lambda_\ell}, \bar \SS_{\lambda_\ell} } \} }  
  \Big) 
  \, ,
  \end{aligned}
\end{equation}
where $\bar \cR^{\lambda_\ell}_{\SS_{\lambda_\ell}, \bar \SS_{\lambda_\ell} }$ is the set of configurations compatible with pseudo-trajectories in a forest satisfying (i)(ii)(iii), and  $\bar  \varphi$ is the restriction of $\varphi$ to minimally connected graphs according to (iv). In particular, denoting by~$\bar \cR^{\lambda }_{\SS_{\lambda }, \bar \SS_{\lambda  } } $ the resulting set of configurations, there holds
 \begin{equation}
 \label{varphibar}
\bar  \varphi_{\{\lambda_1, \dots, \lambda_{|\lambda|}\} }    \indc_{ \{Z_{\M_{\lambda}+ \NN_{\lambda}}    \in \bar \cR^{\lambda}_{\SS_{\lambda}, \bar \SS_{\lambda} } \} } = (-1)^{|\lambda| - 1}    \indc_{ \{Z_{\M_{\lambda}+ \NN_{\lambda}}    \in \bar \cR^{\lambda}_{\SS_{\lambda}, \bar \SS_{\lambda} } \} } 
\end{equation}

\bigskip
Only the contribution of such local minimal cumulants in (\ref{Duhamel-identite}) will be iterated. Starting with~$|B|$ blocks at time $\theta$, we therefore obtain with this partial  iteration ``minimal cumulants" defined by pulling back the observables along backward pseudo-trajectories such that 
   \begin{itemize}
\item[(i)]  internal recollisions (inside any block $j$) are forbidden, i.e.  $K_j = 0$;
 \item[(ii)]  the $|B|$ blocks are  dynamically connected according to a graph with exactly $|B|-1$ edges representing all encounters (corresponding to external recollisions and overlaps in the backward pseudo-trajectory). This graph is therefore minimally connected (meaning there are neither multiple edges nor cycles).
\end{itemize}
\begin{Rmk}\label{independence-rmk} Note that, by definition, particles which are in different packets at time $\theta - r \delta$ are independent on $[\theta- r \delta, \theta]$. In particular, even if they approach at a distance $\eps$ on $[\theta- r \delta, \theta]$, this is neither a recollision nor an overlap.
\end{Rmk}
%


\subsection{Iterated forward dynamics}\label{sec-forward dynamics}
 
For  minimal cumulants,   the forward dynamics can be encoded by simpler combinatorics. Indeed  to describe the partition in forests $\lambda$, it is enough to prescribe  a sequence of signs~$\EE \in \{-1, +1\} ^{|B|-1} $  encoding whether the $|B| - 1$ encounters between particles within different blocks have scattering  (sign +1) or not (sign -1). Since we further assume that there is no internal encounter without annihilation, we end up with a ``minimal  forward dynamics" which is parametrized only by the sets of signs $(\SS , \bar \SS) $ and~$\EE$. 

\begin{Def} [{\bf Minimal forward dynamics}]
\label{def minimal forward dynamics}
 Given~$1 \leq r \leq R$, let~$B$ be any subset of~$  \{1,\dots,P\}$, and consider  integers~$\MM_B:=(m_j )_{j\in B}$ and
 \begin{equation}
\label{eq: notation MMB r delta}
\MM_B^{r\delta} := \MM_B+\NN_B   := (m_j+n_j)_{j\in B}.
\end{equation}
 representing the particle numbers in  the blocks respectively at times~$\theta$ and $\theta-r\delta$.

 A \emph{minimal forward dynamics} on~$[\theta-r\delta, \theta]$ starting from $Z_{\M_B^{r\delta}} $ is completely prescribed by  two sequences $(\SS, \bar \SS)  = (\SS_j, \bar \SS_j)_{j\in B}$ with~$(\SS_j, \bar \SS_j ) \in \{-1, +1\} ^{2n_j }$, and~$\EE \in \{-1, +1\} ^{|B|-1} $. Moving forward in time, each time two particles approach at a distance $\eps$, 
\begin{itemize}
\item if they belong to the same block~$j$:     the particle to be removed and the possible scattering   are encoded by~$\SS_j, \bar \SS_j$;
\item if they belong to two different blocks:   a (signed), minimally connected graph $G_B$ with~$|B|-1$ edges is constructed iteratively by adding an edge decorated with a sign according to~$\EE$ if the two particles are not already in the same connected component of the graph. If the two particles are in the same connected component of $G_B$, then
\begin{itemize}
\item if the before-last edge in this connected component was created in a different time interval~$[\theta-r'\delta, \theta-(r'-1)\delta]$, the particles are unaffected  (this is actually not an encounter, see Remark {\rm\ref{independence-rmk}});
\item if it occurs in the same time interval, the configuration is not admissible (by definition, the graph representing all encounters has to be minimal).
\end{itemize}
 \end{itemize}
We say that  $Z_{\M_B^{r\delta}}$ is a {\rm minimal forward cluster} associated with  $(\SS, \bar \SS , \EE )$ if the $N_B$ annihilations occur, as well as the $|B|-1$ encounters making the graph $G_B$ connected. This is denoted by $Z_{\M_B^{r\delta}} \in \cR^{\rm min}_{\EE  , \SS, \bar \SS  } $. 

\smallskip

The case when~$B = \{ i\}$ is reduced to one singleton  is  stressed by the denomination {\rm single minimal forward cluster}.
Notice that a single minimal forward cluster is simply parametrised by $\SS, \bar \SS$ (as 
$\EE$  becomes irrelevant) and we will write 
$Z_{\M_B^{r\delta}} \in \cR^{\rm min}_{\SS, \bar \SS  }$. 
\end{Def}

\begin{Rmk}
 Note that  the time intervals $[\theta - r' \delta, \theta - (r'-1) \delta]$ (with $r' \leq r$)
 when the different   encounters  happen are not prescribed in the definition above.
\end{Rmk}
%
%
%
With this definition of minimal forward dynamics, we obtain the following representation of  minimal cumulants at time~$\theta-r \delta$:
\begin{equation}
\label{barphi-def}
\bar   \phi_{r \delta}^{(B)}  (Z_{\M_B^{r\delta}})
  := \mu_\eps^{N_B +|B| -1 } \left( \prod_{j \in B} {m_j! \over M_j^{r\delta} !} \right) \sum_{ \SS,\bar  \SS , \EE  }  \sgn(\EE) \sgn(\SS ) \indc_{Z_{\M_B^{r\delta}} \in \cR^{\rm min}_{\EE , \SS, \bar \SS } } 
  ({\sharp}_{r\delta} \bigotimes _{i \in B}  \; \phi^{(i)})  \,,
\end{equation}
where $\cR^{\rm min}_{\EE, \SS, \bar \SS}$ imposes that $Z_{\M_B^{r\delta}}$
is associated with a (unique) minimal forward dynamics on~$[\theta- r \delta, \theta]$.  
%
%
%
As in~\eqref{eq: pushforward}, the pullback during a time $r \delta$ is given by
\begin{equation*}
{\sharp}_{r\delta} \bigotimes_{i \in B} \; \phi^{(i)} (Z_{\M_B^{r\delta}}) := 
\bigotimes_{i \in B}  \; \phi^{(i)} \big(Z_{\M_B}^\e(\theta , Z_{\M_B^{r\delta}})\big) \, .
\end{equation*}
 \begin{Rmk}
 Note that  it is possible to prescribe a priori  the numbers $(n_j^{r'})_{j \in B, r' \leq r}$ of particles annihilated at each time step in each tree, in which case there are additional ``sampling" conditions on $\cR^{\rm min}_{\EE , \SS, \bar \SS }$. 
\end{Rmk}

%
%
%
After iterating Proposition \ref{prop-from theta to theta + delta} up to time $\theta- r \delta$ for~$1 \leq r \leq R$, we expect that 
the main contribution will be given by minimal cumulants of the form $\bar   \phi_{r \delta}^{(B)}$.
  There are however   additional terms  at each time step~$\theta - r'\delta$, for~$1 \leq r' \leq r$, which will  not be iterated; their contribution  will be shown to be negligible in the limit.
In order to analyse these error terms recursively for each $r$, we    need a more general notion of forward cluster on~$[\theta-r \delta, \theta]$,  such that  in the first interval~$[\theta-r \delta, \theta-(r-1) \delta]$ the forward dynamics is not necessarily minimal but it is minimal starting at time~$ \theta-(r-1) \delta$.  
\begin{Def} [{\bf Forward dynamics}]
\label{def forward dynamics}
 Given~$1 \leq r \leq R$, let~$B$ be any subset of~$  \{1,\dots,P\}$, and consider  integers~$\MM_B:=(m_j )_{j\in B}$, $\MM_B  ^{(r-1)\delta} :=\NN_B^{<r}   + \MM_B:= (m_j + n^{<r}_j )_{j \in B} $ and $\MM_B  ^{r\delta} :=\NN_B^r +\NN_B^{<r}   + \MM_B = (m_j +n^{<r}_j +n^r_j)_{j \in B}$  representing the particle numbers in the blocks respectively at times~$\theta$, $\theta-(r-1) \delta$ and $\theta-r \delta$.  
\smallskip
 
A \emph{forward dynamics} on $[\theta-r\delta, \theta]$ starting from $Z_{\M_B^{r\delta}} $ is completely prescribed by the following parameters~:
\begin{itemize}
\item for each~$j \in B$, a sequence $(\SS_j , \bar \SS_j)  \in \{-1, +1\} ^{2n_j }$ encoding the particle to be removed and the possible scattering at each encounter  between two particles of the same block~$j$.
The restriction to the time interval $[\theta- r \delta, \theta-(r-1)\delta]$ of these parameters is 
 denoted by $(\SS^r_j , \bar \SS^r_j)  \in \{-1, +1\} ^{2n^r_j }$ and the other parameters by 
 $(\SS^{<r}_j , \bar \SS^{<r}_j)  \in \{-1, +1\} ^{2n^{<r}_j }$;
\item a partition $\varsigma  $ of $B$, defining packets~$(C_i)_{i \leq |\varsigma |}$ at time $\theta - (r-1) \delta$;
\item a partition $\lambda  \in \cP_{ |\varsigma|}$ in forests on $[\theta- r \delta, \theta-(r-1)\delta]$;
\item a multi-index $\KK  \in \{0,\dots, K_\gamma\}^{M_B^{r\delta}} $ encoding the number of  internal encounters without annihilation per particle on $[\theta- r \delta, \theta-(r-1)\delta]$.
Note that within the packet $C_i$, particles in different blocks are always scattered when they encounter during the time interval $[\theta- r \delta, \theta-(r-1)\delta]$; 
\item for each $i\leq  |\varsigma|$, a sequence $\EE_i  \in \{-1, +1\}^{|\varsigma _i |- 1} $ encoding the types of encounters in the packet~$C_i$ on $[\theta-(r-1)\delta, \theta]$.
\end{itemize}
We say that $Z_{\M_B^{r\delta}}$ is a {\rm forward cluster}  associated with $(\SS , \bar \SS  , \varsigma , \lambda , \KK , (\EE_i )_{i\leq |\varsigma | })$ if all encounters occur in such a way that the graph  coding these encounters is completely connected, modulo the identification of the $m_j$ particles of each block $j$ at time $\theta$ as a unique vertex.
\end{Def}

\begin{Rmk}
 The  forward cluster~$Z_{\M_B^{r\delta}}$  associated with $(\SS , \bar \SS , \varsigma  , \lambda , \KK , (\EE_i )_{i\leq  |\varsigma | })$ is recovered 
 by creating $ |\varsigma | $ independent minimal forward clusters indexed by each $\varsigma_i $ up to time~$\theta-(r-1) \delta$.
The corresponding packets  are   linked dynamically to form a 
forward cluster at time $\theta- r\delta$.
 \end{Rmk}
 \begin{Rmk}
 Note that the partition into packets and the subpartition into forests are only prescribed on the first time interval~$[\theta- r \delta, \theta-(r-1)\delta]$: thanks to the minimality assumption there is no need to prescribe those objects at each intermediate time step. All possible divergences are therefore concentrated on the first time interval, whose size $\delta$ was chosen for them to remain  under control.
 \end{Rmk}

Fix $n_j^r \leq M_j^{r \delta}$ for any block $j$ in $B$.
Going back to the definition (\ref{defphideltaetai}) of cumulants, the forward dynamics  are related to a cumulant $\phi_{r \delta}^{(B)}$ which is obtained in terms of the minimal cumulants 
$\bar \phi_{(r-1)\delta} ^{(\varsigma_i ) }$ defined in~(\ref{barphi-def})  as follows
\begin{equation}
\label{notbarphi-def}
\begin{aligned}
 \phi_{r \delta}^{(B)}(Z_{\M_B^{r\delta}})
: =   \Big( \prod_{j \in B} {M_j^{(r-1)\delta}! \over M_j^{r\delta} !} \Big) \sum _{\varsigma  \in \cP_B}  &  \     \mu_\eps^{N_B^r + | \varsigma  |-1}  \sum_{\KK , \SS^r , \bar \SS^r ,\lambda \atop 
 \varsigma \hookrightarrow \lambda} \sgn (\SS^r  )   
   \varphi_{\{\lambda_1, \dots, \lambda_{|\lambda|}\}} \\
&     \times\prod_{\ell =1}^{|\lambda|}\Big({\sharp}_\delta  
\bigotimes_{\varsigma_i \subset \lambda_\ell} \bar \phi_{(r-1) \delta} ^{ (\varsigma_i)} \Big) 
\indc_{ \{Z_{\M_{\lambda_\ell}^{r\delta} }    \in \cR^{\lambda_\ell, \varsigma}_{{\mathbf K}_{\lambda_\ell},\SS^r_{\lambda_\ell},\bar \SS^r_{\lambda_\ell}} \} }  
   \, ,
\end{aligned}
\end{equation}
where for any $i \leq |\varsigma|$, the minimal cumulant $\bar \phi_{(r-1) \delta}^{ (\varsigma_i)}$
is coded by the parameters $\EE_i$ and $(\SS^{<r}_j , \bar \SS^{<r}_j)_{j \in \varsigma_i}$, and 
$\cR^{\lambda_\ell, \varsigma}_{{\mathbf K}_{\lambda_\ell},\SS^r_{\lambda_\ell},\bar \SS^r_{\lambda_\ell}}  $
is the set of configurations compatible with packets $\varsigma$ and forests $(\lambda_\ell)$ in the first time interval (cf.\,\eqref{eq: R lambda K}).

\bigskip
\subsection{Discarding non minimal dynamics in the  iteration} 
\label{subsec:EMC}

Starting from general functions  $\{ \phi ^{(i )} \}_{ i \in B}$ at time $\theta$, we have explained how they can be aggregated in \eqref{notbarphi-def}
 to form cumulants at time $\theta - r \delta$. 
In what follows, the structure of the functions $\{ \phi ^{(i )} \}_{ i \in B}$ will become relevant.
For $i \geq p$, we will assume that $\phi^{(i)} = \bar \phi^{(i)}$  is  built from a 
{\it single} minimal cumulant obtained by the pullback of the test function $h^{(i)}$ during the time interval $[\theta, \theta_i]$ 
\begin{equation}
\label{eq: cumulant induction}
\bar \phi^{(i)}  (Z_{m_i})= \frac{ \mu_\eps^{m_i-1}}{m_i!} \sum_{\SS_i^\theta , \bar \SS_i^\theta} \; 
\sgn(\SS_i^\theta ) \indc_{Z_{m_i} \in \cR^{\rm min}_{\SS_i^\theta, \bar \SS_i^\theta } } 
({\sharp}_{\theta_i-\theta} h^{(i)})\, ,
\end{equation}
where the signs $(\SS_i^\theta, \bar \SS_i^\theta) \in \{-1,1\}^{2(m_i-1)}$ prescribe the encounters in the time interval $[\theta, \theta_i]$
in the same way as in Definition \ref{def minimal forward dynamics}.
From now on, we shall write $\bar \phi^{(i)}$ instead of $\phi^{(i)}$ to emphasize the minimality of the cumulant. 
This structure will be crucial for the geometric estimates in Sections \ref{geometric-sec} and~\ref{variance-sec}.

\bigskip

The goal of this section is to prove the following approximate preservation of the fluctuation structure on the  generic time interval $[\theta- \tau, \theta]$ (or on a possibly smaller one), discarding non minimal dynamics.
\begin{Prop}
\label{prop-from theta to theta + tau}
Fix~$\theta = \theta_p -k\tau $ for some integer~$k$ such that~$[\theta-\tau, \theta]\subset[ \theta_{p-1}, \theta_p]$. Consider  a set $B \subset\{p, \dots P\}$ and observables $( \bar \phi ^{(i)} )_{i \in B }$
supported on single minimal forward cluster at time $\theta$ as in \eqref{eq: cumulant induction}.
 There is a constant~$C_P$ depending only on~$P$  such that for~$\delta \ll \tau=R\delta \ll 1$ the following holds,  as~$\mu_\eps \to \infty$
\begin{equation}
\label{Duhamel-identite-tau}
\begin{aligned}
&\left| \bbE_\eps   \Big[ \Big( \prod_{u=1} ^{p-1}  \zeta^\eps_{\theta_u} \big( h^{(u)}\big)\Big)  \Otimes_{ i \in B }
 \zeta^\eps_{m_i,\theta} \big(  \bar \phi ^{(i)} \big) \Big]\right. \\
  &\qquad \left. - 
\sum_{ \NN   } \sum_{ \eta  \cup \rho \hbox{ \tiny\rm partition of } B\atop \rho \hbox{ \tiny\rm  clustering} }   \prod _{q= 1}^{|\rho|}  \mu_\eps^{1-\frac{|\rho_q |}2} \bbE_\eps\Big[ \bar \phi_{\tau}^{(\rho_q )}  \Big]  \times\,  \bbE_\eps   \Big[\Big( \prod_{u=1} ^{p-1}  \zeta^\eps_{\theta_u} \big( h^{(u)}\big)\Big)\, \Big( \Otimes_{q=1}^{|\eta|}   \mu_\eps^{\frac12 -\frac{|\eta_{q}|}2}   \zeta^\eps_{M_{\eta_q}^\tau,\theta - \tau}  \big(
\bar \phi_{\tau }^{(\eta_{q} )}  \big) \Big)\Big]\right|  \\
& \qquad \qquad \qquad \qquad \qquad \qquad   \leq  (C_P\Theta)^{M_B} 
\Big(\prod_{i \in B\cup\{1,\dots,p-1\}} \| h^{(i)}\|_{L^\infty}\Big)   \eps^{\frac1{8d}} ,
\end{aligned}
\end{equation}
   with notation~\eqref{barphi-def}, and with~$M_{\eta_q}^\tau $ the total number of particles in the packet $\eta_q$ at time $\theta- \tau$.
\end{Prop}
 The proof of this proposition is the content of the next two sections and proceeds in two steps: 
starting from
\begin{equation}
\label{defJM}
\cJ_{\MM}:=\bbE_\eps   \Big[ \Xi_{p-1}\,   \Otimes_{ i\in B} 
 \zeta^\eps_{m_i,\theta } \big(  \bar \phi ^{(i)}  \big) \Big]\, \quad  \hbox{ with the short notation }  \Xi_{p-1}:= \prod_{u=1} ^{p-1}  \zeta^\eps_{\theta_u} \big( h^{(u)}\big) \, ,
\end{equation}
we first extract  iteratively the
 remainder terms (for which the minimality condition is violated), and identify the main part. This procedure produces a sum of~$R: = \frac{\tau}{\delta} $ error terms, which are estimated in Paragraph~\ref{remove recollisions}.

\begin{Prop}
\label{fubini-prop}
The fluctuation structure~{\rm(\ref{defJM})} is transported on $[\theta- \tau, \theta]$ according to 
\begin{equation}
\label{JM-decomposition}
 \cJ_{\MM}=\bar  \cJ_{\MM}  +\sum_{r = 1} ^R \cR^{\rm int}_r\,,
 \end{equation}
where  the principal part is given by 
\begin{align}
\label{eq: cJ MM 0} 
\bar \cJ_{\MM} & :=
  \sum_{  \NN^{\tau}_B} \sum_{ \eta ^R  \cup \rho^R  \hbox{ \tiny\rm partition of } B \atop  \rho^R \hbox{ \tiny\rm  clustering} }  \prod _{q=1}^{|\rho^R|}\mu_\eps^{1 -\frac{|\rho^R_q|}2} 
 \bbE_\eps\Big[ 
  \bar \phi_\tau^{(\rho^R_q)}  \Big]   \\
& \qquad  \qquad  \times
\bbE_\eps   \Big[  
 \Xi_{p-1}  \Big( \Otimes_{q=1}^{|\eta^R|} \mu_\eps^{\frac12-\frac{ |\eta_q^R| }2}   \zeta^\eps_{M^\tau_q ,\theta-\tau} \Big(
 \bar \phi_\tau^{(\eta^R_{q}) } \Big) \Big)\Big]  \nonumber
\end{align}
with notation~\eqref{barphi-def},    the short notation $M_{q}^{r\delta}:=M_{\eta_{q}^{r}}^{r\delta}$ for the total number of particles in $\eta_{q}^{r}$ at time~$\theta- r \delta$, and  
 $\NN^{r\delta}_B = ( n_j )_{j \in B}$  the number of  particles  removed in each block during $[\theta - r \delta, \theta]$. 
The remainders are defined, with notation~\eqref{notbarphi-def}, as the sums of 
\begin{equation}
\label{eq: R int}
\begin{aligned}
\cR^{\rm int, 1}_r &= 
 \sum_{ \NN^{r\delta}_B }  \sum_{ \eta  ^r  \cup \rho^r  \hbox{ \tiny\rm partition of } B \atop  \rho^r \hbox{ \tiny\rm  clustering} }
 \prod _{q=1}^{|\rho^r|} \mu_\eps^{1-\frac{|\rho_q^r|  }2}  \bbE_\eps\Big[ 
  \bar \phi_{r \delta}^{(\rho^r_q )}  \Big]  \\
&  \qquad \times
\bbE_\eps   \left[  \Xi_{p-1}    \left( \Otimes_{q=1}^{|\eta^r|}  \mu_\eps^{\frac12-\frac{ |\eta^r_{q}| }2}    \zeta^\eps_{ M^{r\delta}_q    ,\theta-r\delta} \Big(  \phi_{r\delta}^{(\eta^r_{q})} \Big)  - \Otimes_{q=1}^{|\eta^r|} \mu_\eps^{\frac12-\frac{ |\eta^r_{q}| }2}    \zeta^\eps_{ M^{r\delta}_q    ,\theta-r\delta} \Big( 
  \bar  \phi_{r\delta}^{(\eta^r_{q})}  \Big) \right)\right] , \\
\cR^{\rm int, 2} _r&= 
 \sum_{ \NN^{r\delta}_B } \sum_{ \eta ^r  \cup \rho^r   \hbox{ \tiny\rm partition of } B \atop  \rho^r \hbox{ \tiny\rm  clustering} }   \left(  \prod _{q=1}^{|\rho^r|} \mu_\eps^{1-\frac{|\rho_q^r|  }2}  \bbE_\eps\Big[   
  \phi_{r\delta}^{(\rho_q^r)}  \Big]  - \prod _{q=1}^{|\rho^r|}  \mu_\eps^{1-\frac{|\rho_q^r|  }2} \bbE_\eps\Big[  
   \bar \phi_{r\delta}^{(\rho_q^r)}  \Big]  \right) \\
&   \qquad   \times 
\bbE_\eps   \Big[  \Xi_{p-1}    
 \Big( \Otimes_{q=1}^{|\eta^r|}    \mu_\eps^{\frac12-\frac{ |\eta^r_{q}| }2}   \zeta^\eps_{ M^{r\delta}_q    ,\theta-r\delta} 
 \Big(  \phi_{r\delta}^{(\eta^r_{q})}  \Big) \Big)\Big] \, .
\end{aligned}
\end{equation}
\end{Prop}

\begin{proof}
Proposition \ref{fubini-prop} is proved recursively. 
More precisely we start with  a decomposition of~$ \cJ_{\MM}$ at step~$r-1$ under the form $$
 \cJ_{\MM}=\bar  \cJ_{\MM}^{r-1} + \sum_{r' = 1} ^{r-1}  \cR^{\rm int}_{r'}\,,
$$
where the main term is defined by
$$\begin{aligned}
\bar  \cJ_{\MM}^{r-1} &:=  \sum_{ \NN^{(r-1) \delta}_B }  \sum_{ \eta ^{ r -1 }  \cup \rho^{ r -1 } \hbox{ \tiny\rm partition of } B \atop  \rho^{ r -1 } \hbox{ \tiny\rm  clustering} }    \prod _{q}\mu_\eps^{1-\frac{|\rho_q^{r-1}|  }2}  \bbE_\eps\Big[   
  \bar \phi_{(r-1)\delta} ^{(\rho_q^{r-1})}  \Big]   \\
& \qquad  \qquad  \times
\bbE_\eps   \Big[
\Xi_{p-1}  \Big( \Otimes_{q}   \mu_\eps^{\frac12-\frac{|\eta_{q}^{r-1}| }2}  \zeta^\eps_{M_{q}^{(r-1)\delta} ,\theta-(r-1)\delta} \Big(
\bar \phi_{(r-1)\delta}^{(\eta^{r-1}_{q}) } \Big) \Big)\Big] .
\end{aligned}
$$
For each~$q$,
we pull back the cumulants~$ \bar \phi_{(r-1)\delta} ^{(\rho_q^{r-1})}$
 and~$\bar \phi_{(r-1)\delta}^{(\eta^{r-1}_{q}) }$   on a time step~$\delta$
 and extract minimal clusters and remainder terms. A Fubini equality will enable us to   replace    the nested partitions by one partition~$ \eta ^{ r }  \cup \rho^{ r   } $ thus completing the induction.
   
\medskip

We first consider the expelled cumulants.
As in (\ref{eq: 1 site}), we have 
$$
 \bbE_\eps\Big[ \bar \phi_{(r-1)\delta}^{(\rho_q^{r-1} )}  \Big]
 = 
 \sum_{\NN^r } 
\bbE_\eps   \Big[   \phi_{r\delta}^{(\rho_q^{r-1 }), {\rm t}}  \Big]  \, ,
$$
 where the superscript $\rm t$ in $\phi_{r\delta}^{(\rho_q^{r-1} ), {\rm t}}$ stands for cumulants transported by the forward dynamics  in the sense of Definition~\ref{def forward dynamics}, without any clustering  in the last time interval $[\theta - r \delta, \theta - (r-1) \delta]$. By construction
\begin{equation}
\label{exp-factor}
\phi_{r\delta}^{(\rho_q^{r-1} ), {\rm t}}
:=    \mu_\eps^{N^r_q } \left( \prod_{j \in \rho_q^{r-1} } \frac  { M_j^{(r-1)\delta} !  }{  M_j^{r\delta}  !}  \right)
\; \sum_{\KK , \SS^r,\bar \SS^r} \sgn (\SS^r  )   
 \Big({\sharp}_\delta   \bar \phi_{(r-1) \delta} ^{ (\rho_q^{r-1})} \Big) \indc_{ \{Z_{\M_q^{r\delta} }    \in \cR^{\rho_q^{r-1}}_{{\mathbf K},\SS ,\bar \SS}\}  } \,,
 \end{equation}
 denoting for simplicity by $N_q^r$ the number of particles annihilated in  $\rho_q^{r-1}$ on $[\theta- r \delta, \theta- (r-1) \delta]$ so that the number of particles in~$\rho_q^{r-1}$ at~$\theta- r \delta$ is~$M_q^{r\delta} = M^{(r-1)\delta}_{\rho_q^{r-1}} +N_q^r$.  
%

\medskip

We turn now to the product of the fluctuation fields. The packets at time $\theta- (r-1)\delta$ are prescribed by 
$\eta^{r-1}$, and we will temporarily keep track of this fact by an additional superscript $\eta^{r-1} $ in the pulled back observables~:
$$\begin{aligned}
\label{fluct-factor}
&\bbE_\eps   \Big[  \Xi_{p-1}\, \Big( \Otimes_{q}   \mu_\eps^{\frac12 -\frac{|\eta^{r-1}_{q}|}2}   \zeta^\eps_{M_{q}^{(r-1)\delta} ,\theta - (r-1)\delta}  \big(
  \bar \phi_{(r-1)\delta }^{(\eta^{r-1}_{q})}  \big) \Big)\Big] \\
  & \  =
\sum_{\NN^r} \sum_{\eta^{r-1} \hookrightarrow \eta_1^r \cup \eta_2^r  \atop \eta_2^r {\rm clustering}  } 
  \prod _{q}\mu_\eps^{1-\frac{|\eta^r_{2,q}|}2} \bbE_\eps\Big[  
   \phi_{r\delta}^{(\eta^r_{2,q}), \eta^{r-1}  }  \Big]    \times 
\bbE_\eps   \Big[ \Xi_{p-1}    \Big( \Otimes_{q}   \mu_\eps^{\frac12 -\frac{|\eta^r_{1,q}|}2}  \zeta^\eps_{M_{q}^{r\delta},\theta-r\delta} \Big( 
\phi_{r\delta} ^{(\eta_{1,q}^r), \eta^{r-1}  }  \Big) \Big)\Big]   \, ,
\end{aligned}
$$
where $\eta_1^r\cup \eta_2^r$ is a coarser partition than $\eta^{r-1}$, and each part of $\eta_2^r$ contains at least two components of $\eta^{r-1}$: this term only appears if~$|\eta^{r-1}| \geq 2$. 

Our goal is to sum over all~$\eta^{r-1}$ compatible with a given coarser decomposition $\eta^r$, 
in order to retrieve cumulants at step $r$ (see Figure \ref{fig:fubini nested}).
Let us first introduce
\begin{equation}
\label{notbarphi-def-eta}
\begin{aligned}
 \phi_{r \delta}^{(\eta_{1,q}^r)}
& :=  \Big(  \prod _{j \in \eta_{1,q}^r}  \frac  { M_j^{(r-1)\delta} !  }{  M_j^{r\delta} !} \Big) \sum_{ \varsigma ,\lambda\atop \varsigma \hookrightarrow\lambda \hookrightarrow \eta_{1,q}^r}  \mu_\eps^{N^r + |\varsigma  |-1} \sum_{\KK , \SS^r, \bar\SS^r} \sgn (\SS^r  )   
   \varphi_{\{\lambda_1, \dots, \lambda_{|\lambda|}\}} \\
   &\qquad \qquad\qquad \times \prod_{\ell =1}^{|\lambda|}
 \Big({\sharp}_\delta  \bigotimes_{\varsigma _j \subset \lambda_\ell} \bar \phi_{(r-1) \delta} ^{ (\varsigma _j)} \Big) 
 \indc_{ \{Z_{\M_{\lambda_\ell}^{r\delta} }    \in \cR^{\lambda_\ell}_{{\mathbf K}_{\lambda_\ell},\SS_{\lambda_\ell}} \} }  \\
 \phi_{r\delta} ^{(\eta_{2,q}^r),c} 
 & := \Big(  \prod _{j \in \eta_{2,q}^r}  \frac  { M_j^{(r-1)\delta} !  }{  M_j^{r\delta} !} \Big) \sum_{ \varsigma ,\lambda\atop{ \varsigma \hookrightarrow\lambda \hookrightarrow \eta_{2,q}^r\atop |\varsigma| \geq 2}}  \mu_\eps^{N^r + |\varsigma  |-1}  \sum_{\KK , \SS^r, \bar\SS^r}  \sgn (\SS^r  )   
   \varphi_{\{\lambda_1, \dots, \lambda_{|\lambda|}\}} \\
   & \qquad\qquad\qquad
\times\prod_{\ell =1}^{|\lambda|}\Big({\sharp}_\delta  \bigotimes_{\varsigma _j \subset \lambda_\ell} \bar \phi_{(r-1) \delta} ^{ (\varsigma _j)} \Big) \indc_{ \{Z_{\M_{\lambda_\ell}^{r\delta} }    \in \cR^{\lambda_\ell}_{{\mathbf K}_{\lambda_\ell},\SS_{\lambda_\ell}} \} }  
 \end{aligned}
\end{equation}
denoting by $N^r$ the total number of particles annihilated  on~$[\theta- r \delta, \theta-(r-1) \delta]$ in the  packets.
Observe that, with respect to Eq.\,\eqref{defphideltaetai}, this definition contains an additional sum over  previous packets $\varsigma$ at time $\theta- (r-1) \delta$.
Now we can apply Fubini's equality to sum over~$\eta^{r-1}$ thanks to the above definitions. We define indeed~$\rho^r := \rho^{r-1} \cup \eta_2^r$ and~$\eta^r := \eta_1^{r}$. The sets~$ \rho^{r-1}$ and~$\eta_2^r$   play symmetric roles  and the cumulants~$\phi_{r\delta}^{(\rho_q^{r-1} ), {\rm t}}
$ and~$ \phi_{r\delta} ^{(\eta_{2,q}^r),c} $ can be combined into one cumulant~$
 \phi_{r\delta}^{(\rho_q^r)} $. Now we    sum~(\ref{fluct-factor}) over~$ \eta^{r-1}$ and use~(\ref{notbarphi-def-eta}) to recover
 $$
 \begin{aligned}
 \bar \cJ_{\MM}^{r-1} & = \sum_{ \NN^{r }_{B} }   \sum_{\eta  ^{ r  }  \cup \rho^{ r   } \hbox{ \tiny\rm partition of } B \atop  \rho^{ r  } \hbox{ \tiny\rm  clustering} }  \left(  \prod _{q}  \mu_\eps^{1 -\frac{|\rho^r_q | }2} 
\bbE_\eps\Big[ 
 \phi_{r\delta}^{(\rho_q^r)}  \Big]  \right)   \times  \bbE_\eps   \Big[  \Xi_{p-1}     \Big( \Otimes_{q}  \mu_\eps^{\frac12-\frac{|\eta^r_{q}| }2}   \zeta^\eps_{ M^{r\delta}_{\eta^r_{q}}    ,\theta-r\delta} \Big(
  \phi_{r\delta}^{(\eta^r_{q})}  \Big) \Big)\Big] \,.
\end{aligned}
$$

We finally extract from  the main term  the minimal cumulants.  We  restrict the support of the cumulants~$ \phi_{r\delta}^{(\sigma)}$  (with~$\sigma = \eta^r_{q}$ or~$\rho^r_q$) to configurations such that the graph recording all the encounters on the time interval~$[\theta-r\delta,\theta-(r-1)\delta]$  has no cycle (nor multiple edge). In particular  there is no internal encounter without annihilation inside the blocks so that the sum  over~$\KK $   in~(\ref{notbarphi-def-eta}) disappears, and by definition, all ``admissible'' partitions~$\varsigma = (\varsigma_i )_{1 \leq i \leq |\varsigma|}$ of~$\sigma$ into  packets are characterized by the graph recording the encounters between  the~$\varsigma_i$ which must be   minimally connected. 
  This local minimal dynamics in~$[\theta-r\delta,\theta-(r-1)\delta]$
  is then entirely prescribed by the signs~$\SS^r,\bar  \SS^r $ encoding encounters between particles of the same tree on the time interval, and  signs~$\EE^r  \in \{-1, +1\} ^{|\varsigma| -1} $. 
  \begin{figure}[h] 
\includegraphics[width=5in]{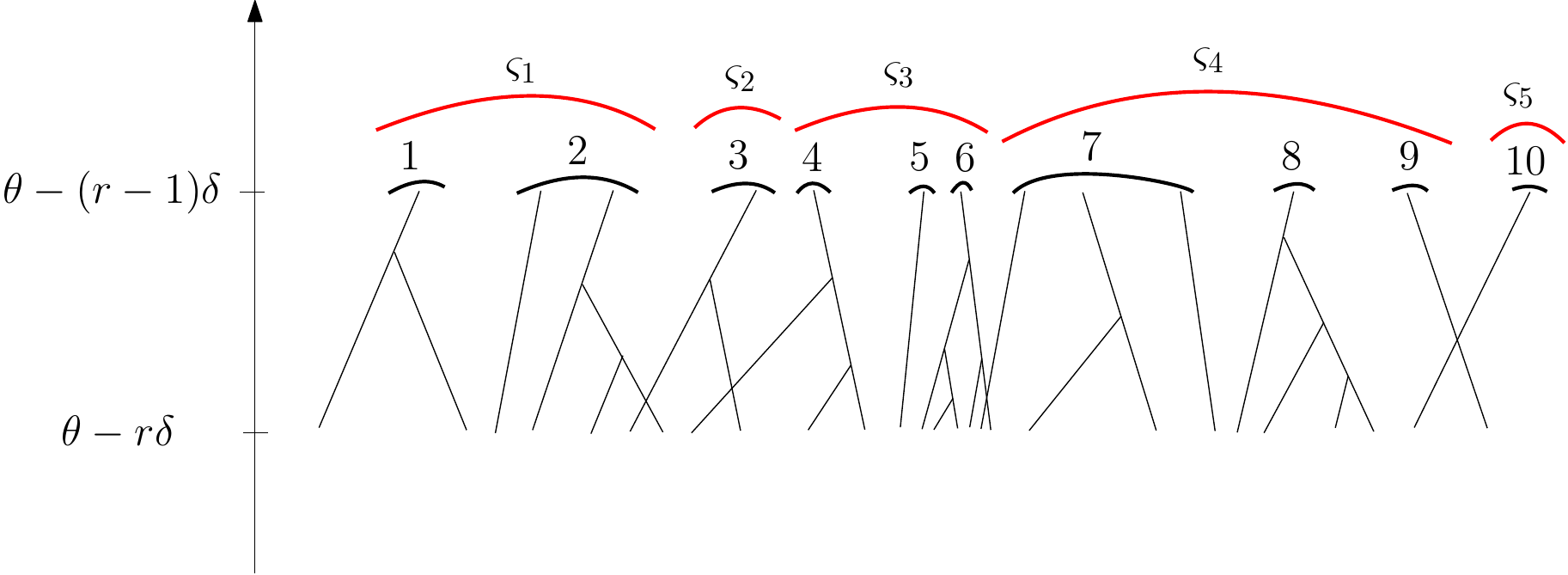}  
\caption{ A schematic picture of generic contribution to the minimal cumulant $\bar \phi_{r \delta}^{(\sigma) }$, in the case of a $\sigma$ consisting of $10$ blocks (for simplicity $\sigma=\{1,2,\cdots,10\}$) grouped into  $s=5$ packets. An  admissible partition $\varsigma$ is  represented  in red~: it  imposes some dynamical constraints after $\theta- (r-1) \delta$. The dynamics in the depicted time interval is prescribed by $ \EE^r  , \SS^r ,\bar  \SS^r$. Note that two blocks of the same packet $\varsigma_i$ should not collide or overlap on $[\theta- r \delta, \theta-(r-1) \delta]$.
}
 \label{fig:fubini nested}
\end{figure}

For such an admissible $\varsigma$, there is a unique partition $\lambda $ in forests and   (\ref{varphibar}) holds.
Thus  the restriction of~$\phi_{r \delta}^{(\sigma) }$  to minimal local forward dynamics is given by
\begin{equation}
\label{barphi-defloc}
\begin{aligned}
\bar \phi_{r \delta}^{(\sigma) } = 
 & \sum_{s \leq |\sigma| }\mu_\eps^{N^r_\sigma +s-1}  \sum_{ \EE^r  , \SS^r ,\bar  \SS^r  }  
 \sgn(\EE^r) \, \sgn( \SS^r ) \\
&\qquad\qquad\qquad\qquad   \times
 \sum_{ \hbox{\tiny \rm admissible $\varsigma$ }  \atop |\varsigma|=s} \indc_{Z_{\MM_{\sigma}^{r\delta}} \in \cR^{\rm loc \, min, \varsigma}_{\EE^r , \SS^r,\bar  \SS^r} }  \prod_{i=1} ^s  \prod_{j \in \varsigma_i } {M_{j}^{(r-1)\delta}  ! \over  M_{j} ^{r\delta} !}   ({\sharp}_\delta   \bar \phi_{(r-1)\delta}  ^{ (\varsigma_i)}  )  \,,
\end{aligned}
\end{equation}
where~$\cR^{\rm loc \, min, \varsigma}_{\EE^r  , \SS ^r,\bar  \SS^r}$ is the set of all configurations leading to minimal forward dynamics on the time interval~$[\theta-r\delta, \theta-(r-1)\delta]$ compatible with~$\EE^r \in 
\{ -1,1\}^{s-1}$, $  \SS^r$, $\bar \SS^r $ and $\varsigma$. 
 
  Now by the induction assumption,~$\bar \phi_{(r-1)\delta}  ^{( \varsigma_i)}$ is a minimal cumulant given by  formula~(\ref{barphi-def})  at time~$\theta- (r-1)\delta$ 
$$
\begin{aligned}
 \bar \phi_{(r-1)\delta}  ^{ (\varsigma_i) } (Z_{M_{\varsigma_i}^{(r-1)\delta}}) 
 = & \mu_\eps^{N^{<r} _{\varsigma_i}  +|\varsigma_i|-1  }  
  \left( \prod_{j \in \varsigma_i}{ m_j! \over   M^{{(r-1)\delta}}_j !} \right) \\
& \sum_{\bar \SS_i^{<r}  ,\SS_i^{<r}  , \EE_i^{<r}   }  \sgn(\EE_i^{<r}) \sgn(\SS_i^{<r}  ) \indc_{ Z_{M_{\varsigma_i}^{(r-1)\delta}} \; \in \; \cR^{\rm min}_{\EE_i^{<r}  , \SS_i^{<r} ,\bar \SS_i^{<r}  } }   
\big( {\sharp}_{(r-1) \delta} \bigotimes _{j \in \varsigma_i }  \; \phi^{(j)}   \big)  \,,
\end{aligned}
$$
where~$N^{<r} _{\varsigma_i}  $ denotes the number of particles annihilated on the full time interval~$[\theta - (r-1) \delta, \theta]$ and~$\EE_i^{<r}  ,\SS_i^{<r} ,\bar\SS_i^{<r} $ is the collection of all signs on that same time interval.
Plugging this formula into~(\ref{barphi-defloc}) (and using again Fubini's identity to pass from the parametrization in terms of $\left(\EE^r  , \SS^r ,\bar  \SS^r , \left(\bar \SS_i^{<r}  ,\SS_i^{<r}  , \EE_i^{<r} \right)_i\right)$ to a global parametrization $\left(\EE , \SS, \bar \SS\right)$), we recover formula~(\ref{barphi-def}) at time~$\theta- r\delta$.
 Thus we obtain the expected decomposition at $\theta-r \delta$:
$$\begin{aligned}
 \bar \cJ_{\MM}^{r-1} & = \sum_{ \NN^{r }_{B} }  \sum_{\eta  ^{ r  }  \cup \rho^{ r   } \hbox{ \tiny \rm partition of }  B\atop  \rho^{ r  } \hbox{ \tiny\rm  clustering} }  \left(  \prod _{q} \mu_\eps^{1 -\frac{|\rho^r_q | }2} 
\bbE_\eps\Big[  
 \bar \phi_{r\delta}^{(\rho_q^r)}  \Big]  \right) \bbE_\eps   \Big[  \Xi_{p-1}     \Big( \Otimes_{q}    \mu_\eps^{\frac12 - \frac{|\eta^r_{q}|  }2} 
 \zeta^\eps_{ M^{r\delta}_{\eta^r_{q}}    ,\theta-r\delta} \Big(
\bar  \phi_{r\delta}^{(\eta^r_{q})}  \Big) \Big)\Big]   +\cR^{{\rm int}}_r  \,.
\end{aligned}
$$
Proposition~\ref{fubini-prop}
 is proved. \end{proof}

\subsection{Estimates of the remainders}
\label{remove recollisions}

We   now establish the following estimates for the remainders, which thanks to Proposition~\ref{fubini-prop} imply Proposition~\ref{prop-from theta to theta + tau} as an immediate corollary. 
\begin{Prop}
\label{remove recollisions on delta}
Under the assumptions of Proposition {\rm\ref{prop-from theta to theta + tau}}, 
there is a constant~$C_P$ depending only  on $P$
such that the remainder $ \cR ^{{\rm int}}_r $   defined in~{\rm(\ref{eq: R int})} satisfies the following estimate~:
$$
\begin{aligned}
\left| \sum_{r = 1} ^R \cR ^{{\rm int}}_r  \right|& \leq (C_P\Theta)^{M_B}   \Big(\prod_{i \in B \cup\{1,\dots,p-1\}} \| h^{(i)}\|_{L^\infty}\Big)   \eps^{\frac1{8d}} \,,
\end{aligned}
$$
with $R = \tau/\delta$ and $M_B = \sum_{i \in B} m_i  $.
\end{Prop}

In our argument, the specific form of the function $\Xi_{p-1}$ (see \eqref{defJM}) will be irrelevant. 
Only the following  two features are needed :
\begin{itemize}
\item $ \Xi_{p-1}$ depends  on the particle configurations before (and at) time $\theta_{p-1}$;
\item $ \Xi_{p-1}$ has a uniformly bounded variance.
\end{itemize}
Notice indeed that  by H\"older's inequality, 
 \begin{equation}
\label{eq: borne L2 Xi}
 \bbE_\eps [  \Xi_{p-1}^2 ] = \bbE_\eps   \Big[
\Big( \prod_{u=1} ^{p-1}  \zeta^\eps_{\theta_u} \big( h^{(u)}\big) \Big)^2    \Big]
\leq 
\prod_{u=1} ^{p-1}  \bbE_\eps   \Big[   \zeta^\eps_{\theta_u} \big( h^{(u)}\big)^{2 (p-1)} 
\Big]^{\frac{1}{ (p-1)} }
\leq C_p \prod_{u=1} ^{p-1} \| h^{(u)} \|_\infty^2\, ,
\end{equation}
and thus~$ \Xi_{p-1}$ is bounded in $L^2$, by Proposition \ref{Proposition - estimates on g0}  which ensures that the fluctuation fields  are bounded in $L^{2 (p-1)}$. 

The key  ingredient will be the following lemma controlling the expectation and variance of cumulants  based on their clustering structure. It will be   proved  in Sections \ref{geometric-sec} and~\ref{variance-sec}: see Section~\ref{geometric-sec} for the proof of~(\ref{expectation-phi-sigma}), Paragraph~\ref{sec: Expectation of centered Otimes-products} for the proof of \eqref{E-HM 1}  and Paragraph~\ref{variance} for the proof of (\ref{variance-phi-sigma}).

\begin{Lem}
\label{phi-sigma-lemma} Let $B \subset \{p,p+1,\dots,P\}$.
Consider observables $( \bar \phi ^{(i)} )_{i \in B}$,  supported on single minimal forward clusters as in \eqref{eq: cumulant induction} at time  $\theta = \theta_p - k \tau$.
Let $\sigma = (\sigma_i)_{i \leq |\sigma|}$ be a collection  of   packets in $B$, and the corresponding cumulants $ \phi_{r\delta}^{(\sigma_i) }$ defined by  {\rm(\ref{notbarphi-def})}. Denote by~$M_{\sigma_i}$ the number of particles in~$\sigma_i $ at time~$\theta$, by  $N_{\sigma_i }^r$ the number of particles to be removed  in the last time interval $[\theta- r \delta, \theta-(r-1) \delta]$, by $N_{\sigma_i}^{ <r}$ the number of other particles to be removed    in $[\theta- (r-1) \delta, \theta]$ and by $M_{\sigma_i}^{r\delta} = M_{\sigma_i} + N_{\sigma_i}^{r}+N_{\sigma_i}^{ <r}  $ the total number of particles in $\sigma_i$ at time $\theta- r\delta$.
Then there is constant $C_P>0$ depending  on $P$ and $K_\gamma$ such that 
such that
\begin{equation}
\label{expectation-phi-sigma}
 \bbE_\eps\Big[    \Big| \phi_{r\delta}^{(\sigma_i) }\Big|  \Big] \leq  C_P  \Big(  \prod_{ j \in \sigma_i } \| h^{(j)}\|_{L^\infty}  \Big)
  \;  (C_P\Theta)^{M_{\sigma_i} - |\sigma_i|}  (C_P \delta)^{ N_{\sigma_i }^{r}}  (C_P \tau)^{N_{\sigma_i}^{<r}+|\sigma_i | - 1}  \, ,
\end{equation}
 and  
\begin{equation}
 \begin{aligned}
\label{E-HM 1}  
\left|\bbE_\eps  \Big[  \Otimes_{i\leq |\sigma|}  \zeta^\eps_{M_{\sigma_i}^{r\delta}, \theta-r \delta  } 
\big( \phi_{r\delta}^{(\sigma_i) } \big)  \Big]  \right|
&\leq C_P \eps \prod_{i\leq |\sigma|}  \Big(  \prod_{ j \in \sigma_i } \| h^{(j)}\|_{L^\infty}  \Big)
\\
  &\quad \times  (C_P \Theta)^{M_{\sigma_i} - |\sigma_i|}  \Big( (C_P \delta)^{ N_{\sigma_i }^{r}}  (C_P \tau)^{N_{\sigma_i}^{<r}+|\sigma_i | - 1} \Big)^{1/2} \;,
   \end{aligned}
\end{equation}
as well as 
\begin{equation}
\label{variance-phi-sigma}
\begin{aligned}
\bbE_\eps\Big[ \Big( \Otimes_{i\leq |\sigma| }  \zeta^\eps_{M_{\sigma_i}^{r\delta}, \theta-r \delta  } 
\big(  \phi_{r\delta}^{(\sigma_i) } \big) \Big)^2 \Big] 
&\leq C_P   \prod_{i\leq |\sigma|}  \Big(  \prod_{ j \in \sigma_i } \| h^{(j)}\|_{L^\infty}  \Big)^2
 \\
 &\qquad  \times \Big( (C_P \Theta)^{ 2M_{\sigma_i}  + N_{\sigma_i }^r + N_{\sigma_i }^{<r} - |\sigma_i|  }  (C_P \delta)^{ N_{\sigma_i}^{r }}  (C_P \tau)^{N_{\sigma_i }^{<r}+|\sigma_i | - 1} \Big) \, .
 \end{aligned}
\end{equation}
The same estimates hold  for  the minimal cumulants $\bar \phi_{r\delta}^{(\sigma_i) }$ defined by {\rm(\ref{barphi-def})}.
\end{Lem}

Recall that  the remainders $\cR ^{{\rm int}}_r$ in~(\ref{eq: R int}) are due to the non minimal dynamics. The smallness will come from the fact that at least one factor in the product of expectations, or one fluctuation field involves a dynamical graph with a cycle (or multiple dynamical edge)  
\begin{equation}
\label{phicyc-def}
\phi_{r\delta}^{(\sigma_i ), {\rm cyc} } 
:= \phi_{r\delta}^{(\sigma_i )} - \bar  \phi_{r\delta}^{(\sigma_i )}\,.
\end{equation}
We will therefore also need the following lemma, proved in Section \ref{geometric-sec} for~(\ref{expectation-phi-rec}) and Section~\ref{variance}  for~(\ref{variance-phi-rec}). 

\begin{Lem}
\label{phi-rec-lemma}
Consider observables $( \bar \phi ^{(i)} )_{i \in B}$ at time $\theta$ supported on single minimal forward cluster as in \eqref{eq: cumulant induction}.
Define  $ \phi_{r\delta}^{(\sigma_i ), {\rm cyc}}$ by {\rm(\ref{phicyc-def})}. 
Then, with the   notations of Lemma {\rm\ref{phi-sigma-lemma}}, we have that for any $i\leq |\sigma|$
\begin{equation}
\label{expectation-phi-rec}
\begin{aligned}
& \bbE_\eps\Big[   \Big|  
 \phi_{r\delta}^{(\sigma_i ), {\rm cyc}}\Big|  \Big]\leq   C_P \Big(  \prod_{ j \in \sigma_i } \| h^{(j)}\|_{L^\infty}  \Big)
\\
&\quad \times  \eps\delta  |\log \eps| (|\log \eps| \Theta) ^{ 2d+4}    (C_P \Theta)^{M_{\sigma_i } -|\sigma_i|} (C_P \delta)^{ (N_{\sigma_i  }^{r}-1)_+}  (C_P \tau)^{(N_{\sigma_i}^{ <r}+|\sigma_i  | - 2)_+},
\end{aligned}
\end{equation}
 and  
\begin{equation}
\label{variance-phi-rec}
\begin{aligned}
&\bbE_\eps\Big[  \Big( \zeta^\eps_{M_{\sigma_i}^{r\delta} , \theta-r \delta  } 
\big(
 \phi_{r\delta}^{(\sigma_i ), {\rm cyc} }  \big) \Otimes \Big( \Otimes_{j \neq i } \zeta^\eps_{M_{\sigma_j}^{r\delta}, \theta-r \delta  }
\big(
\phi_{r\delta}^{(\sigma_j) } \big)\Big)  \Big)^2 \Big] \\
& \leq C_P  \prod_{i\leq |\sigma|}  \Big(  \prod_{ j \in \sigma_i } \| h^{(j)}\|_{L^\infty}  \Big)^2
\eps\delta  | \log \eps | (|\log \eps| \Theta) ^{ 2d+4}  (C_P \Theta)^{ M_{\sigma_i }+M_{\sigma_i}^{r\delta}   - |\sigma_i |} (C_P \tau)^{(N_{\sigma_i }^{<r}+|\sigma_i  | - 2)_+}\\
&\quad   \times  (C_P \delta)^{( N_{\sigma_i  }^{r}-1)_+}  \prod_{j \neq i}   \Big( (C _P\Theta)^{2M_{\sigma_j}  + N_{\sigma_j }^r + N_{\sigma_j }^{<r}  - |\sigma_j |}  (C_P \delta)^{  N_{\sigma_j}^{r} }  (C_P \tau)^{N_{\sigma_j}^{<r}+|\sigma_j | - 1)}\Big)  \, .
\end{aligned}
\end{equation}
The same estimates hold  when replacing the cumulants $ \phi_{r\delta}^{(\sigma_j) }$ by   the minimal cumulants $\bar \phi_{r\delta}^{(\sigma_j) }$ defined by {\rm(\ref{barphi-def})}.
\end{Lem}

\medskip

\begin{proof}[Proof of Proposition~{\rm\ref{remove recollisions on delta}}]

Using the homogeneity, we can assume without loss of generality  that $\| h^{(i)}\|_{L^\infty} \leq 1$  so  we do not keep track of~$ \| h^{(i)}\|_{L^\infty} $ in the estimates.
Recall  the definition~(\ref{eq: R int}) 
of $\cR ^{{\rm int}}_r$. The fluctuation terms in $\cR ^{{\rm int}}_r$ can be decoupled from the function~$ \Xi_{p-1}$ by using the Cauchy-Schwarz estimate, leading in the case of~$\cR ^{{\rm int, 1}}_r$ to
$$
\begin{aligned}
\big| \cR^{\rm int, 1}_r \big| &\leq 
 \sum_{ \NN^{r\delta}_B }  \sum_{ \eta  ^r  \cup \rho^r  \hbox{ \tiny\rm partition of } B \atop  \rho^r \hbox{ \tiny\rm  clustering} }
 \prod _{q=1}^{|\rho^r|} \mu_\eps^{1-\frac{|\rho_q^r|  }2}  \bbE_\eps\Big[ 
  \bar \phi_{r \delta}^{(\rho^r_q )}  \Big]  \bbE_\eps ^\frac12[  \Xi_{p-1}^2 ]
 \\
&  \qquad \times\bbE_\eps^\frac12   \left[   \left( \Otimes_{q=1}^{|\eta^r|}  \mu_\eps^{\frac12-\frac{ |\eta^r_{q}| }2}    \zeta^\eps_{ M^{r\delta}_q    ,\theta-r\delta} \Big(  \phi_{r\delta}^{(\eta^r_{q})} \Big)  - \Otimes_{q=1}^{|\eta^r|} \mu_\eps^{\frac12-\frac{ |\eta^r_{q}| }2}    \zeta^\eps_{ M^{r\delta}_q    ,\theta-r\delta} \Big( 
  \bar  \phi_{r\delta}^{(\eta^r_{q})}  \Big) \right)^2\right] \\
   &\leq C_P
 \sum_{ \NN^{r\delta}_B }  \sum_{ \eta  ^r  \cup \rho^r  \hbox{ \tiny\rm partition of } B \atop  \rho^r \hbox{ \tiny\rm  clustering} }
 \prod _{q=1}^{|\rho^r|} \bbE_\eps\Big[ 
  \bar \phi_{r \delta}^{(\rho^r_q )}  \Big] \bbE_\eps^\frac12   \left[   \left( \Otimes_{q=1}^{|\eta^r|}   \zeta^\eps_{ M^{r\delta}_q    ,\theta-r\delta} \Big(  \phi_{r\delta}^{(\eta^r_{q})} \Big)  - \Otimes_{q=1}^{|\eta^r|}   \zeta^\eps_{ M^{r\delta}_q    ,\theta-r\delta} \Big( 
  \bar  \phi_{r\delta}^{(\eta^r_{q})}  \Big) \right)^2\right ]\end{aligned}$$
    thanks to (\ref{eq: borne L2 Xi}) and the fact that $\mu_\eps^{ 1- |\eta_q^r|} \leq 1$.
    Then, since there is at least one factor  $\phi_{r\delta}^{(\eta_q^r  ), {\rm cyc} } 
= \phi_{r\delta}^{(\eta_q^r )} - \bar  \phi_{r\delta}^{(\eta_q^r )}$, using   (\ref{expectation-phi-sigma})  and~(\ref{variance-phi-rec})  we find that~$\cR^{\rm int, 1}_r$
     is bounded by
$$
\begin{aligned}
&  C_P  (\eps\delta|\log \eps|  )^{1/2} ( |\log \eps| \Theta) ^{d+2} \; (C_P  \delta)^{N_{\rho^r}^{r} }  (C_P   \tau)^{N_{\rho^r}^{ <r }+\sum_i|\rho^r_i | - |\rho^r| }(C_P \Theta)^{M_{\rho^r}-\sum_i|\rho^r_i |}  \\
&\quad \times (C_P  \Theta)^{(M_{\eta ^r}-\sum_i|\eta_{i}^r | +M_{\eta ^r}^{r\delta})/2}  (C_P \delta)^{ (N_{\eta ^r}^{r}-1)_+ /2}  (C_P  \tau)^{(N_{\eta ^r}^{ <r}+\sum_i|\eta_{i}^r |  - |\eta|-1)_+/2}\, .
\end{aligned}
$$
The reasoning is similar for~$\cR^{\rm int, 2}_r$, using  (\ref{variance-phi-sigma}) and  (\ref{expectation-phi-rec}). Summing over  $(\NN^{r'}) _{r' \leq r}$, then over $r \leq R= \frac{\tau}{\delta}$, we get 
\begin{equation}
\label{Rint-est}
 \Big| \sum_{r=1}^R \cR_r ^{{\rm int} } \Big| \leq   (C_P \Theta)^{M_B}  (\eps\delta|\log \eps| )^{1/2} ( |\log \eps| \Theta) ^{d+2} {\tau\over \delta}\leq    (C_P \Theta)^{M_B} \eps^{\frac1{8d}} \,,
 \end{equation}
with  the choice $\delta = \e^{1 - \frac{1}{2d}}$ in \eqref{eq: choix parametres}  and~$\tau$ satisfying (\ref{tau-conditions}).
This concludes the proof of Proposition~\ref{remove recollisions on delta}, and thus of Proposition~\ref{prop-from theta to theta + tau}.

Notice that the  choice of the parameter $\delta$ is an optimisation   between the fact that~$\delta$ has to be small so that $\Upsilon^\eps_N$ is a typical event and the necessity for $\delta$ to be larger than $\eps$ for the  estimate (\ref{Rint-est})  to converge to 0. 
\end{proof}


\section{Almost-preserving of the fluctuation structure}
\label{fluctuation-sec}

\subsection{Subexponential  clusters}
\label{subsec: Subexponential forward clusters}

 In Proposition \ref{prop-from theta to theta + tau}, we proved that  the fluctuation structure at a time $\theta$ can be pulled back  to time $\theta - \tau$ up to small error terms. We now want to iterate this formula to pull back the fluctuation on any macroscopic  time interval~$[\theta_{p-1}, \theta_p]$ ($2\leq p \leq P$).

For this, we choose the parameter~$\tau$ so that for all~$i \in [1,P], (\theta_i-\theta_{i-1})/\tau$ is not an integer.
Each time interval~$\big[\theta_{i-1},\theta_i\big]$ is cut  into~$k_i= [(\theta_i-\theta_{i-1})/\tau]+1$  slices (of size~$\tau$,  except for the last slice~$\big[\theta_{i-1},\theta_i  - (k_i-1)  \tau\big]$ which is smaller due to  this assumption on~$\tau$). This leads to a decomposition of~$[0,\Theta]$ into
$
K_P\displaystyle  :=\sum_{i = 2} ^P k_i  $
 slices, denoted~$I_\ell= [\tau_{\ell+1}, \tau_\ell]$ (in decreasing order): thus~$I_1:= (\Theta-\tau, \Theta),  I_2 = (\Theta-2\tau, \Theta-\tau) \dots$. 

In particular, we introduce the decreasing sequence of integers $\{ \kappa_p \}_{p \leq P}$ 
such that 
\begin{equation}
\label{eq: kappa}
\theta_p = \tau_{\kappa_p}
\quad \text{and  we set} \quad 
\LL_p = \{ \kappa_p,\dots, \kappa_{p-1} - 1\}.
\end{equation}

\begin{figure}[h] 
\includegraphics[width=4in]{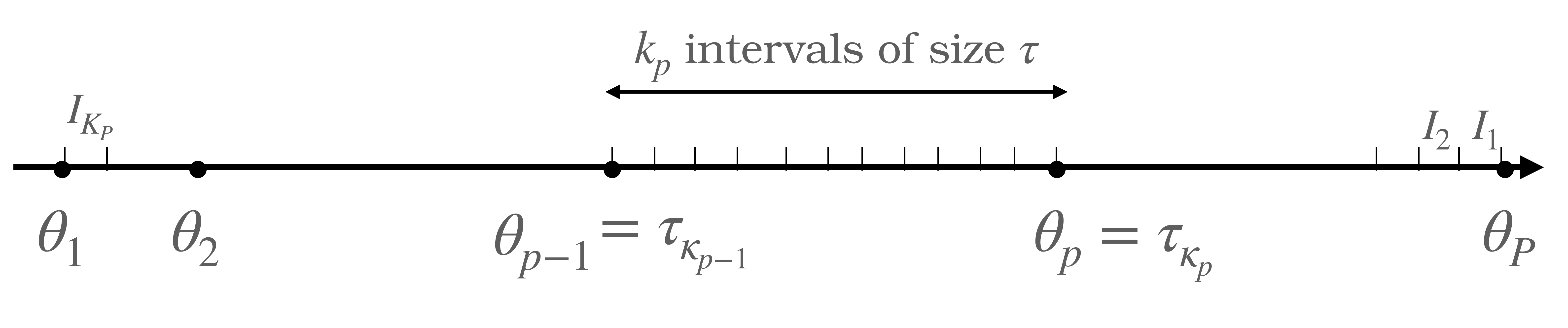}  
\caption{
The time interval $[\theta_1, \theta_P ]$ is split into $K_P$ intervals of smaller size $\tau$ denoted by $I_1, \dots, I_{K_P}$ and ranked in a decreasing order.
}
 \label{fig:f scales tau}
\end{figure}

In the next definition, we are going to strengthen the notion of single minimal forward cluster introduced in Definition~{\rm\ref{def minimal forward dynamics}}.
 \begin{defi}[Subexponential  cluster]
\label{def: phi tau single cluster}
Let~$i \in \{1,\dots,P\}$ and~$\tau_\ell< \theta_i = \tau_{\kappa_i}$ be given.  We consider 
a single  minimal forward cluster during the time interval $[\tau_\ell, \theta_i]$, 
 originating from a single particle.
This cluster is said to be subexponential  at  $\tau_\ell$ if on
 each time interval~$I_k$ for~$\kappa_i \leq k<\ell $ the number of annihilations~$n_i^k$ in the forward dynamics  is less than~$2^k$.
 
The corresponding single minimal cumulant $\bar \phi^{(i)}$
is defined as in  \eqref{eq: cumulant induction} as the pullback during~$[\tau_\ell, \theta_i]$ 
of the function $h^{(i)}$, with the appropriate subexponential restrictions on the annihilation numbers.
\end{defi} 
Note that the reference time for the subexponential growth is chosen to be~$\Theta $ rather than~$\theta_i$ as one might have expected (since the backward flow only starts at time~$\theta_i$). As will be apparent later (see Proposition~\ref{prop:superexp}), the reason for this choice is that the contribution of one single superexponential cluster  must be small enough to compensate the size of all other subexponential clusters, so the reference time has to be the same for all~$i \in \{1,\dots,P\}$.  
 
 \medskip

The main result of this section is the following: it shows that the fluctuation structure involving single subexponentials cumulants is preserved on any macroscopic time interval~$[ \theta_{p-1}, \theta_{p}]$.
Its proof is the goal of the following paragraphs.
 \begin{Prop}
\label{theta-iteration-prop}
Given a subset  $B$   of $\{p, \dots, P\}$, consider   for each $i \in B$ a   single minimal cumulant $\bar \phi^{(i)} $  of~$m_i$ variables, supported on  a  subexponential    cluster  at time~$\theta_p = \tau_{\kappa_p}$ as in Definition~{\rm{\ref{def: phi tau single cluster}}}.
Then denoting by~$\NN^\ell:= (n^\ell_i)_{i \in B}$ the number of particles annihilated in each block on the time step~$\ell$, and 
by~$M_{i}^{\kappa_{p-1}} = m_i + \sum_{\ell \in \LL_p} n_i^\ell$ the number of particles at time~$\theta_{p-1}  $, there holds
 \begin{equation*}
\begin{aligned}
& \Big| \bbE_\eps   \Big[   \Xi_{p-1}    \Big( \Otimes_{i \in B}   \zeta^\eps_{m_i, \theta_p} ( \bar \phi^{(i)}  ) \Big)\Big] - 
\sum_{(\NN^\ell)_{\ell \in \LL_p} \atop {\rm subexp}}   
\bbE_\eps   \Big[  \Xi_{p-1}    \Big( \Otimes_{i \in B}   \zeta^\eps_{M_i^{\kappa_{p-1}}, \theta_{p-1} } ( \bar \phi^{(i)}_{ \theta_p - \theta_{p-1}  }  ) \Big)\Big] \Big|\\
& \qquad \qquad \leq C_P \Big(  \prod_{i \in B\cup\{1,\dots, p-1\} } \| h^{(i)}\|_{L^\infty} \Big)
\left(\left(C_P  \Theta \right)^{ P\cdot 2^{\kappa_{p}}  } \eps^{\frac1{8d}}
+(C_P\Theta^{2P-1}  \tau ) ^{ 2^{\kappa_p - 1}}  \right)\, , 
\end{aligned}
\end{equation*}
where~$\bar \phi^{(i)}_{ \theta_p - \theta_{p-1}}  $ is supported on   a (single minimal) subexponential     cluster  at time~$\theta_{p-1 }$.
\end{Prop}

Using the homogeneity, we can assume without loss of generality  that $\| h^{(i)}\|_{L^\infty} \leq 1$  so from now on we no longer keep track of~$ \| h^{(i)}\|_{L^\infty} $ in the estimates.

 \subsection{The main term and the remainders on a small time step}

 Let us start by considering one time step~$I_\ell= [\tau_{\ell +1}, \tau_\ell] \subset[\theta_{p-1},\theta_p]$.
Given a subset  $B$   of~$\{p, \dots, P\}$, consider   for each $i \in B$ a   
function~$\bar \phi^{(i)} $  of~$m_i^\ell$ variables, supported on  a  (single minimal)  subexponential cluster  at time~$\tau_\ell$.
 Define
 $$\cI_{\MM ^\ell}^\ell:=  \bbE_\eps \left[   \Xi_{p-1}   \Otimes_{ i \in  B}
 \zeta^\eps_{m_i^\ell,\tau_\ell} \big(  \bar \phi^{(i)} \big)  \right]\, .
$$
Let us apply Proposition~\ref{prop-from theta to theta + tau}: we are going to show that asymptotically as~$\mu_\eps \to \infty$, the fluctuation structure at time $\tau_{\ell+1}  $  is similar to  the fluctuation structure at time $\tau_\ell$.
For this, consider the principal part $\bar \cI_{\MM^\ell}^\ell $ defined in Proposition \ref{fubini-prop},
but now on the time interval $I_\ell$
$$
\bar \cI_{\MM^\ell}^\ell :=  \sum_{  \NN^\ell  } \sum_{\eta   \cup \rho  \hbox{ \tiny \rm partition of }  B\atop  \rho  \hbox{ \tiny\rm  clustering} }  \prod _{q} \mu_\eps^{1 -\frac{|\rho _q|}2} 
 \bbE_\eps\Big[ 
    \bar \phi_{|I_\ell | } ^{(\rho _q)}  \Big]  \times
\bbE_\eps   \Big[
\Xi _{p-1}  \Big( \Otimes_{q}  \mu_\eps^{\frac12-\frac{ |\eta_q | }2}   \zeta^\eps_{M^{\ell+1}_{\eta_q} ,\tau_{\ell+1} } \Big(
 \bar \phi_{|I_\ell | } ^{(\eta_{q}) } \Big) \Big)\Big]\;,
$$
where
\begin{equation} \label{eq:estsec4ell}
\left| \cI_{\MM ^\ell}^\ell - \bar \cI_{\MM^\ell}^\ell \right|
\leq (C_P\Theta)^{M^\ell_B}    \eps^{\frac1{8d}}
\end{equation}
by Propositions \ref{fubini-prop}-\ref{remove recollisions on delta},
and split in turn $\bar \cI_{\MM^\ell}^\ell$ into different contributions
$$
\bar \cI_{\MM^\ell}^\ell  =  \cI^{ \ell+1}_{\MM^{\ell+1} }+\cI^{\rm exp,\ell}_{\MM} + \cI^{2,\ell}_{\MM}+\cI^{>, \ell }_{\MM}
 $$
which are defined below. We have written~$M_{\eta_q}^{\ell+1} $ for the number of particles in~$\eta_q$ at time~$\tau_{\ell+1}$, and $\NN^\ell$ the number of particles annihilated on $[\tau_{\ell+1}, \tau_\ell]$.

The main contribution is  a product where each of the $|B |$ terms has evolved independently~$(|\rho| = \emptyset$, $\eta $ consisting only of singletons) in a controlled way
$$
\begin{aligned}
 \cI^{\ell+1 }_{\MM^{\ell+1} }&:=\sum_{ \NN^\ell \ \hbox{\tiny\rm subexp} } 
\bbE_\eps   \Big[   \Xi_{p-1}     \Big( \Otimes_{i \in B}   \zeta^\eps_{ m_i^\ell +n_i^\ell ,\tau_{\ell+1}  } (  \bar  \phi_{|I_\ell|} ^{(i)}  ) \Big)\Big] \,,
\end{aligned}$$
  where the sum is restricted to subexponential $\NN^\ell$, meaning that   for all $i \in B$, $n_i^\ell \leq  2^\ell $.
The function~$\bar \phi_{|I_\ell|} ^{(i)} $  is thus  supported on  a (single minimal) subexponential  cluster at time~$\tau_{\ell+1}$ as   in Definition~\ref{def: phi tau single cluster}.
We stress the fact that $ \cI^{\ell+1 }_{\MM^{\ell+1} }$ 
has the same product structure as $\cI^\ell_{\MM^\ell}$.

\medskip

The remainder~$\bar \cI^\ell_{\MM^\ell}- \cI^{\ell+1 }_{\MM^{\ell+1} }$ is  split into the following three terms: 

\medskip

--- the
higher order cumulants
\begin{equation}
 \begin{aligned}
 \label{eq: > MM}
\cI^{>,\ell}_{\MM} & := \sum_{\eta   \cup \rho  \hbox{ \tiny \rm partition of }  B\atop  \rho  \hbox{ \tiny\rm  clustering} } \sum_{\NN_\rho ^\ell }    {\mathbf 1}^>  \prod _{q} \mu_\eps^{1 -\frac{|\rho _q|}2} 
 \bbE_\eps\Big[  \bar  \phi_{|I_\ell | } ^{(\rho_q )}  \Big]    \times  \sum_{\NN_{\eta}^\ell } 
\bbE_\eps   \Big[ \Xi_{p-1}    \Big(  \Otimes_{q}  \mu_\eps^{\frac12-\frac{ |\eta_q | }2}  \zeta^\eps_{ M_{\eta_{q}} ^{\ell+1} ,\tau_{\ell+1} } \Big(  \bar \phi_{|I_\ell |} ^{(\eta_{q})}  \Big) \Big)\Big] \, ,
 \end{aligned}
\end{equation}
where ${\mathbf 1}^>$ indicates that either at least one $\eta_{q}$ has more than one element, or   at least one $\rho_q$ has more than two elements;

\smallskip

--- a term collecting pair cumulants in $\rho$
$$
\cI^{2,\ell}_{\MM}  :=
 \sum_{\eta   \cup \rho  \hbox{ \tiny \rm partition of }  B \atop  \eta  \hbox{ \tiny \rm singletons}, 
 \,   \rho  \hbox{ \tiny \rm pairs } } 
   \sum_{\NN_\rho ^\ell }  \prod _{q}  \bbE_\eps\Big[   \bar  \phi_{|I_\ell | } ^{(\rho_q )}  \Big]   \times   \sum_{\NN_{\eta}^\ell } 
\bbE_\eps   \Big[  \Xi_{p-1}    \Big( \Otimes_{q}  
 \zeta^\eps_{ M_q  ^{\ell+1} ,\tau_{\ell+1} } \Big(   \bar \phi_{|I_\ell |} ^{(q)}  \Big) \Big)\Big] \, ,
$$

\smallskip

--- a term with only single minimal cumulants but at least one has a  superexponential growth
\begin{equation}
 \begin{aligned}
 \label{eq: superexp}
\cI^{\rm exp,\ell}_{\MM}  &:=\sum_{\NN^\ell   \hbox{\tiny\rm superexp} } 
\bbE_\eps   \Big[  \Xi_{p-1}    \Big( \Otimes_{i\in B}  
 \zeta^\eps_{ m_i^\ell + n_i^\ell ,\tau_{\ell+1} } \Big(   \bar \phi_{|I_\ell |} ^{(i)}  \Big) \Big)\Big] ,
\end{aligned}
\end{equation}
where the sum is restricted to superexponential $\NN^\ell$, meaning that  at least one $i \in B$ satisfies that~$  n_i^\ell > 2^{\ell }$.

 \medskip

The following Paragraphs \ref{remove clusters}, \ref{sec:PC} and \ref{remove superexp} of this section consist in proving that the terms~$\cI^{>, \ell}_\MM$,   $\cI^{2, \ell}_\MM$ and~$\cI^{\rm exp, \ell}_\MM$ are small.
As a consequence of Propositions 
\ref{prop: high order clustering}, \ref{Lem: miracle}, \ref{prop:superexp} and of \eqref{eq:estsec4ell},
we   deduce the following result on the time step $[\tau_{\ell+1}, \tau_\ell]$.
\begin{Prop}
\label{Prop: transport structure on time tau}
The following estimate holds:
\begin{equation*}
 \big| \bar \cI^\ell_{\MM^\ell} -  \cI^{\ell+1 }_{\MM^{\ell+1} } \big|
\leq
  \left(C_ P   \Theta\right)^{ 2^\ell  P}  \eps^{\frac1{8d}}
+\left(C_ P   \Theta^{2P-1}\tau \right)^{ 2^{\ell  -1}} \,.
\end{equation*}
\end{Prop}

\bigskip
\begin{proof}[Proof of Proposition~{\rm\ref{theta-iteration-prop}}]
Using  repeatedly the results of \eqref{eq:estsec4ell} and Proposition~\ref{Prop: transport structure on time tau},  which transports the fluctuation structure on any intermediate interval $I_\ell$, we can  recover the fluctuation structure on the  longer time interval~$[\theta_{p-1}, \theta_p]$.

\smallskip

Recall that $\theta_p = \tau_{\kappa_p}$ with the notation \eqref{eq: kappa}.
We consider  the following fluctuation structure at time $\theta_p$
\begin{align*}
 \cI_{\MM}& := 
\bbE_\eps   \Big[  \Xi_{p-1}    \Big( \Otimes_{i \in B}   \zeta^\eps_{m_i, \theta_p} (\bar\phi^{(i)}  ) \Big)\Big],
\end{align*}
where $B$ is a subset of $\{p, \dots, P\}$ and for each $i \in B$, $\bar\phi^{(i)} $ is a function  
 of $m_i  $ variables,
 supported on  a   subexponential   minimal  cluster at time~$\theta_p$. 

\bigskip
 Using repeatedly Proposition~\ref{Prop: transport structure on time tau} on $I_\ell$ for $\ell$ in $\LL_p$, we get that  
\begin{align*}
 \cI_{\MM}& = \sum_{(\NN^\ell)_{\ell \in \LL_p}  \atop {\rm subexp}} 
\bbE_\eps   \Big[   \Xi_{p-1}     \Big( \Otimes_{i \in B}   \zeta^\eps_{M_i ^{\kappa_{p-1}}, \theta_{p-1} } (  \bar \phi^{(i)}_{ \theta_p - \theta_{p-1}  }  ) \Big)\Big] 
+\sum_{\ell  \in \LL_p}   \cR^{\ell}_{\MM},
\end{align*}
where for each $i \in B$, $  \bar \phi^{(i)}_{ \theta_p-\theta_{p-1} } $ is a function  
 of $M_i ^{\kappa_{p-1}} := m_i + \sum_{ \ell \in \LL_p} n_i^\ell$ variables,
 supported on a subexponential   minimal single cluster at time~$\theta_{p-1}$.
 
The remainders $\cR^{\ell}_{\MM}$ come from the terms which are neglected at each step :  big clusters, vanishing pairings, and superexponential terms,  as
well as the remainder terms~$\sum_r \cR ^{{\rm int}}_r$. 
By Propositions \ref{remove recollisions on delta} (see Eq.\,\eqref{eq:estsec4ell}) and   \ref{Prop: transport structure on time tau}, we get at each step $\ell$
\begin{equation*}
\cR^{\ell }_{\MM} \leq 
   \left(C_P \Theta\right)^{ 2^\ell  P}  \eps^{\frac1{8d}}
+ \left(C_P\Theta^{2P-1}\tau  \right)^{2^{\ell-1} }\, .
\end{equation*}
Note that the exponential factor takes also into account the sum over all choices 
of $(\NN^k)_{k \leq \ell}$ since~$2^{1 + \dots + \ell} \leq 2^{\ell^2}$.
\end{proof}

\subsection{Removing big clusters}
\label{remove clusters}
We treat here the high order cumulants collected in $\cI^{>, \ell}_{\MM}$ (see~\eqref{eq: > MM}).
These cumulants describe dynamical correlations which  are negligible at the scale of the fluctuations.
More precisely, we have the following result.

\begin{Prop}
\label{prop: high order clustering}
The remainder accounting for  big clusters satisfies the following estimate:
\begin{equation}
\label{eq: estimation J>M}
\left| \cI^{>,\ell}_{\MM} \right| 
\leq  (C_P \Theta )^{2^\ell P}  \;   \mu_\eps^{-\frac12} \;  \tau^{\frac12}\, .
\end{equation}
\end{Prop}

\begin{proof}
By construction each cumulant $\bar \phi ^{(\rho_q)}_{|I_\ell |}$ appearing in \eqref{eq: > MM} is supported on a  minimal cluster  and the product of the expectations
can be estimated by Lemma \ref{phi-sigma-lemma}:
\begin{align}
\label{eq: inegalite rho intermediaire 1}
\left| 
 \bbE_\eps\Big[   \bar  \phi_{|I_\ell | } ^{(\rho_q )}  \Big] \right|
\leq  C     \left(  C_P\Theta \right)^{M _{\rho_q}^{\ell}-|\rho_q|}  (C_P\tau)^{N_{\rho_q}^\ell + |\rho_q|-1},
\end{align}
where $N_{\rho_q}^\ell  $ is the total number of particles annihilated in $\rho_q$ during $I_\ell$ and  $M_{\rho_q}^{\ell} = \sum_{i \in \rho_q} m_i^{\ell}$.
The functions $\bar \phi^{(i)}$ at time $\tau_\ell$ are supported on (single minimal) subexponential clusters whose sizes are bounded by
\begin{equation}
\label{eq: subexp taille}
m_i^{\ell} \leq 2^\ell
\quad \text{so that} \quad \sum_q M_{\rho_q}^{\ell} \leq 2^\ell  P.
\end{equation}
In the inequality \eqref{eq: inegalite rho intermediaire 1}, the power of $\tau$ keeps track of the total number~$N_{\rho_q}^\ell$ of annihilated particles  and 
of the~$ |\rho_q|-1 $  clusterings in the time interval $I_\ell$. Since $ \tau \ll 1$, the sums with respect to~$\NN_{\rho_q}^\ell$ converge and a factor $\tau^{|\rho_q|-1}$ remains.
We get
\begin{equation}\label{eq: big cluster decay 2}
  \left|\sum_{\NN_{\rho }^\ell}   \prod_q \,  \mu_\eps^{1 -\frac{|\rho _q|}2} 
 \bbE_\eps\Big[ \bar  \phi_{|I_\ell | } ^{(\rho_q )}  \Big] \right|
\leq       \left(  C_P  \Theta \right)^{2^\ell P }  \prod_q   \mu_\eps^{1 -\frac{|\rho _q|}2}  \tau^{ |\rho_q|-1}   \,
    \, .
\end{equation}
If one of the~$\rho_q$   is not a pair  then~$ |\rho_q|/2 \geq 3/2$ and this leads to an additional decay in $\tau\mu_\eps^{-\frac12}$.

We turn now to the estimate of the part of $\cI^{>,\ell}_{\MM}$ which is weighted by a product of fluctuation fields. 
By H\"older's inequality and Lemma~{\rm\ref{phi-sigma-lemma}}, we   have that 
\begin{equation}\label{eq: decoupling estimate}
\begin{aligned}
&  \left|\bbE_\eps   \Big[   \Xi_{p-1}   \;   \Big( \Otimes_{q }    \mu_\eps^{\frac12-\frac{  |\eta_q| }2}
\zeta^\eps_{M_{\eta_{q}}  ^{\ell+1}  ,\tau_{\ell+1} } \Big(    \bar \phi_{|I_\ell |}^{(\eta_{q})}  \Big) \Big)\Big]\right|      \\
& \qquad  
\leq
   \prod_{u=1} ^{p-1} 
\bbE_\eps \left[   \zeta^\eps \big( h^{(u)}\big)^{2(p-1)} \right]^{\frac{1}{2(p-1)}}
\bbE_\eps   \Big[\Big( \Otimes_{q}    \mu_\eps^{\frac12-\frac{  |\eta_q| }2}
  \zeta^\eps_{M_{\eta_{q}}  ^{\ell+1} ,\tau_{\ell+1} } \big(   \bar \phi_{|I_\ell |}^{(\eta_{q})} \big) \Big)^2\Big]^{\frac{1}{2}}    \\
& \qquad  
\leq  C_P \, 
\prod_q \,  
(C_P  \Theta)^{ M_{\eta_q} ^{\ell}+  \frac{1}2  N_{\eta_q}^\ell  -  \frac{1}2  |\eta_q|}
\; (C_P\tau)^{\frac{N_{\eta}^\ell }2} \;
\left(\frac{C_P \tau}{  \mu_\eps}\right)^{\frac{  |\eta_q| }2-\frac12} \,,  \end{aligned}
 \end{equation}
 where the moments of the fluctuation field are bounded by Proposition \ref{Proposition - estimates on g0}.
Summing over~$\NN_\eta^\ell$ gives
\begin{equation} 
\label{eq: big cluster decay 3}
\begin{aligned}
\sum_{\NN_\eta^\ell}  \left|\bbE_\eps   \Big[    \Xi_{p-1}   \;   \Big( \Otimes_{q }    \mu_\eps^{\frac12-\frac{  |\eta_q| }2}
\zeta^\eps_{M_{\eta_{q}}  ^{\ell+1} ,\tau_{\ell+1} } \Big(    \bar \phi_{|I_\ell |}^{(\eta_{q})}  \Big) \Big)\Big]\right|   
  \leq   (C_P  \Theta)^{2^{\ell } P }
\prod_q\left(\frac{C\tau}{  \mu_\eps}\right)^{\frac{  |\eta_q| }2-\frac12} \,  .
 \end{aligned}
 \end{equation}
 In particular if one~$\eta_q$ satisfies $|\eta_q|>1 $, we gain at least one power of~$\mu_\eps^{-1/2}\tau^{1/2}$. Combining \eqref{eq: big cluster decay 3} with \eqref{eq: big cluster decay 2}, we recover
 $$\left| \cI^{>,\ell}_{\MM} \right| 
\leq    (C_P \Theta )^{2^\ell P}  \; \mu_\eps^{-\frac12} \;  \tau^{\frac12}
$$
and~\eqref{eq: estimation J>M} follows thanks to the fact that~$\tau \leq1$. 
 We stress that the combinatorial factors arising from partitioning $B \subset \{1, \dots, P\}$ into $  \rho,\eta$ depend only on $P$.  Proposition~\ref{eq: estimation J>M} is proved.
\end{proof}

\subsection{Control of pair cumulants at equilibrium}
\label{sec:PC}
If~$\rho$ is made only of pairs and~$\eta$ of singletons, then~\eqref{eq: big cluster decay 2} and~\eqref{eq: big cluster decay 3} do not provide any decay as a power of $\mu_\eps$.
In fact out of equilibrium, these pairings contribute to the covariance and they were first analysed in \cite{S81} in terms of a recollision operator (see also \cite{BGSS3}). Instead at equilibrium, these terms vanish in the limit $\mu_\eps \to \infty$ due to a symmetry property of the limiting  measure. Thus to avoid the bookkeeping exercise of tracking these terms in the  iteration, we prefer to show that they do not contribute in the equilibrium regime considered in this paper.
\begin{Prop}
\label{Lem: miracle}
The remainder accounting for pair cumulants is estimated as follows~:
\begin{equation}
\begin{aligned}
\left| \cI^{2,\ell}_{\MM} \right| \leq    (C_P\Theta)^{2^\ell P}  \;
 \eps^{\frac1{8d}}
  \;.
\end{aligned}
\end{equation}
\end{Prop}

\begin{proof}
The key estimate is to show that 
for pairs the expectation $\bbE_\eps\Big[\bar  \phi^{(\rho_q )}_{|I_\ell |}  \Big]$ vanishes in the limit when~$\mu_\eps \to \infty$.
To fix ideas, we consider a clustering $\rho_q$ of the form $\{1 , 2 \}$
 connecting the single minimal cumulants $\bar \phi^{(1)}$, $\bar \phi^{(2)}$ supported on subexponential clusters at time $\tau_\ell$
   by an encounter on $I_\ell = [\tau_{\ell +1}, \tau_\ell]$.
Let us show that 
\begin{equation}
\label{eq: key equilibrium estimate}
 \; \,\Big| \sum_{\NN^\ell_{\{1, 2\}} }  \,\bbE_\eps\Big[   \bar\phi^{( \{1 , 2 \})}_{|I_\ell |}  \Big] \; \Big|
\leq  (C_P \Theta)^{m_1^{\ell}+m_2^{\ell}}  \,  \eps^{\frac1{8d}} \;, 
\end{equation}
from which   Proposition \ref{Lem: miracle} follows immediately by summing over the partitions and taking into account the size \eqref{eq: subexp taille} of the subexponential clusters at time $\tau_\ell$.

Using Proposition  \ref{prop-from theta to  theta + tau}  with $|B|=2, p =1$  on the time interval $I_\ell$ leads to the explicit decomposition
\begin{align}
\label{eq: pairing expansion}
\bbE_\eps   \Big[  \Otimes_{ i =1 }^2 
 \zeta^\eps_{m_i^{\ell},\tau_{\ell}} \big(  \bar \phi ^{(i)} \big) \Big] 
& =   \sum_{\NN^\ell_{\{1, 2\}}} 
 \bbE_\eps\Big[    \bar \phi^{( \{1 , 2 \})}_{|I_\ell |}  \Big] + \sum_{n_1^\ell, n_2^\ell} 
\bbE_\eps   \Big[   \Otimes_{ i =1 }^2 \zeta^\eps_{m_i^{\ell} +n_i^\ell , \tau_{\ell+1}} ( \bar \phi_{|I_\ell |}^{(i)}  )  \Big]  \\
& \qquad \qquad 
+ O\Big( (C_P\Theta)^{m_1^{\ell}+m_2^{\ell}}  \eps^{\frac1{8d}} \Big)
 \, . \nonumber
\end{align}
In this way,  estimating the expectation of the pair correlations $\bar \phi^{( \{1 , 2 \})}_{|I_\ell |}$ can be achieved by controlling the two $\Otimes$-products : one at time~$\tau_\ell$ and the other one at time $\tau_{\ell+1}$.
By construction $\bar \phi_{|I_\ell |}^{(i)}$ is supported on a single minimal cluster, thus by \eqref{E-HM 1} in Lemma~\ref{phi-sigma-lemma}, the expectation is small
\begin{equation}
\label{eq: tensor bounds}
\begin{aligned}
&\left|  \sum_{\NN^\ell_{\{1, 2\}}} \bbE_\eps   \Big[ \Otimes_{ i =1 }^2 \zeta^\eps_{m_i^{\ell}+n_i^\ell , \tau_{\ell+1}} ( \bar \phi_{|I_\ell |}^{(i)}  ) \Big] \right| 
\leq  (C_P\Theta)^{m_1^{\ell}+m_2^{\ell}} \, \sum_{n_1^\ell, n_2^\ell} (C_P\tau)^{(n_1^\ell +n_2^\ell)/2 } \;  \eps  \, .  
\end{aligned}
\end{equation}
The expectation $\bbE_\eps   \Big[  \Otimes_{ i =1 }^2 
 \zeta^\eps_{m_i,\tau_{\ell}} \big( \bar  \phi ^{(i)} \big) \Big]$ can be  estimated in the same way. Since $ \tau \ll 1$,
summing over $n_1^\ell, n_2^\ell$ completes the proof of~\eqref{eq: key equilibrium estimate}.
\end{proof}

\subsection{Removing superexponential collision trees}
\label{remove superexp}

In this section, we estimate dynamical flows exhibiting a superexponential number of annihilations, namely
$\cI^{\rm exp, \ell}_{\MM}$ defined in~(\ref{eq: superexp}). The result is the following.
\begin{Prop} 
\label{prop:superexp}
The remainder corresponding to superexponential clusters is estimated as follows~: 
$$
 \left| \cI^{\rm exp,\ell}_{\MM}  \right| 
\leq    \left(C_P\Theta^{2P-1}\tau\right)^{2^{\ell-1}}\;. 
$$
\end{Prop}

\begin{proof}
Compared to the previous sections, the control of $\cI^{\rm exp, \ell}_{\MM} $ requires a more careful description of the functions $ \phi^{(i)}$, taking into account the time sampling. They are supported on subexponential clusters at time $\tau_\ell$ (see Definition \ref{def: phi tau single cluster}).  
As a consequence each function depends at most on $2^{\ell}$ particles and the   total number of particles  at time~$\tau_\ell$ is at most~$|B |\,  2^{\ell}$. 

By definition of $\cI^{\rm exp, \ell}_{\MM} $,  there is~$ i \in B $ such that on~$I_\ell$, the number $n_i^\ell$ of annihilated particles associated with~$i$    is larger than~$2^{ \ell}$.
This means that on a time step of size~$\tau$, at least half of the particles (up to the factor~$|B |$) are removed.  %

By  H\"older's inequality  as in \eqref{eq: borne L2 Xi} and by \eqref{variance-phi-sigma} (without microscopic time splitting), we then have that 
  \begin{align*}
\left|\cI^{\rm exp,\ell}_{\MM}\right| 
&\leq \sum_{ \NN^\ell \  \hbox{\tiny\rm superexp}} \;
\left|  \bbE_\eps   \Big[
    \Xi_{p-1}  \;
\Big( \Otimes_{i \in B}    
\zeta^\eps_{m_i ^{\ell}+ n_i^\ell ,\tau_{\ell+1}} ( \bar \phi^{(i)}_{|I_\ell |}  ) 
\Big)\Big]\right|
\\
& \leq \sum_{\NN^\ell \  \hbox{\tiny\rm superexp}}
 \; \prod_{u=1} ^{p-1} \bbE_\eps \left[   \zeta^\eps \big( h^{(u)}\big)^{2(p-1)} \right]^{\frac{1}{2(p-1)}}
   \bbE_\eps   \Big[\Big( \Otimes_{i \in B}     \zeta^\eps_{m_i^{\ell}+n_i^\ell ,\tau_{\ell+1} } ( \bar \phi^{(i)}_{| I_\ell |}  ) \Big)^2\Big]^{\frac{1}{2}}\\
 & \leq C_P \sum_{\NN^\ell \  \hbox{\tiny\rm superexp}}\prod_{i\in B} \Big( (C_P \Theta)^{ 2m^\ell_{i}  +  n^\ell_i- 1 }   \tau^{n_{i }^{\ell}} \Big)^{1/2}  \\
 & \leq      \left(C_P \Theta^{2|B|-1} \tau \right)^{2^{\ell}/2}\,.
\end{align*}
 Proposition \ref{prop:superexp} follows using that $|B| \leq P$.
\end{proof}


\section{Asymptotics of the principal term}
\label{main-sec}

As explained in Section~\ref{sec:strategy}, our goal  is to pull back the test functions in time 
in order to build pairings and establish the Wick factorisation of the moments. 
In Section \ref{fluctuation-sec}, we have been able to pull back minimal cumulants from one sampling time $\theta_p$ to the next $\theta_{p-1}$.
In order to proceed to the next step and reach $\theta_{p-2}$, one has to take into account the 
new  structure of the expectation   at time $\theta_{p-1}$ after the multiplication by  the function $h^{(p-1)}$.
This will induce the pairing mechanism identified in Section~\ref{sec:strategy} 
which will be quantified  in this section.

In Section \ref{sec: Asymptotic pairing}, we   analyse  the repeated indices at time $\theta_{p-1}$ when taking the product of the fluctuation field  $\zeta^\eps_{\theta_{p-1}} (h^{(p-1)})$ with the   $\Otimes$-product obtained in Proposition~\ref{theta-iteration-prop} from the iteration.
The induction for the derivation 
of Proposition \ref{prop: appariement} is completed in Section \ref{sec: Gaussian structure of the limiting fluctuation field}.

\medskip


\subsection{Asymptotic pairing}
\label{sec: Asymptotic pairing}

Recall that at time $\theta_{p-1}$, for any $i\in \{p, \dots, P\}$, the  function~$\bar \phi^{(i)}_{ \theta_i - \theta_{p-1}  }$ is the pullback of the  test function $h^{(i)}$ on $[ \theta_{p-1}, \theta_i ]$ 
(in the sense of Definition \ref{def: phi tau single cluster}). 
We now study the product 
\begin{equation*}
\begin{aligned}
& \zeta^\eps_{\theta_{p-1}} (h^{(p-1)} )\times \Big( \Otimes_{i \in B}   \zeta^\eps_{m_i ^{\kappa_{p-1}},\theta_{p-1} } (\bar \phi^{(i)}_{   \theta_i - \theta_{p-1} }  ) \Big)   \\
&=\mu_\eps^{(|B|+1) /2}  
   \Big( \pi^\eps _{\theta_{p-1}} \big( h^{(p-1)} \big) - \bbE_\eps [ h^{(p-1)}] \Big) 
\; \left( 
\sum_{A \subset B  } \prod_{ i \in B \setminus  A}  \bbE_\eps  \Big[-\bar \phi^{(i)}_{  \theta_i - \theta_{p-1} } \Big] 
\pi^\eps_{ M_A^{\kappa_{p-1}},\theta_{p-1}} \Big( \Otimes_{i \in A}  \bar \phi^{(i)}_{\theta_i - \theta_{p-1} } \Big)
\right),
 \end{aligned}
\end{equation*}
where
\begin{equation}
\label{eq: borne mi} 
m_i ^{\kappa_{p-1}} \leq 2^{\kappa_{p-1}},
\end{equation}
as by definition, the number of  particles added in a subexponential cluster on~$I_\ell$ is smaller than~$2^\ell$.
  In particular, the following crude bound holds~:
$M_A ^{\kappa_{p-1}}= \sum_{i \in A} m_i ^{\kappa_{p-1}}\leq 2^{\kappa_{p-1}} |A| $. For the sake of readability, we will omit the superscript $\kappa_{p-1}$ in the rest of this paragraph.

We split the sum in $\pi^\eps _{\theta_{p-1}} \big( h^{(p-1)} \big)$ according to the repeated indices : when the index does not appear in the sum $\pi^\eps_{M_A,\theta_{p-1}}$, we get a $\ostar$-product, else we get a contracted product: 
$$
\begin{aligned}
  \zeta^\eps_{\theta_{p-1}} (h^{(p-1)} )&\times \Big( \Otimes_{i \in B}   \zeta^\eps_{m_i ,\theta_{p-1} } (\bar \phi^{(i)}_{ \theta_i - \theta_{p-1} }  ) \Big)   \\
&=    \mu_\eps^{(|B|+1) /2} \sum_{A \subset B  } 
\left( \prod_{ i \in B \cup\{p-1\}  \setminus  A}  \bbE_\eps  \Big[-\bar \phi^{(i)}_{\theta_i - \theta_{p-1} } \Big] \right) 
\,  \pi^\eps_{M_A,\theta_{p-1}} \Big( \bigotimes_{i \in A} \bar \phi^{(i)}_{  \theta_i - \theta_{p-1} } \Big) \\
& + \mu_\eps^{(|B|+1) /2}  \sum_{A \subset B  }
\left(    \prod_{ i \in B \setminus  A}  \bbE_\eps  \Big[-\bar \phi^{(i)}_{ \theta_i - \theta_{p-1} } \Big] \right) 
\,   \pi^\eps_{M_A+1,\theta_{p-1}} \Big( h^{(p-1)} \bigotimes_{i \in A}\bar  \phi^{(i)}_{\theta_i - \theta_{p-1} } \Big) \\
 & + \mu_\eps^{(|B|-1) /2} \sum_{A \subset B  }  
 \left(  \prod_{ i \in B \setminus  A}  \bbE_\eps  \Big[-\bar \phi^{(i)}_{\theta_i - \theta_{p-1} } \Big] \right) 
 \,  \pi^\eps_{M_A,\theta_{p-1}} \Big(  \sum_{ j \in A}  \psi^{(j,p-1)}\bigotimes_{i \in A\atop i \neq j } \bar  \phi^{(i)}_{  \theta_i - \theta_{p-1} } \Big) ,
 \end{aligned}
$$
where $\bar \phi^{(p-1)}_{ 0} := h^{(p-1)} $ and 
\begin{equation}
\label{psi-j,p} 
\psi^{(j,p-1)} (Z_{m_j})
: =
\bar \phi^{(j)}_{  \theta_j - \theta_{p-1} } (Z_{m_j})  \sum_{\ell =1}^{m_j} h^{(p-1)} (z_\ell )\,.
\end{equation}
Using the definition of the $\ostar$-product, we get the identity
$$
\begin{aligned}
& \zeta^\eps_{\theta_{p-1}} (h^{(p-1)} )\times \Big( \Otimes_{i \in B}   \zeta^\eps_{m_i ,\theta_{p-1} } (\bar \phi^{(i)}_{  \theta_i - \theta_{p-1} }  ) \Big)  =    \Otimes_{i \in B \cup\{p-1\} }   \zeta^\eps_{m_i ,\theta_{p-1} } (\bar \phi^{(i)}_{ \theta_i - \theta_{p-1} }  )  \\
 &\quad  + \mu_\eps^{(|B|-1) /2} \sum_{ j \in B} \sum_{\tilde A \subset B \setminus \{j \}  }   
\Big(\prod_{ i \in (B\setminus \{j \})  \setminus  \tilde A} 
  \bbE_\eps  \Big[- \bar \phi^{(i)}_{\theta_i - \theta_{p-1} } \Big] \Big)\pi^\eps_{M_{\tilde A} + m_j,\theta_{p-1}} \Big(   \psi^{(j,p-1)}  \bigotimes_{i \in \tilde A} \bar \phi^{(i)}_{  \theta_i - \theta_{p-1} }\Big) \, .
 \end{aligned}
 $$
Decomposing $\mu_\eps^{-m_j} \sum \psi^{(j,p-1)}$ in its expectation plus a fluctuation as in 
\eqref{exp-fluct-dec}
\begin{equation*}
\pi^\eps_{m_j, \theta_{p-1}} ( \psi^{(j,p-1)} )
 =  \bbE_\eps [\psi^{(j,p-1)}] + \mu_\eps^{-\frac12}    \zeta^\e_{m_j, \theta_{p-1} }(   \psi^{(j,p-1)}) 
  \, ,
\end{equation*}
we finally obtain, using again the definition of the $\ostar$-product, 
\begin{equation}
\label{eq: produit donne Psi}
\begin{aligned}
  \zeta^\eps_{\theta_{p-1}} (h^{(p-1)} )\times \Big( \Otimes_{i \in B}   \zeta^\eps_{m_i ,\theta_{p-1} } (\bar \phi^{(i)}_{ \theta_i - \theta_{p-1} }  ) \Big)  
 &=    \Otimes_{i \in B \cup\{p-1\} }   \zeta^\eps_{m_i ,\theta_{p-1} } (\bar \phi^{(i)}_{ \theta_i - \theta_{p-1} }  ) \\
  &+  \sum_{ j \in B}   \bbE_\eps [ \psi^{(j,p-1)} ]  \Otimes_{i \neq j }  \zeta^\eps_{m_i ,\theta_{p-1} } (\bar \phi^{(i)}_{  \theta_i - \theta_{p-1} }  ) \\
  & + \mu_\eps^{-1/2} \sum_{ j \in B}   \zeta^\eps_{m_j, \theta_{p-1}} ( \psi^{(j,p-1)} )\Otimes \Big( \Otimes_{i \neq j }  \zeta^\eps_{m_i,\theta_{p-1} } (\bar \phi^{(i)}_{  \theta_i - \theta_{p-1} }  )\Big)   \,.
 \end{aligned}
\end{equation}
As the function $\psi^{(j,p-1)}$ has the same cluster structure of 
$\bar \phi^{(j)}_{  \theta_j - \theta_{p-1} }$ (with only a slightly different weight), the inequalities of Lemma~{\rm\ref{phi-sigma-lemma}} apply also for $\psi^{(j,p-1)}$.
Proceeding as in \eqref{eq: decoupling estimate}, we apply 
 H\"older's inequality and Lemma~{\rm\ref{phi-sigma-lemma}} to show that
the last term in the previous decomposition has a vanishing contribution in the limit
$$
\begin{aligned}
\mu_\eps^{-1/2} &\left| \bbE_\eps   \Big[ \Xi _{p-2}
\; \Big( \sum_{ j \in B}   \zeta^\eps_{m_j^{\kappa_{p-1}}, \theta_{p-1}} ( \psi^{(j,p-1)} )    \Otimes_{i \neq j }  \zeta^\eps_{m_i ^{\kappa_{p-1}},\theta_{p-1} } ( \bar \phi^{(i)}_{\theta_i - \theta_{p-1} }  ) \Big) \Big]  \right|  
\leq    \mu_\eps^{-1/2} (C_P\Theta)^{M} ,
\end{aligned} 
$$
with $M= \sum_{i \in B}  m_i \leq 2^{\kappa_p} \, P$ and 
 $ \Xi_{p-2}:= \prod_{u=1} ^{p-2}  \zeta^\eps_{\theta_u} \big( h^{(u)}\big)$
 was introduced in \eqref{defJM}.

\medskip

Thus we  obtain the following result.
\begin{Prop}
\label{Spohn-prop}
For any $j \geq p$,  let $\bar \phi^{(j)}_{  \theta_j - \theta_{p-1} }$   be the generic  subexponential dynamical cluster from the expansion of $h^{(j)}$, and denote  $\psi^{(j,p-1)} $ its contraction with $h^{(p-1)}$  defined by {\rm(\ref{psi-j,p})}.   We set $\phi^{(p-1)}_0 = h^{(p-1)}$.
Then\begin{equation}
\label{Spohn-decomposition}
\begin{aligned}
\Bigg| \; 
\bbE_\eps  & \Big[   \Xi_{p-1} \;  \Big( \Otimes_{i \in B}   \zeta^\eps_{m_i ^{\kappa_{p-1}},\theta_{p-1} } (\bar \phi^{(i)}_{\theta_i - \theta_{p-1} }  )  \Big) \Big]  
 -  \bbE_\eps\Big[ \Xi _{p-2} \;  \Big(  \Otimes_{i \in B\cup\{p-1\} }   \zeta^\eps_{m_i^{\kappa_{p-1}} ,\theta_{p-1} } (\bar \phi^{(i)}_{ \theta_i - \theta_{p-1} } ) \;  \Big) \Big]\\
& - \sum_{j \in B}  \bbE_\eps [ \psi^{(j,p-1)} ] \bbE_\eps\Big[ 
\Xi_{p-2}    \Big( \Otimes_{i \in B \setminus \{j\} }   
\zeta^\eps_{m_i ^{\kappa_{p-1}},\theta_{p-1} } ( \bar \phi^{(i)}_{  \theta_i - \theta_{p-1} }  ) 
\Big) \Big] \Bigg|
 \leq C_P \mu_\eps^{-1/2} (C_P\Theta)^{2^{\kappa_p} \, P}  .
\end{aligned}
\end{equation}
\end{Prop}

\smallskip
The second term in (\ref{Spohn-decomposition}) has the required structure to be pulled back up to time $\theta_{p-2}$ by considering cumulants indexed by the larger set $B \cup \{ p - 1\}$. The  sum in (\ref{Spohn-decomposition}) involves the product of $\bbE_\eps \left[   \psi^{(j,p-1)} \right]$, which will be linked to a covariance 
in  Corollary \ref{covariance extraction}, and a $\Otimes$-product which has the right structure to be pulled back up to time $\theta_{p-2}$ by considering cumulants indexed by the smaller set $B \setminus \{j\}$.

\begin{Cor}
\label{covariance extraction}
For any $j \geq p$,  denote by  $\bar \phi^{(j)}_{ \theta_j- \theta_{p-1} }$    the subexponential minimal cumulants  issued from $h^{(j)}$, and denote by $\psi^{(j,p-1)} $ their contraction with $h^{(p-1)}$  defined by {\rm(\ref{psi-j,p})}.  Then, one has
\begin{equation*}
\Bigg|  \sum_{\NN_j \atop {\rm subexp}  }
\bbE_\eps \left[   \psi^{(j,p-1)} \right] - \bE_\eps \Big[ \zeta^\eps _{\theta_{p-1} } (h^{(p-1)} ) \,  \zeta^\eps _{\theta_j} (h^{(j)} )  \Big] \;  \Bigg|
\leq (C _P \Theta  )^{ 2^{\kappa_p}  }\eps^{\frac1{8d}} 
 +  (C _P \Theta^3\tau  )^{ 2^{\kappa_p-1} }\, \cdotp
\end{equation*}
In particular, one has the uniform bound
 $$\left| \sum_{\NN_j \atop {\rm subexp}  }
\bbE_\eps \left[   \psi^{(j,p-1)} \right] \right| \leq C\,.$$
\end{Cor}

\begin{proof}[Proof of Corollary {\rm\ref{covariance extraction}}]
Using repeatedly Proposition \ref{theta-iteration-prop}  with  only two test functions $h^{(p-1)}, h^{(j)}$ but with the same time sampling $I_\ell$, we get 
\begin{equation*}
\begin{aligned}
&\left| \bbE_\eps   \Big[  \zeta^\eps_{\theta_{p-1}} \big( h^{(p-1)}\big)     \zeta^\eps_{\theta_j} (h^{(j)} ) \Big)\Big] -    \sum_{\NN_j \atop {\rm subexp}  }
\bbE_\eps   \Big[ \zeta^\eps_{\theta_{p-1}} \big( h^{(p-1)}\big)      
\zeta^\eps_{m_{j} ^{\kappa_{p-1}},\theta_{p-1}} (  \bar \phi^{(j)}_{ \theta_j - \theta_{p-1} }  ) \Big)\Big]  \right| \\
&\qquad \qquad \leq  (C _P \Theta  )^{ 2^{\kappa_p}  }\eps^{\frac1{8d}} 
 +  (C _P \Theta^3\tau  )^{ 2^{\kappa_p-1} }\
 \, .
\end{aligned}
\end{equation*}
Then, from Proposition \ref{Spohn-prop}  with $B = \{j\}$, we get 
$$
 \begin{aligned}
&\Bigg| \bbE_\eps   \Big[ \zeta^\eps_{\theta_{p-1}} \big( h^{(p-1)}\big)      \zeta^\eps_{m_j ^{\kappa_{p-1}},\theta_{p-1}} (\bar \phi^{(j)}_{ \theta_j - \theta_{p-1} }  ) \Big)\Big]  \\
& \qquad -
\bbE_\eps\Big[  \zeta^\eps_{\theta_{p-1}} \big( h^{(p-1)}\big)   \Otimes  \zeta^\eps_{m_j ^{\kappa_{p-1}},\theta_{p-1} } (\bar \phi^{(j)}_{ \theta_j - \theta_{p-1} }  ) \Big] 
- \bbE_\eps [ \psi^{(j,p-1)} ]  
\Bigg| 
\leq C \; \mu_\eps^{-1/2} (C \Theta)^{2^{\kappa_p}}  .
\end{aligned}
$$
It remains to sum over 
the subexponential annihilation numbers $\NN_j = (n_j^\ell)_{ \kappa_j < \ell \leq \kappa_{p-1} -1 }$.
Since~$\NN_j$ takes at most $2^{1 + \dots + 2^{\kappa_p}}$ values, the error terms remain under control  and we get
$$
 \begin{aligned}
&\left| \bbE_\eps   \Big[\zeta^\eps_{\theta_{p-1}} \big( h^{(p-1)}\big)     \zeta^\eps_{\theta_j} (h^{(j)} )\Big]  - \sum_{\NN_j \atop {\rm subexp}  }  \bbE_\eps [ \psi^{(j,p-1)} ]
-   \sum_{\NN_j \atop {\rm subexp}  }
\bbE_\eps\Big[ \zeta^\eps_{\theta_{p-1}} \big( h^{(p-1)}\big)   
\Otimes  
\zeta^\eps_{m_j^{\kappa_{p-1}},\theta_{p-1} } (\bar \phi^{(j)}_{ \theta_j - \theta_{p-1} }  ) \Big] \right|
\\
&\qquad \qquad  \leq (C _P \Theta  )^{ 2^{\kappa_p}  }\eps^{\frac1{8d}} 
 +  (C _P \Theta^3\tau  )^{ 2^{\kappa_p-1} }\ \;.
\end{aligned}
$$
By  (\ref{eq: tensor bounds}), 
we find that the first term in the right-hand side vanishes in the limit~:
\begin{align*}
\Bigg|   \sum_{\NN_j \atop {\rm subexp}  } 
  \bbE_\eps\Big[ \zeta^\eps_{\theta_{p-1}} \big( h^{(p-1)}\big)   \Otimes  
\zeta^\eps_{m_j^{\kappa_{p-1}},\theta_{p-1} } (\bar \phi^{(j)}_{ \theta_j - \theta_{p-1} }  ) \Big]  \Bigg|
\leq  (C _P \Theta  )^{ 2^{\kappa_p+1}  }\eps\,.
\end{align*}
This completes Corollary \ref{covariance extraction}.
 \end{proof}

In the following section, we are going to iterate these propositions in order to decompose the moments of the field as a product of covariances and some remainder terms.

\subsection{Proof of Proposition 
\ref{prop: appariement} : convergence of the moments}
\label{sec: Gaussian structure of the limiting fluctuation field}

We are now going to combine the previous results  to prove Proposition 
\ref{prop: appariement}.
We proceed by induction and at time~$\theta_p$, we assume that the following decomposition  holds, with notation~(\ref{psi-j,p}):
\begin{equation}
\begin{aligned}
\bbE_\eps \Big[   \prod_{u=1} ^{P}  \zeta^\eps_{\theta_u} \big( h^{(u)}\big)   \Big] 
& \sim
\sum_{B\cup B^c = \{p,\dots, P\} \atop  B^c \cap B = \emptyset }
\sum_{\eta_p \in {\mathfrak S}_{ B^c }^{\rm pairs}}
\prod_{\{i,j\} \in \eta _p\atop i <j } \Big( \sum_{\NN_j \atop {\rm subexp}  }  \bbE_\eps [ \psi^{(j,i)} ] \Big)  \\
& \qquad \qquad 
\times \sum_{\NN_B \atop {\rm subexp} }
\bbE_\eps \Big[  \Big( \prod_{u=1}^{p-1}  \zeta^\eps_{\theta_u} \big( h^{(u)}\big) \Big)    \Big(   \Otimes_{i \in B} \zeta^\eps_{m_i^{\kappa_{p-1}} ,\theta_{p-1} } (\bar \phi^{(i)}_{ \theta_i - \theta_{p-1} }  )\Big) \Big] \, ,
\label{eq: induction p}
\end{aligned}
\end{equation}
where $\sim$ means that the difference is bounded by $  (C _P \Theta  )^{ P\cdot 2^{\kappa_P}  }\eps^{\frac1{8d}} 
 +  (C _P \Theta^{2P-1}\tau)^{1/2} $.
Notice that the decomposition is valid at time $\Theta$ with $B = \{P\}$ and $B ^c   = \emptyset$.

Given $B, \eta_p \in  {\mathfrak S}_{ B^c }^{\rm pairs}$, we are going to apply the procedure described in the 
previous sections to expand the expectation in the second line of \eqref{eq: induction p} and to derive the induction relation at time $\theta_{p-1}$.
Combining  (\ref{Spohn-decomposition}) and  (\ref{eq: induction p}), we obtain 
\begin{equation}
\label{eq: induction p-1 decomposition}
\begin{aligned}
& \bbE_\eps  \Big[     \prod_{u=1} ^{P}  \zeta^\eps_{\theta_u} \big( h^{(u)}\big)   \Big] 
 \sim
\sum_{B\cup B^c = \{p,\dots, P\} \atop  B^c \cap B = \emptyset }
\sum_{\eta_p \in {\mathfrak S}_{ B^c }^{\rm pairs}}
\prod_{\{i,j\} \in \eta _p\atop i <j }\Big(  \sum_{\NN_j \atop {\rm subexp}  }  \bbE_\eps [ \psi^{(j,i)} ]  \Big)  \\
&   \qquad \qquad \qquad \times \sum_{\NN_B \atop {\rm subexp } }
\bbE_\eps\Big[   \Big( \prod_{u=1} ^{p-2}  \zeta^\eps_{\theta_u} \big( h^{(u)}\big) \Big)    \Big( \Otimes_{i \in B\cup\{p-1\} }   
\zeta^\eps_{m_i ^{\kappa_{p-1}},\theta_{p-1} } (\bar \phi^{(i)}_{ \theta_i - \theta_{p-1} }  )\Big)  \Big] \\
&+\sum_{B\cup B^c = \{p,\dots, P\} \atop  B^c \cap B = \emptyset }
\sum_{\eta_p \in {\mathfrak S}_{ B^c }^{\rm pairs}}
\prod_{\{i,j\} \in \eta _p\atop i <j }\Big(  \sum_{\NN_j \atop {\rm subexp}  }  \bbE_\eps [ \psi^{(j,i)} ]  \Big)    \\
&   \times
\sum_{j' \in B}  \left(  \sum_{\NN_{j'} \atop {\rm subexp} } \bbE_\eps [ \psi^{(j' ,p-1)}]  \right)
 \sum_{\NN_{B \setminus \{ j'\} } \atop {\rm subexp} }
\bbE_\eps\Big[  \Big( \prod_{u=1} ^{p-2}  \zeta^\eps_{\theta_u} \big( h^{(u)}\big) \Big)    \Big( \Otimes_{i \in B \setminus \{j' \} }   \zeta^\eps_{m_i ^{\kappa_{p-1}},\theta_{p-1} } (\bar \phi^{(i)}_{  \theta_i - \theta_{p-1} }  ) \Big) \Big]\;.  \\
\end{aligned}
\end{equation}
 In the first contribution, since  there is no new pairing, we set $\eta_{p-1} = \eta_p$ and the product form holds now on the set
 $B \cup \{p-1\} \subset \{p-1, \dots, P\}$.
For the second contribution,  we define the new set $\eta_{p-1} = \eta_p \cup \{ (p-1,j') \}$ with the additional pair.
The $\Otimes$-product at time $\theta_{p-1}$ holds on the reduced set $B \setminus \{ j' \}$.

\medskip

Thus the induction assumption \eqref{eq: induction p} is also valid at $\theta_{p-1}$ and it can be iterated  up to time~$\theta_1$:
$$
\begin{aligned}
\bbE_\eps  \Big[ & \Big( \prod_{u=1} ^{P}  \zeta^\eps_{\theta_u} \big( h^{(u)}\big) \Big)  \Big] \\
&  \sim
\sum_{B \subset \{1,\dots, P\}}
\; \sum_{\eta_1 \in {\mathfrak S}_{ B^c }^{\rm pairs}} \; 
\prod_{\{i,j\} \in \eta _1\atop i <j }\Big(  \sum_{\NN_j \atop {\rm subexp}  }  
\; \bbE_\eps [ \psi^{(j,i)} ]  \Big)  \; \sum_{\NN_B \atop {\rm subexp} }
\bbE_\eps \Big[     \Otimes_{i \in B} \Big(  \zeta^\eps_{m_i ^{\kappa_{1}},\theta_{1} } ( \bar \phi^{(i)}_{ \theta_i - \theta_{1} }  )\Big) \Big] .
\end{aligned}
$$
As the induction is applied only $P$ times, the difference remains bounded by  $  (C _P \Theta  )^{ P\cdot 2^{\kappa_P}  }\eps^{\frac1{8d}} 
 +  (C _P \Theta^{2P-1} \tau )^{1/2}$.

By \eqref{E-HM 1}  in Lemma \ref{phi-sigma-lemma},  the terms for which $B \not = \emptyset$
can be neglected. Thus the only remaining term is $B = \emptyset$ and 
$\eta_1$ is an element of ${\mathfrak S}_P^{\rm pairs}$ :
\begin{align}
\label{eq: quantitative convergence}
\bbE_\eps  \Big[    \prod_{u=1} ^{P}  \zeta^\eps_{\theta_u} \big( h^{(u)}\big)   \Big]  
\sim
\sum_{ \eta_1 \in {\mathfrak S}_P^{\rm pairs} }
\prod_{\{i,j\} \in \eta _1\atop i <j }\Big(  \sum_{\NN_j \atop {\rm subexp}  }  \bbE_\eps [ \psi^{(j,i)} ]  \Big)  .
\end{align}
Recall that~$\sim$ means that the difference is bounded by~$  (C _P \Theta  )^{ P\cdot 2^{K_P}  }\eps^{\frac1{8d}} 
 +  (C _P \Theta^{2P-1}\tau)^{1/2} $, so the factorisation estimate \eqref{eq: quantitative convergence} is quantitative and  remains valid for (slowly) diverging times.
 
 \smallskip
 
   Identifying the covariances with Corollary \ref{covariance extraction} concludes the proof of Proposition \ref{prop: appariement}. \qed




\section{Geometric estimates}
\label{geometric-sec}

In this section, we adapt previous results from \cite{BGSS2} to the present context, allowing to prove the first parts of Lemma \ref{phi-sigma-lemma} and Lemma \ref{phi-rec-lemma}: we control the cluster functions by establishing the bound on the expectation~\eqref{expectation-phi-sigma} as well as    the smallness estimate~\eqref{expectation-phi-rec} for non minimal clusters.

In this section, since we will consider only one packet $\sigma_i$, we will drop the index $i$ to lighten the notation, as well as  the time dependence in the test functions. 
We  thus consider a collection $( \bar \phi^{(j)})_{j \in \sigma}$ of single minimal cumulants  originating from single particles at times~$\theta_j$ as in \eqref{eq: cumulant induction}.
These cumulants are  aggregated on $[\theta-r \delta, \theta]$ as in \eqref{notbarphi-def} to form the cumulant $\phi^{(\sigma)}$ at time $\theta-r\delta$  supported on  forward clusters as in Definition~\ref{def forward dynamics}.

We therefore have   $|\sigma|$ blocks of cardinalities $\MM^{r\delta}  = (M_j^{r\delta})_{j \in \sigma}$ at time $\theta-r \delta$.
We denote by 
\begin{itemize}
\item $\NN^r$ the number of annihilations  in the different blocks on $\cI_\delta := [\theta- r\delta, \theta-(r-1) \delta]$;
\item $\NN^{<r}$    the number of annihilations  in the different blocks  on $\cI_\tau  := [\theta- (r-1)\delta, \theta]$;
\end{itemize}
and define 
$N^r=\sum_{j\in \sigma}  N_j^r$, $N^{<r} = \sum _{j \in \sigma} N_j^{<r} $, $M^{r\delta} = M+ N^r +N^{<r} = \sum_{j\in \sigma } M_j^{r\delta}$.

To determine the forward dynamics on $[\theta-r \delta, \Theta]$, we also need to fix  as in Definition~\ref{def forward dynamics}
\begin{itemize}
\item $\KK \in \{0, \dots,K_\gamma\}^{M^{r \delta}}$ counting for each particle the internal encounters without annihilation on $\cI_\delta$  (recall that the number of such encounters is under control by construction of $\phi^{(\sigma) }$);
\item $(\SS , \bar \SS)\in \{1,-1\}^{2(M^{r \delta}-|\sigma|)}$ prescribing  the encounters with annihilation  on $\cI_\delta \cup \cI_\tau \cup \cI_\theta$, among which we denote by~$(\SS^r , \bar \SS^r)\in \{1,-1\}^{2 N^r}$ the signs prescribing  the encounters  on $\cI_\delta$,~$(\SS^{< r} , \bar \SS^{< r})\in \{1,-1\}^{2N^{< r}}$ those prescribing  the encounters  on $\cI_\tau$ and finally~$(\SS^\theta , \bar \SS^\theta)\in \{1,-1\}^{2(M-|\sigma|)}$ the signs prescribing  the encounters on $\cI_\theta$;
\item  a partition $\varsigma$ of $\sigma$ prescribing the packets at time~$\theta- (r-1)\delta$, as well as a partition $\lambda $ of these packets in forests, prescribing the external encounters on $\cI_\delta$;
\item $\EE = (\EE_i)$ where for each packet $\varsigma_i$,  $\EE_i \in \{1,-1\}^{|\varsigma_i|-1}$ prescribing the external encounters on $\cI_\tau $.
\end{itemize}
  With this notation, plugging~(\ref{barphi-def}) into~(\ref{notbarphi-def}) provides
$$\begin{aligned}
 \phi^{(\sigma)}_{r\delta}(Z_{\M^{r\delta}})
  &=   \Big( \prod_{j \in \sigma} {M_j^{(r-1)\delta}! \over M_j^{r\delta} !} \Big) \sum _{\varsigma  \in \cP_\sigma}   \  \mu_\eps^{N^r + | \varsigma  |-1}  \sum_{\KK , \SS^r , \bar \SS^r ,\lambda \atop 
 \varsigma \hookrightarrow \lambda} \sgn (\SS^r  )  
   \varphi_{\{\lambda_1, \dots, \lambda_{|\lambda|}\}} \\
&   \qquad\qquad  \times  
\prod_{\ell =1}^{|\lambda|}\Big({\sharp}_\delta  \bigotimes_{\varsigma_i \subset \lambda_\ell} 
\bar \phi_{(r-1) \delta} ^{ (\varsigma_i)} \Big) 
\indc_{ \{Z_{\M_{\lambda_\ell}^{r\delta} }    \in \cR^{\lambda_\ell}_{{\mathbf K}_{\lambda_\ell},\SS^r_{\lambda_\ell},\bar \SS^r_{\lambda_\ell}} \} }  
 \\
 &=   \left( \prod_{j \in \sigma} {m_j! \over M_j^{r\delta} !} \right) \mu_\eps^{N^r +N^{<r} + | \sigma  |-1}  \sum _{\varsigma  \in \cP_\sigma}   \     
 \sum_{\KK , \SS^{\leq r}  , \bar \SS^{\leq r}  ,\EE, \lambda \atop 
 \varsigma \hookrightarrow \lambda}   \sgn(\EE) \sgn(\SS^{\leq r} )
   \varphi_{\{\lambda_1, \dots, \lambda_{|\lambda|}\}} \\
&   \qquad 
\times
\indc_{\{  Z_{\M^{r\delta} } \mbox{\ \tiny forward cluster associated with $(\SS^{\leq r} , \bar \SS^{\leq r}  , \varsigma , \lambda , \KK , \EE)$} \}}
   \prod_{\ell =1}^{|\lambda|}\prod_{\varsigma_i \subset \lambda_\ell} 
  \big( {\sharp}_{(r-1)\delta} \bigotimes _{j \in \varsigma_i}  \; \bar \phi^{(j)}  \big)
\end{aligned}
$$
where we wrote $ (\SS^{\leq r}, \bar \SS^{\leq r}) = ( \SS^{< r},\SS^{r}, \bar \SS^{< r}, \bar \SS^{r})$, and $ \sgn(\SS^{\leq r}) $ for the product of all the components of $\SS^{\leq r}$. 
On the other hand 
by the assumption \eqref{eq: cumulant induction} on the structure of the test functions,
one has
\begin{equation}
\bar \phi^{(j)}  (Z_{m_j})= \frac{ \mu_\eps^{m_j-1}}{m_j!} \sum_{\SS_j^\theta , \bar \SS_j^\theta}\sgn(\SS_j^\theta ) 
\indc_{Z_{m_j} \in \cR^{\rm min}_{  \SS_j^\theta, \bar \SS_j^\theta } } 
 \big ({\sharp}_{\theta_j-\theta} h^{(j)}  \big )\, ,
\end{equation}

\bigskip
 Recall that we can assume without loss of generality that $\| h^{(j)}\| _\infty \leq 1$ for all $j \in \sigma$. Since~${\mathbf 1}_{\Upsilon^\eps_N} \leq 1$, inequalities~\eqref{expectation-phi-sigma} and~\eqref{expectation-phi-rec} follow from their counterparts at equilibrium
\begin{equation}
\label{expectation-phi-sigma-eq}
 \bbE_\eps^{\rm eq}\Big[    \Big| \phi^{(\sigma) }\Big|  \Big] \leq  C_P 
 ( |\sigma| K_\gamma) ^{M^{r\delta}}   \;  (C\Theta)^{M - |\sigma|}  (C \delta)^{ N^{r}}  (C \tau)^{N^{<r}+|\sigma | - 1}  \, ,
\end{equation}
and
\begin{equation}
\label{expectation-phi-rec-eq}
\begin{aligned}
 \bbE_\eps^{\rm eq}\Big[   \Big|
 \phi^{(\sigma ), {\rm cyc}}\Big|  \Big] 
 \leq   \eps\delta  |\log \eps| (\Theta |\log \eps|) ^{ 2d+4} \; C_P (  |\sigma| K_\gamma) ^{M^{r\delta}}  (C \Theta)^{M -|\sigma|} (C \delta)^{ (N^{r}-1)_+}  (C \tau)^{(N^{ <r}+|\sigma | - 2)_+}\, ,
 \end{aligned}
\end{equation}
which  we now prove.

\bigskip
Estimating roughly the $L^\infty$ norm of the cumulant $ \varphi_{\{\lambda_1, \dots \lambda_{|\lambda|}\}}$  by  $|\lambda|! $ (note that~$|\lambda|!\leq |\sigma|! $), 
and using that 
\begin{equation} \label{eq:combinfacp}
{M^{r\delta} !  \over  \prod_{j \in \sigma} M^{r\delta}_j! }  
\leq   |\sigma|^{ M^{r\delta} } \;,
\end{equation}
we infer that 
\begin{equation}
\label{phi-geometric-est-sec7}
| \phi^{(\sigma)}(Z_{\M^{r\delta}}) | \leq  |\sigma|^{ M^{r\delta} }  \Big(\prod_{j \in \sigma}\| h^{(j)} \|_\infty  \Big)\, \frac{\mu_\eps^{M^{r\delta}-1}}{M^{r\delta} !}\,\sum_{ \bar \SS, \SS, \KK, \varsigma, \lambda, \EE} |\lambda| !   \;   \indc_{\{  Z_{\M^{r\delta} } \mbox{\ \tiny forward cluster} \}}\,,
\end{equation}
where the forward cluster in the indicator refers to the dynamics in $[\theta-r \delta, \Theta]$ constructed as described above and depends on  the whole set of global parameters $\bar \SS, \SS, \KK, \varsigma, \lambda, \EE$. Note that the cardinality of such parameters is bounded by $C_P(CK_\gamma)^{M^{r\delta}}$ for some pure constant $C$, and $C_P$ depending only on $P$.

\medskip

We now  need to describe more precisely the  support of~$\phi^{(\sigma)}$, i.e. to extract from the cluster structure some ``independent" geometric  conditions.
To show Inequality \eqref{expectation-phi-sigma-eq} on $\bbE_\eps^{\rm eq}\Big[   \Big| \phi^{(\sigma)} \Big|  \Big] $ it is enough to prove that  
the size of the support is controlled by
\begin{equation}
\label{expectation-phi-sigma bis}
\begin{aligned}
\sup_{x_{M^{r\delta}} \in \T^d} \int  \indc_{\{  Z_{\M^{r\delta} } \mbox{\ \tiny forward cluster} \}} &
  \cM^{\otimes M^{r\delta}}\left( V_{M^{r\delta}}\right)
 dX_{M^{r\delta}-1 }dV_{M^{r\delta} }  \\
& \leq C {M^{r\delta} ! \over \mu_\eps^{M^{r\delta} - 1}} \;  (C \Theta)^{ M-|\sigma|}  (C \delta)^{ N^r}
(C \tau)^{N^{<r}+|\sigma | - 1}
\end{aligned}
\end{equation}
 for some pure constant $C>0$, uniformly in $(\bar \SS, \SS, \KK, \varsigma, \lambda, \EE)$.
Indeed, assuming \eqref{expectation-phi-sigma bis}, we deduce from 
\eqref{phi-geometric-est-sec7} that 
\begin{equation}
\label{phi-geometric-est-sec7 bis}
\begin{aligned}
\bbE_\eps^{\rm eq}\Big[   \left| \phi^{(\sigma)} \right| \Big]   &
 =  \int G^{\eps, {\rm eq}}_{M^{r\delta}}
 \,  | \phi^{(\sigma)} | \, dZ_{M^{r\delta}}
 \\ &
 \leq  |\sigma|^{ M^{r\delta} }\frac{\mu_\eps^{M^{r\delta}-1}}{M^{r\delta}!}\, \prod_{j \in \sigma}\| h^{(j)} \|_\infty \, \sum_{ \bar \SS, \SS, \KK, \varsigma, \lambda, \EE} |\lambda| !   \, \int G^{\eps, {\rm eq}}_{M^{r\delta}}\,  \indc_{\{  Z_{\M^{r\delta} } \mbox{\ \tiny forward cluster} \}} \, dZ_{M^{r\delta}}
 \\
&   \leq |\sigma|^{ M^{r\delta} }\frac{\mu_\eps^{M^{r\delta}-1}}{M^{r\delta}!}\, \prod_{j \in \sigma}\| h^{(j)} \|_\infty \, \sum_{ \bar \SS, \SS, \KK, \varsigma, \lambda, \EE} |\lambda| !   \, \int  \cM^{\otimes M^{r\delta}}\,  \indc_{\{  Z_{\M^{r\delta} } \mbox{\ \tiny forward cluster} \}} \, dZ_{M^{r\delta}}
 \\
& \leq  C_P  \,|\sigma|^{ M^{r\delta} } \, {(CK_\gamma) ^{ M^{r\delta}} } (C \Theta)^{ M-|\sigma|}  (C \delta)^{ N^r  }
(C \tau)^{N^{<r}+|\sigma | - 1}\,,
\end{aligned}
\end{equation}
where the sums over the parameters $\bar \SS, \SS, \KK, \varsigma, \lambda, \EE$ have been bounded by $(CK_\gamma)^{M^{r\delta}}$ and combinatorial factors depending only on $P$.
Notice that, to compute the expectation, we have made use of the correlation functions of the 
 invariant measure \eqref{eq: initial measure}, which we recall:
$$G^{\eps, {\rm eq}}_M \left( Z_M\right):= 
\frac{\cM^{\otimes M}}{\cZ^ \eps} \,\sum_{n=0}^{\infty}\frac{\mu_\eps^n}{n!} 
\int_{ (\T ^d\times\R^d)^n} dz_{M+1}\dots dz_{M+n} \,\indc_{
{\mathcal D}^{\eps}_{M+n}}(Z_{M+n}) 
 \, \cM^{\otimes n} \;,\qquad M = 1,2,\cdots
 $$
with $\cZ^\eps$ given by \eqref{eq: partition function}. Since $\indc_{{\mathcal D}^{\eps}_{M+n}}(Z_{M+n}) \leq \indc_{{\mathcal D}^{\eps}_{M}}(Z_{M})
\indc_{{\mathcal D}^{\eps}_{n}}(z_{M+1},\dots,z_{M+n})$, these correlation functions satisfy the pointwise bound
$G^{\eps, {\rm eq}}_M \left( Z_M\right) \leq
\cM^{\otimes M}\;,$ which justifies the third line of \eqref{phi-geometric-est-sec7 bis}.
This concludes the derivation of inequality \eqref{expectation-phi-sigma-eq} and therefore of \eqref{expectation-phi-sigma}.

\bigskip

The proof of \eqref{expectation-phi-sigma bis} follows the strategy of Lemma 4.2 in \cite{BGSS2}.
We adapt it to this new framework.

 \begin{Def}[Forward tree] A \emph{forward tree} $T_\prec=  (q_i, \bar q_i)_{1\leq i \leq M^{r \delta} -1}$ is constructed by recording in increasing order of times, denoted by~$t_i$,  the   encounters  of the forward dynamics (recall Definition~{\rm\ref{def:encounter}}) which do not create any cycle (nor multiple edge). These encounters are said to be \emph{clustering}.
\end{Def}
 Note that~$q_i, \bar q_i$ are generic notations for the indexes of the two  particles involved in the~$i$-th encounter, they can of course take several times the same value.

Even though the encounters can be of different nature, they lead to similar geometric constraints  in the forward dynamics and they are coded in the same way in terms of the dual variables. 
The type of each link $(q_i, \bar q_i)$
(with or without annihilation, with or without scattering) is encoded in the set of parameters~$\varsigma, \lambda, \SS, \bar \SS,\KK,\EE$. 
Then,
 $$ \sum_{\varsigma, \lambda, \SS, \bar \SS,\KK,\EE}
\indc_{\{ Z_{\M^{r \delta}} 
\mbox{\ \tiny forward cluster} \} } 
 = 
 \sum_{\varsigma, \lambda, \SS, \bar \SS,\KK,\EE}
\; \sum_{T_\prec \in \cT^\prec_{M^{r \delta}}} 
\indc_{\{ Z_{\M^{r \delta}} \;  \in \; \cR_{T_\prec}^{{\rm comp}}  \}    } ,
$$
where $\cR_{T_\prec}^{{\rm comp}}  $ is the set of  configurations compatible with $( \varsigma, \lambda, \SS, \bar \SS,\KK,\EE,T_\prec )$,
and $\cT^\prec_{M^{r \delta}}$ stands for the set of all ordered trees with $M^{r \delta} -1$ edges.  
The above sum over ordered trees corresponds to a partition, meaning that  for any given~$Z_{\M^{r \delta}} $, at most one term is non zero.
Note, for future reference  (see Section \ref{variance-sec} below), that \eqref{phi-geometric-est-sec7} 
implies
\begin{equation}
\label{eq: phi-geometric-est-sec7 bis}
| \phi^{(\sigma)}(Z_{\M^{r\delta}}) | 
\leq C_P \;  {|\sigma|}^{M^{r\delta}}
\Big(\prod_{j \in \sigma} \| h^{(j)} \|_\infty\Big)
\frac{ \mu_\eps^{M^{r\delta}-1}}{M^{r\delta}!}\,  \sum_{ \bar \SS, \SS, \KK, \varsigma, \lambda, \EE}   \; \sum_{T_\prec \in \cT^\prec_{M^{r \delta}}} 
\indc_{\{ Z_{\M^{r \delta}} \;  \in \; \cR_{T_\prec}^{{\rm comp}}  \}    } \;.
\end{equation}

\smallskip

 We then need to integrate over the variables $Z_{M^{r \delta}}$ restricted to the set~$\cR_{T_\prec}^{{\rm comp}} $. This set is a collection of constraints which are not independent one from the other. However, exploiting the ordering of edges in $T_\prec$, we can identify a sequence of ``independent" variables (see Definition~\ref{Def seq independent} below). The basic idea is that, when we follow the dynamics forward in time, each new edge corresponds to an encounter  involving at least one {\it new} variable. A convenient way to proceed is by using as new variables the relative positions between particles realizing encounters, keeping all velocities fixed.
More precisely, given   an admissible tree $T_\prec$, let us define the relative positions at time~$\theta- r\delta$ 
 \begin{equation} \label{eq:relpos}
 \hat x_i:=x_{q_i}-x_{\bar q_i}\, .
 \end{equation}
Given the relative positions $\left(\hat x_s \right)_{s < i}$ and the velocities $V_{M^{r \delta}}$, we fix a forward flow with clustering encounters at times $t_{ 1}< \dots < t_{i-1} $. By construction, $q_i$ and $\bar q_i$ belong to two forward pseudo-trajectories that have not interacted yet. In other words,~$q_i$ and $\bar q_i$ do not belong to the same connected component in the graph~$G_{i-1} := (q_j,\bar q_j)_{1 \leq j \leq i-1}$. Inside each connected component, relative positions are fixed by the previous constraints, and one degree of freedom remains. Therefore we can vary $\hat x_i$ so that   an encounter at time $t_i $ occurs between $q_i$ and $\bar q_i$ (moving rigidly the corresponding connected components). This encounter condition defines recursively   the sets~$\cB_{T_\prec, i}   (\hat x_{1}, \dots, \hat x_{i-1}, V_{M^{r \delta}})$ prescribing the constraints on $\hat x_i$.
 \begin{Def}
 \label{Def seq independent} 
 We   say that the sets $(\cB_{T_\prec, i})_{i \leq M^{r \delta}-1}  $ are \underbar{sequentially independent} if for all $i$ the set $\cB_{T_\prec, i}$ is defined by constraints depending only on  $\hat x_{1}, \dots, \hat x_{i-1}, V_{M^{r \delta}}$.
\end{Def} 
Suppose that the time $t_i$ belongs to  the set $\cI \in \{\cI_\delta, \cI_\tau,\cI_\theta\}$.
If the particles $q_i$ and $\bar q_i$ move in straight lines, then  the measure of $\cB_{T_\prec, i}$ can be estimated by
\begin{align}
\label{eq: B prec i}
|\cB_{T_\prec, i} | \leq \frac{C}{\mu_\eps} |v^\e_{q_i}(t^+_{i-1}) - v^\e_{\bar q_i}(t^+_{i-1})| 
 \int \indc_{ t_i \in \cI } \;  \indc_{t_i \geq t_{i-1}} dt_i\,  .
\end{align}
Thus, by a Cauchy-Schwarz inequality there holds
\begin{equation}
\label{eq: borne somme qi}
\begin{aligned}
\sum_{q_i,\bar q_i} |\cB_{T_\prec, i} | 
\leq \frac{C}{\mu_\eps} \left( V_{M^{r \delta}}^2 + M^{r \delta} \right) M^{r \delta}  \int \indc_{ t_i \in \cI }  \; \indc_{t_i \geq t_{i-1}} dt_i\,.
\end{aligned}
\end{equation}
Note however that, by definition of the forward tree, particles  $q_i, \bar q_i$ may  encounter on $[t_{i-1}, t_i]$ with other particles from their respective connected component in  the graph~$G_{i-1}$. In this case the particle trajectories are piecewise affine but a bound similar to \eqref{eq: B prec i} is obtained by summing over all the portions of the trajectory
  and the upper bound \eqref{eq: borne somme qi} still holds (see Section 8.1 of \cite{BGSS3} for details).

\medskip

  At this point we proceed as in \cite{BGSS2} (the proof of Lemma 4.1 therein contains the same computation that follows, except for the time sampling condition $t_i \in \cI$ appearing in \eqref{eq: B prec i}-\eqref{eq: borne somme qi}).

 We first apply the change of variables
 $$
  X_{M^{r \delta}-1 }  \longrightarrow  \hat X_{M^{r \delta}-1} 
 $$
 and observe that, for any fixed $x_{M^{r \delta}}$, this is a map of translations with $dX_{M^{r \delta} -1}= d\hat X_{M^{r \delta} -1}$.
The constraints, imposed by the $M^{r \delta}-1$ encounters, can be evaluated one after the other following the order prescribed above. 
Hence by Fubini's theorem
\begin{align}
& \sum_{T_\prec \in \cT^{\prec }_{M^{r \delta}} }\int d\hat X_{M^{r \delta} -1} 
\prod_{i=1}^{M^{r \delta} -1} \indc_{\cB_{T_\prec, i }} \leq 
\sum_{T_\prec \in \cT^{\prec }_{M^{r \delta}} } \int d\hat x_1
 \indc_{\cB_{T_\prec, 1 }} \int d\hat x_{ 2}\,  \dots\int d\hat x_{M^{r \delta} -1} \indc_{\cB_{T_\prec, M^{r \delta} -1}}
 \nonumber \\
 & \qquad \qquad \qquad \qquad 
 \leq \left( \frac{C}{\mu_\eps}\right)^{M^{r \delta}-1}  
 \left( V_{M^{r \delta}}^2 + M^{r \delta} \right)^{M^{r \delta}-1} 
 ( M^{r \delta})^{M^{r \delta} -1}
 \int dt_{ 1}\, \dots dt_{M^{r \delta}-1}  \indc_{\rm samp},
 \label{eq: integration Bi}
\end{align}
where $\indc_{\rm samp}$ is the constraint on the encounter times respecting the sampling. 
Retaining only the information on the number of  encounters  in each time interval,
we get by integrating over these ordered times an upper bound of the form
\begin{equation}
\label{eq: time constraints 4.1}
\frac{\delta^{N^r}}{ N^r !} \;  
\frac{\tau^{N^{<r}+|\sigma| -1}}{(N^{<r} + |\sigma| -1)!} \;
\frac{\Theta^{M^{r \delta}}}{ M !} 
\leq  \frac{3^{M^{r \delta}-1}}{(M^{r \delta}-1) !} \, 
\delta^{N^r \, \tau^{N^{<r}+|\sigma| -1}} \,  \Theta^{M^{r \delta}}\; ,
\end{equation}
where we used  the inequality
\begin{equation*}
\frac{(M^{r \delta} -1) !}{N^r ! \, (N^{<r} + |\sigma| -1)! \, M !} \leq  3^{M^{r \delta}-1}.
\end{equation*}
Up to a factor $C^{M^{r \delta}}$,  the  factorial $(M^{r \delta} -1) !$ compensates the factor  $(M^{r \delta})^{M^{r \delta}}$ in \eqref{eq: integration Bi}.
Furthermore,  in any dimension, for any $R,N$ 
\begin{equation}
\label{eq: inegalite exponentielle}
\sup_{ V } \Big\{ \exp \big( - \frac18 |V|^2 \big)  \; ( |V|^2 + R)^N  \Big\} \leq C^N e^R \; N^N.
\end{equation}
After integrating the velocities with respect to the measure~$\cM^{\otimes M^{r \delta}}$, we deduce from the previous inequality that the term $\left( V_{M^{r \delta}}^2 + M^{r \delta} \right)^{M^{r \delta}}$ gives another factor of order $(M^{r \delta})^{M^{r \delta}}$ which leads, up to a factor $C^{M^{r \delta}}$, to  the term $M^{r \delta} !$ in \eqref{expectation-phi-sigma bis}.
This completes the proof of~\eqref{expectation-phi-sigma bis} and therefore of the inequality \eqref{expectation-phi-sigma-eq}.
  \qed

 \bigskip
 
 We turn now to the proof of Inequality~\eqref{expectation-phi-rec-eq}. We proceed as for \eqref{expectation-phi-sigma bis} and our purpose is to show that
\begin{equation}
\label{expectation-phi-rec bis} 
\begin{aligned}
&  \sup_{x_{M^{r\delta}} \in \T^d}\int  \indc_{\{  Z_{\M^{r\delta} } \mbox{\ \tiny non minimal forward cluster} \}}
 \cM^{\otimes M^{r\delta}}\left( V_{M^{r\delta}}\right)
dX_{M^{r\delta}-1 }dV_{M^{r\delta} } \\
& \qquad \qquad 
 \leq C { M^{r\delta} ! \over \mu_\eps^{M^{r\delta} - 1}} 
 \eps\delta  |\log \eps| (\Theta |\log \eps|) ^{2d+4}  \;  (C \Theta)^{ M - |\sigma|}  (C \delta)^{ (N^r-1)_+}  
 (C \tau)^{(N^{<r}+ | \sigma | - 2)_+} \, ,
 \end{aligned}
\end{equation}
for some pure constant $C>0$, uniformly in $(\bar \SS, \SS, \KK, \varsigma, \lambda, \EE)$.

By construction, there is at least one encounter violating the minimality, and therefore at least one clustering in $\cI_\delta$.
 The support estimate \eqref{expectation-phi-sigma bis} can be refined by using that the graph encoding all encounters has strictly more than $(M^{r\delta} - 1)$ edges (i.e.\,it strictly contains the forward tree $T_\prec$),  which means that there will be at least one cycle in this graph. This  reinforces one of the geometric  conditions on the sets $\cB_{T_\prec, i} $ (see \eqref{eq: B prec i}), leading to the following estimate
\begin{equation}
  \label{eq: integration Bi'}
\begin{aligned}
& \sum_{T_\prec \in \cT^{\prec}_{M^{r \delta}}}   
\int \cM^{\otimes M^{r \delta}}(V_{M^{r \delta}}) \,  dV_{M^{r \delta}} \int d\hat X_{M^{r \delta}-1} 
\indc_{\rm cycle}} \prod_{i=1}^{M^{r \delta}-1} \indc_{\cB_{T_\prec, i }\\
& \quad  \leq \left( \frac{C}{\mu_\eps}\right)^{M^{r \delta}-1} \eps |\log \eps| (\bbV  \Theta)^{2d+4} 
(M^{r \delta})^{2 + 2  M^{r \delta}} 
 \int_{ \theta-r\delta}^{{ \theta-(r-1)\delta } }dt_{ 1}\, \dots \int_{t_{M^{r \delta}-2} }^{\Theta} d t_{M^{r \delta}-1} 
 \,\indc_{\rm samp}\,.
\end{aligned}
\end{equation}
We refer to \cite{BGSS2} for the proof of this estimate (see Eq.\,(5.12) in \cite{BGSS2}, which is derived under  the same assumptions on the sets $\cB_{T_\prec, i }$ except for the minor difference in the time sampling condition $t_i \in \cI$).
We recall the choices~(\ref{eq: choix parametres}) for~$\bbV $ and~$ \Theta$. Then
integrating over the simplex in time represented by $\indc_{\rm samp}$ leads to \eqref{expectation-phi-rec bis},
where the contribution $\delta$ comes from the first edge of the forward tree. Since we did not track the nature of this edge,  the terms  $( N^r -1)_+$ and  $(N^{<r} +|\sigma| - 2)_+$ have been adjusted to take into account all the possibilities. 

This concludes the proof of~\eqref{expectation-phi-rec bis}, hence of~\eqref{expectation-phi-rec-eq}.  \qed


\section{Expectation and variance of $\Otimes$-products }
\label{variance-sec}

The aim of this section is to control the expectation and variance of $\Otimes$-products in order to complete the proofs of Lemma \ref{phi-sigma-lemma} and Lemma \ref{phi-rec-lemma}.

Without loss of generality, we suppose from now on that the sets $\sigma_i$ are indexed by~$i \in \{1,\dots,q\}$. 

We start by proving the estimates on  the equilibrium measures (in Paragraphs~\ref{sec: Expectation of centered Otimes-products} and~\ref{variance}), and then show in Paragraph~\ref{sec: Conditioned products}
how to conclude to Estimates~(\ref{E-HM 1}),~(\ref{variance-phi-sigma}) and~(\ref{variance-phi-rec}).

\subsection{Expectation of centered $\Otimes$-products }
\label{sec: Expectation of centered Otimes-products}

In this section, we prove the following inequality:
\begin{equation}
\label{E-HM 1-eq}  
\left|\bbE_\eps ^{\rm eq} \left[ \Otimes_{i=1}^q  \zeta^{\eps,\rm eq}_{M_{i}^{r\delta}   } 
\big( \phi^{(\sigma_i) } \big)  \right]  \right|
\leq C_q \eps  \prod_{i = 1} ^q 
  (|\sigma_i|K_\gamma) ^{M_{i}^{r\delta}}\times \Big( (C \Theta)^{M_{i} - |\sigma_i|}  (C \delta)^{ N_{i }^{r}}  (C \tau)^{N_{i}^{<r}+|\sigma_i | - 1} \Big)\;.
\end{equation}

Inequality (\ref{E-HM 1-eq}) follows from the control on the structure of the test functions~$\phi^{(\sigma_i) }$ given by Eq.\,\eqref{phi-geometric-est-sec7}. Notice that in the latter estimate, the function
on the right-hand side is invariant by translations (simultaneous of the~$M_{i}^{r\delta} $ particles in the space $\T^d$), and bounded in an~$L^1-$weighted norm (by \eqref{expectation-phi-sigma bis}). Using these two ingredients we shall now prove that
\begin{equation}
 \begin{aligned}
& \left|\bbE_\eps ^{\rm eq} \left[ \Otimes_{i=1}^q  \zeta^{\eps,\rm eq}_{M_{i}^{r\delta}   } 
\big( \phi^{(\sigma_i) } \big)  \right]  \right|
  \leq C_q \,\e\,\prod_{i=1}^q
\frac{\left( C |\sigma_i|\mu_\e\right)^{M_{i }^{r\delta} -1}}{M_i^{r\delta}!}\\
& \qquad\qquad\qquad\qquad \times
\sum_{ \bar \SS, \SS, \KK, \varsigma, \lambda, \EE} \;\sup_{x_{M_i^{r\delta}} \in \T^d}\int  \indc_{\{  Z_{\M_{i }^{r\delta} } \mbox{\ \tiny forward cluster} \}}
 \cM^{\otimes  M_{i }^{r\delta} }  dX_{M_{i }^{r\delta}-1 } dV_{M_{i }^{r\delta} }  \end{aligned}
 \label{E-HM}
\end{equation}
where the sum is taken over the collection $\left(\varsigma, \lambda, \KK, \SS,\bar\SS,  \EE\right)$ parametrising the clusters and  $\M_{i }^{r\delta}$ codes the cardinalities of the blocks in $\sigma_i$ (as in \eqref{phi-geometric-est-sec7}).
The small factor $\e$ will be obtained by a cluster expansion of the invariant measure, tracing the small correlations between different fluctuation fields. Eq.\,(\ref{E-HM 1-eq}) follows then from \eqref{E-HM} and \eqref{expectation-phi-sigma bis}.

In order to establish (\ref{E-HM}), we  first use the definition of $\Otimes$-product given by~\eqref{tensor-product-eq}
and write
$$
 \Otimes _{i=1} ^q \left(   {1\over \mu_\eps^{M_{i}^{r\delta}} }\sum \phi^{(\gs_i)}  - \bbE_\eps^{\rm eq} [\phi^{(\gs_i)}]  \right) = \sum_{A \subset \{1,\dots, q\} }  \pi^\eps_{M^{r\delta} _A}\left(\Phi_{{\mathbf M}^{r\delta}_A}\right)   \prod_{ j \in A^c} \bbE_\eps^{\rm eq}  [-\phi^{(\gs_j)}]\,,  $$
with $M^{r\delta} _A =\sum_{ j \in A}M_{j}^{r\delta}$, ${\mathbf M}^{r\delta}_A = (M_{j}^{r\delta})_{j \in A}$ and $\Phi_{{\mathbf M}^{r\delta}_A} :=\otimes _{j \in A} \phi^{(\gs_j)} $ and where $A^c$ is the complement of $A$ in $\{1,\dots,q\}$. Then,
$$\begin{aligned} 
&\bbE_\eps^{\rm eq}  \left[  \Otimes_{i=1}^q \zeta^{\eps,\rm eq}_{M_{i }^{r\delta} } 
\left(\phi^{(\sigma_i) } \right)  \right]   =  \frac{1}{\cZ^\eps} \mu_\eps^{q/2}  \sum_{A \subset \{1,\dots, q\} }   \prod_{ j \in A^c} \bbE_\eps^{\rm eq}  [-\phi^{(\gs_j)}] \\
     & \qquad \qquad  \times 
  \sum_{p\geq 0 }  \frac{\mu_\eps^{p}}{p!} 
   \int dZ_{M^{r\delta} _A}  d\bar Z_p   \,\Phi_{{\mathbf M}^{r\delta}_A} (Z_{M^{r\delta} _A}) \indc_{\cD^\eps_{{ M}^{r\delta}_A+p}} (Z_{M^{r\delta} _A} ,\bar Z_p)  \cM^{\otimes ({  M}^{r\delta}_A +p) } (V_{M^{r\delta} _A} ,  \bar V_p)\, .
    \end{aligned}
$$
We decompose $Z_{{ M}^{r\delta}_A}$ in $|A|$ subconfigurations  $(Z^{(i)}_{M_{i }^{r\delta} } )_{i\in A} $ (each one containing possibly several blocks).
We then use a cluster expansion of the exclusion~$\indc_{\cD^\eps_{{ M}^{r\delta}_A  +p}}$, representing each 
 $Z^{(i)}_{M_{i }^{r\delta} } $  by one vertex, and $(\bar z_j)_{1\leq j\leq p}$ as $p$ separate vertices. We denote by $d(y,y^*)$ the minimum relative distance (in position) between elements~$y,y^*\in S_{|A|+p} :=  \{(Z^{(i)}_{M_{i }^{r\delta} } )_{i\in A} , \bar z_1, \dots \bar z_{p} \}$. 

We  define the cumulants 
$$\gp(Z_{{ M}^{r\delta}_A} , \bar z_1, \dots \bar z_{p}) :=
\sum_{G \in \cC_{S_{|A|+p}}} \prod_{\{y,y^*\}\in E(G)}(- \indc_{d(y,y^*) \leq \eps})\,,
$$
and more generally for any subpart of $Y \subset S_{|A|+p}$
\begin{equation}
\label{cumulant-def}
\gp(Y) =
\sum_{G \in \cC_Y} \prod_{\{y,y^*\}\in E(G)}(- \indc_{d(y,y^*) \leq \eps})\,,
\end{equation}
denoting by $\cC_Y$ the connected graphs with vertices in $Y$ and by~$E(G)$ the edges of the graph~$G$.
Notice that this definition is analogous to the one used to treat the dynamical correlations in \eqref{def-phirho}, but now the exclusion is static
$$
\prod_{ y \neq y^* \atop y,y^* \in Y} \indc_{d(y,y^*) > \eps}
=  \sum_{G \in \GG_{Y}} \prod _{\{y, y^*\} \in E(G)} (- \indc_{d(y,y^*) \leq \eps} ) =\sum_{\rho \in \cP_{Y}  }
\prod_{q=1}^{{ |\rho|}} \varphi (\rho_q) \, ,
$$
 where $\varphi (\rho_q)$ is defined by  \eqref{cumulant-def}.

We then have the following cumulant expansion
$$
\begin{aligned}
\indc_{\cD^\eps_{{ M}^{r\delta}_A +p} }\left(Z_{{ M}^{r\delta}_A} ,\bar Z_p\right)   &= \left(\prod_{i \in A}  \indc_{\cD^\eps_{M_{i }^{r\delta} }} (Z^{(i)} _{M_{i }^{r\delta} } )\right)    
\left(\prod_{ y,y^*\in S_{|A|+p} \atop y\neq y^*} \indc_{ d(y,y^*) >\eps} \right)\left(Z_{{ M}^{r\delta}_A} ,\bar Z_p\right) \\
&=  \left(\prod_{i \in A}  \indc_{\cD^\eps_{M_{i }^{r\delta} }} (Z^{(i)} _{M_{i }^{r\delta} } )  \right)
\sum_{  \bar\gs_0\subset\{1,\dots,p\} }
\indc_{\cD^\eps_{|\bar\gs_0|}} (\bar Z_{\bar\gs_0})   \sum_{\eta \in \cP_{A}}
\sum_{
\substack{
\bar\gs_1,\dots,\bar\gs_{|\eta|} \\
\cup_{i=0}^{|\eta|} \bar\gs_i = \{1,\dots,p\} \\
\bar\gs_i \cap \bar\gs_{i'} = \emptyset, i \neq i'
}
}\,
\prod_{i=1}^{|\eta|}\gp(Z_{{\eta_i}},\bar Z_{\bar\gs_i}),
\end{aligned}
$$
 where $\cP_{A}$ is the set of partitions of $A$, and 
 $Z_{{\eta_i}} = \left( Z^{(j)}_{ M_{j}^{r\delta}}\right)_{j \in \eta_i}$.
Note that the~$\bar \sigma_i$ may be empty (in particular all $\bar \sigma_i$ are empty if~$| \bar\gs_0| = p$).
 Using the symmetry in the exchange of particle labels, we get, denoting~$n_i :=|\bar\gs_i|$,
$$ 
 \binom{p}{ n_0}
 \binom{p-n_0}{ n_1}   \dots  \binom{p-n_0- \dots-n_{|\eta|-1}}{n_{|\eta|}}    =  {p! \over n_0 ! \; \dots \; n_{|\eta|} !}
$$
choices for the repartition of the background particles $\bar Z_p$.

Then, using the definition of the partition function $\cZ^\eps$, we obtain
\begin{align}
  \frac{1}{\cZ^\eps} \sum_{p\geq 0 }  & \frac{\mu_\eps^{p}}{p!} \int  d\bar Z_p \, \indc_{\cD^\eps_{{ M}^{r\delta}_A +p}}  \left(Z_{{ M}^{r\delta}_A} ,\bar Z_p\right)\cM^{\otimes p } (\bar V_{p} ) \nonumber \\
 &=   \frac{1}{\cZ^\eps}\left(\prod_{i \in A}  \indc_{\cD^\eps_{M_{i }^{r\delta} }} (Z^{(i)} _{M_{i }^{r\delta} } )  \right)
 \sum_{\eta \in \cP_{A}}
\sum_{p \geq 0} 
\sum_{\substack{n_0,\dots,n_{|\eta|} \geq 0\\ \sum n_i = p}}
\left( \frac{\mu_\e^{n_0}}{n_0 !} \int \cM^{\otimes n_0}  \indc_{\cD^{\eps}_{n_0 }}\left(\bar Z_{n_0}\right) d\bar Z_{n_0}\right)  
\nonumber\\
& \nonumber
\qquad \qquad \qquad \qquad \qquad \qquad \times\prod_{i=1}^{|\eta|} \frac{\mu_\e^{n_i}}{n_i !}\int \cM^{\otimes n_i} \gp(Z_{{\eta_i}},  \bar Z_{n_i})d \bar Z_{n_i}\\
&=  \left(  \prod_{i \in A}  \indc_{\cD^\eps_{M_{i }^{r\delta} }} (Z^{(i)} _{M_{i }^{r\delta} } )  \right) 
 \sum_{\eta \in \cP_{A}}
\prod_{\ell =1}^{|\eta|}   \left( \sum_{n_\ell \geq 0} \frac{\mu_\e^{n_\ell}}{n_\ell !}\int \cM^{\otimes n_\ell} \gp(Z_{{\eta_\ell}},  \bar Z_{n_\ell})d \bar Z_{n_\ell}\right) \, .
\label{eq: cluster expansion}
\end{align}
By Fubini's equality, we finally get that
$$\begin{aligned} 
   &  \bbE_\eps^{\rm eq}  \left[  \Otimes_{i=1}^q \zeta^{\eps,\rm eq}_{M_{i }^{r\delta} } 
\left(\phi^{(\sigma_i) } \right)  \right] 
      =  \mu_\eps ^{q/2} \sum_{A\subset \{1, \dots, q\}}  \left( \prod_{ j \in A^c} \bbE_\eps^{\rm eq}  [-\phi^{(\gs_j)}] \right)    \\
      &\ \ \ \ \ \ \ \ \ 
      \times \sum_{\eta \in \cP_{A}}   \prod_{\ell = 1} ^{|\eta|}  \left[  \sum_{n_\ell \geq 0} \frac{\mu_\e^{n_\ell}}{n_\ell !}  \int \cM^{\otimes  \left( M^{r\delta}_{\eta_\ell}+n_\ell\right)}  \gp(Z_{{\eta_\ell}},  \bar Z_{{n_\ell}})   
     \left(\prod _{j \in \eta_\ell} \phi^{(\gs_j)}\indc_{\cD^\eps_{M_{j}^{r\delta}}}\left( Z^{(j)}_{ M_{j}^{r\delta}}\right)\right)  \,d \bar Z_{n_\ell} dZ_{ \eta_\ell} \right]\, .
       \end{aligned}
$$

By definition, if  one part $\eta_\ell$ is a singleton, say $\{j\}$, we find that the corresponding factor of the product is
(using again Eq.\,\eqref{eq: cluster expansion} with $A = \{j\}$)
 \begin{equation}
 \label{eq:singleton}
 \sum_{n_\ell \geq 0} \frac{\mu_\e^{n_\ell }}{n_\ell  !}\int \cM^{\otimes \left( M_{j}^{r\delta}+n_\ell\right) } \gp(Z_{M_{j}^{r\delta}},  \bar Z_{n_\ell })  \,\phi^{(\gs_j)}\indc_{\cD^\eps_{M_{j}^{r\delta}}}\left( Z_{M_{j}^{r\delta}}\right)\,d \bar Z_{n_\ell } dZ_{M_{j}^{r\delta}} = \bbE_\eps^{\rm eq} \big[  \phi^{(\gs_j)}\big] \,. 
 \end{equation}We will therefore  split any partition $\eta$ of $A$ in a union of singletons $\{j\}$ for $j \in A\setminus B$, and a partition $\tilde \eta$ of  $B$ with no singleton.
In particular, we have that $\tilde \eta$ has a  number of parts~$|\tilde\eta|  \leq \frac12  |B| $. Thus absorbing the sum over singletons  as in \eqref{eq:partnosing}, we get that 
\begin{equation}
\begin{aligned} 
&     \bbE_\eps^{\rm eq}  \left[  \Otimes_{i=1}^q \zeta^{\eps,\rm eq}_{M_{i }^{r\delta} } 
\left(\phi^{(\sigma_i) } \right)  \right]     =  \mu_\eps^{q/2} \sum_{B\subset \{1, \dots, q\}} \sum_{A \subset B^c} (-1)^{|B^c|-|A|} \left( \prod_{ j \in B^c} \bbE_\eps ^{\rm eq}  [ \phi^{(\gs_j)}] \right) \\
   &\qquad	\times \sum_{ \eta \in   ( \cP_{B })^* }   \prod_{\ell = 1} ^{|\eta|}  \left[ \sum_{n_\ell \geq 0} \frac{\mu_\e^{n_\ell}}{n_\ell !}  \int \cM^{\otimes  \left(M^{r\delta}_ {\eta_\ell}+n_\ell\right)}  \gp(Z_{{\eta_\ell}},  \bar Z_{{n_\ell}})  \left( \prod _{j \in \eta_\ell} \phi^{(\gs_j)}\indc_{\cD^\eps_{M_{j}^{r\delta}}}\left( Z^{(j)}_{M_{j}^{r\delta}}\right)\right)  \,d \bar Z_{n_\ell} dZ_{\eta_\ell} \right]\\
      & =   \mu_\eps^{q/2}  \sum_{ \eta \in   \left( \cP_{\{1,\dots,q\}}\right)^* }   \prod_{\ell = 1} ^{|\eta|}  \left[ \sum_{n_\ell \geq 0} \frac{\mu_\e^{n_\ell}}{n_\ell !}  \int \cM^{\otimes  \left(M^{r\delta} _{\eta_\ell}+n_\ell\right)}  \gp(Z_{{\eta_\ell}},  \bar Z_{{n_\ell}})   \left(\prod _{j \in \eta_\ell} \phi^{(\gs_j)}\indc_{\cD^\eps_{M_{j}^{r\delta}}}\left( Z^{(j)}_{M_{j}^{r\delta}}\right) \right) d \bar Z_{n_\ell} dZ_{\eta_\ell} \right]
\end{aligned}
\label{eq:abssinglexp}
\end{equation}
where $( \cP_A)^*$ stands for the partitions without singletons of a set $A$.

Recall that the cumulants defined by (\ref{cumulant-def}) can be controlled by  the tree inequality (see e.g.\,\cite{Pe67,PU09}) 
\begin{equation}
 \label{eq:treeineq}
 \left| \gp(Y)  \right| \leq \sum _{T\in \cT_{Y} } \prod_{\{y,y^*\} \in E(T)}  \indc_{ d(y,y^*) \leq \eps}\,,
 \end{equation}
 where $\cT_{Y}$ is the set of minimally connected graphs (trees) with  vertices in $Y$.
 Thus inside each connected component $\eta_\ell$,  a tree connects the $|\eta_\ell|$ 
 vertices $Z^{(j)} _{M_{j}^{r\delta}} $  and the $n_\ell$ background particles (where each edge corresponds to the distance being smaller than $\eps$).

For a given tree $T$, let $d_1,\dots, d_{|\eta_\ell|+n_\ell}$ be the degrees of the graph (number of edges per vertex).  Integrating with respect to $\bar Z_{n_\ell},Z_{\eta_\ell}$ leads to
\begin{equation}
\begin{aligned}
&\int \cM^{\otimes  \left(M^{r\delta} _{\eta_\ell}+n_\ell\right)}  \prod_{\{y,y^*\} \in E(T)}  \indc_{ d(y,y^*) \leq \eps}   \left(\prod _{j \in \eta_\ell}\left| \phi^{(\gs_j)}\right|\indc_{\cD^\eps_{M_{j}^{r\delta}}}\left( Z^{(j)}_{M_{j}^{r\delta}}\right)\right)  \,d \bar Z_{n_\ell} dZ_{\eta_\ell} \\ 
& \quad  \leq  
C_q\,
(C' \eps^d)^{|\eta_\ell|+n_\ell- 1}    
\prod_{j\in \eta_\ell} (M^{r\delta}_{j})^{d_j} \\
&\qquad \times
\prod_{ j\in \eta_\ell} \frac{\left( C |\sigma_j|\mu_\e\right)^{ M_{j}^{r\delta} -1}}{ M_{j}^{r\delta} !}
\sum_{ \bar \SS, \SS, \KK, \varsigma, \lambda, \EE}
\;\sup_{x_{M_j^{r\delta}} \in \T^d}
\int  \indc_{\{  Z_{\M_{j}^{r\delta}} \mbox{\ \tiny forward cluster} \}}
 \cM^{\otimes  M_{j}^{r\delta}}  dX_{M_{j}^{r\delta}-1}dV_{M_{j}^{r\delta}}\;,
\end{aligned}
\label{eq: facteur L1}
\end{equation}
where  the constant $C'>0$ depends only on $d$.

To justify \eqref{eq: facteur L1}, we first notice that for each $j$ in  $\eta_\ell$, the subconfiguration $X^{(j)}_{M_{j}^{r\delta}}$ covers a volume of order $M_{j}^{r\delta} \eps^d$. 
Thus overlapping two such configurations indexed by $j, j'$ leads to a factor 
$M^{r\delta}_j M^{r\delta}_{j'} \eps^d$,  and overlapping the subconfiguration $j$ and a single particle to a factor 
$M^{r\delta}_j \eps^d$,  while overlapping two particles leads to a simple factor $\e^d$.
Therefore, overall each edge of the tree brings a factor $\e^d$, and
each subconfiguration  $X^{(j)}_{M_{j}^{r\delta}}$ brings a factor $M^{r\delta}_j$ per edge attached
to the vertex $j$ of the tree.
Furthermore the integral in the last line is a consequence of \eqref{phi-geometric-est-sec7} and of the translation invariance of the indicator functions of forward clusters.
Indeed the tree $T$ encoding the static overlaps imposes a geometrical constraint only on the position of a single particle in each 
$X^{(j)}_{M_{j}^{r\delta}}$, say $x^{(j)}_{M_{j}^{r\delta}}$. Therefore by Fubini, we can first fix the variables $\left(\hat  X^{(j)}_{M_{j}^{r\delta}-1},V^{(j)}_{M_{j}^{r\delta}}\right)$ defined by (\ref{eq:relpos}) in such a way  that the dynamical constraints are satisfied,  integrate  the variables $\left(\bar X_{n_\ell},\left(x^{(j)}_{M_{j}^{r\delta}}\right)_j\right)$ according to the tree structure, and then integrate with respect to $\left(\hat  X^{(j)}_{M_{j}^{r\delta}-1},V^{(j)}_{M_{j}^{r\delta}}\right)$ for all $j$. This leads to \eqref{eq: facteur L1}.

{\color{black}
There are $\left( n-2\right)! / \prod_j \left( d_j-1\right)!$
 trees of size $n$ with specified vertex degrees (see e.g.\,Lemma 2.4.1 in \cite{BGSS3}), so that 
 summing \eqref{eq: facteur L1} over all trees leads to the combinatorial factor
 $$ ( | \eta_\ell| + n_\ell -2)! \sum_{d_1,\cdots, d_{|\eta_\ell|+n_\ell} \geq 1}
\left(\prod_{j \in \eta_\ell} \frac{(M^{r\delta}_{j})^{d_j}}{{\left( d_j-1\right)!}} \right)
\left( \prod_{j \notin \eta_\ell} \frac{1}{{\left( d_j-1\right)!}} \right) 
\leq |\eta_\ell|! n_\ell! 2^{|\eta_\ell|+n_\ell}
\left(\prod_{j \in \eta_\ell}M^{r\delta}_{j}e^{M^{r\delta}_{j}}\right)e^{n_\ell}\;.
  $$
Thus, enlarging the constants $C_q, C, C'$ from line to line and recalling that 
$M^{r \delta} = \sum_{i =1}^q M^{r \delta}_i$, we deduce that 
\begin{equation}
\label{eq: facteur L1'''}
\begin{aligned} 
&  \left| \bbE_\eps ^{\rm eq}  \left[  \Otimes_{i=1}^q \zeta^{\eps,\rm eq}_{M_{i }^{r\delta} } 
 \left(\phi^{(\sigma_i) } \right)  \right] \right|
    \  \leq   C_q\, C'^{M^{r\delta}}\mu_\eps^{q/2}    \,\sum_{ \eta \in   \left( \cP_{\{1,\dots,q\}}\right)^* }\,  \prod_{\ell = 1} ^{|\eta|} 
    (C' \eps^d)^{|\eta_\ell|- 1}     
 \sum_{n_\ell \geq 0}  (C' \mu_\e   \eps^d)^{n_\ell} \\
&  \qquad \quad                \;\;\times \prod_{i=1}^q
\frac{\left( C |\sigma_i| \mu_\e\right)^{M_{i }^{r\delta} -1}}{M_{i }^{r\delta} !}
 \sum_{ \bar \SS, \SS, \KK, \varsigma, \lambda, \EE}\sup_{x_{M_i^{r\delta}} \in \T^d}\int \indc_{\{  Z_{M_{i }^{r\delta} } \mbox{\ \tiny forward cluster} \}}
 \cM^{\otimes  M_{i }^{r\delta} }  dX_{M_{i}^{r\delta}-1}dV_{M_{i}^{r\delta}}\\ 
&\quad      
\leq  C_q \, C'^{M^{r\delta}} \big(  \mu_\e \eps^d \big)^{q/2} \\
&  \qquad \quad                \;\;\times
\prod_{i=1}^q\frac{\left( C |\sigma_i| \mu_\e\right)^{M_{i }^{r\delta} -1}}{M_{i }^{r\delta} !}
\sum_{ \bar \SS, \SS, \KK, \varsigma, \lambda, \EE} \sup_{x_{M_i^{r\delta}} \in \T^d}\int  \indc_{\{  Z_{M_{i }^{r\delta} } \mbox{\ \tiny forward cluster} \}}
 \cM^{\otimes  M_{i }^{r\delta} }  dX_{M_{i}^{r\delta}-1}dV_{M_{i}^{r\delta}} \\
 &   \quad     \leq C_q \,\, C'^{M^{r\delta}}\e\\&  \qquad \quad                \;\;\times\prod_{i=1}^q
\frac{\left( C |\sigma_i| \mu_\e \right)^{M_{i }^{r\delta} -1}}{M_{i }^{r\delta} !}
\sum_{ \bar \SS, \SS, \KK, \varsigma, \lambda, \EE}\sup_{x_{M_i^{r\delta}} \in \T^d} \int  \indc_{\{  Z_{M_{i }^{r\delta} } \mbox{\ \tiny forward cluster} \}}
 \cM^{\otimes  M_{i }^{r\delta} }  dX_{M_{i}^{r\delta}-1}dV_{M_{i}^{r\delta}}\;,
     \end{aligned}
\end{equation}
where in the second inequality we used that~$\eps$ is small to sum the series
and that for $|\eta|  \leq q/2$, $q \geq 2$
\[
 \mu_\eps^{q/2} \prod_{\ell = 1} ^{|\eta|}  \big( \eps^d \big)^{|\eta_\ell|- 1} 
 =   \mu_\eps^{q/2} \big( \eps^d \big)^{q-|\eta| } 
 \leq  \big( \mu_\eps \eps^d \big)^{q/2}\;.
\]
Equation \eqref{E-HM} is proved.
\qed
}

\subsection{Variance of   $\Otimes$-products }
\label{variance}

We  aim at proving the bound~\eqref{variance-phi-sigma}, let us compute
\begin{equation}
\label{eq: var prel}
\begin{aligned}
\bbE_\eps^{\rm eq}\Big[ \Big( \Otimes_{i=1}^q \zeta^\eps_{M_{i}^{r\delta}} 
\big(  \phi^{(\sigma_i) } \big) \Big)^2 \Big] & = 
\bbE_\eps^{\rm eq}  \left[ \mu_\eps ^{q }  \left(\Otimes _{i=1} ^q  \left(   {1\over \mu_\eps^{M_{i}^{r\delta}} }\sum \phi^{(\sigma_i)}  - \bbE_\eps [\phi^{(\sigma_i)}]  \right) \right)^2\right] \\
&    =  \sum_{A\subset \{1, \dots, q\} \atop A' \subset \{1',\dots, q'\}}  
\prod_{ j \in A^c  \cup \left(A'\right)^c } \bbE_\eps  [-\phi^{(\sigma_j)}]
 \times   \sum_{\ell = 0} ^{\min ({ M}^{r\delta}_A, M^{r\delta}_ {A'})}  \cE_{ A, A', \ell} 
\end{aligned}
\end{equation}
where the joint expectation with $\ell$ repeated indices is denoted by 
$$
\begin{aligned} 
\cE_{ A, A', \ell} := 
     &   {1\over \mu_\eps^{{ M}^{r\delta}_A+M^{r\delta}_{A'}-q}} \bbE^{\rm eq}_\eps \left[ \sum _ {\big| ( i_1, \dots , i_{{ M}^{r\delta}_A} ) \cap ( i'_1, \dots, i'_{M^{r\delta} _{A'}} )\big| = \ell  } 
 \Phi_{{\mathbf M}^{r\delta}_A} \big({\mathbf z}^{\eps}_{i_1}, \dots , {\mathbf z}^{\eps}_{i_{{ M}^{r\delta}_A} } \big)  \Phi_{{\mathbf M}^{r\delta}_{A'}} \big({\mathbf z}^{\eps}_{i'_1}, \dots , {\mathbf z}^{\eps}_{i'_{M^{r\delta} _{A'}}} \big)\right]  \\
\end{aligned}
$$
with notations introduced above, which we recall again:
${ M}^{r\delta}_A =\sum_{ j \in A} M_{j}^{r\delta}$, ${\mathbf M}^{r\delta}_A = (M_{j}^{r\delta})_{j \in A}$, $\Phi_{{\mathbf M}^{r\delta}_A}  = \bigotimes_{j \in A} \phi^{(\sigma_j)}$ {	\color{black} and
${ M}^{r\delta}_{A'} =\sum_{ j \in A'} M_{j}^{r\delta}$, ${\mathbf M}^{r\delta}_{A'} = (M_{j}^{r\delta})_{j \in A'}$,
$\Phi_{{\mathbf M}^{r\delta}_{A'}}  = \bigotimes_{j \in A'} \phi^{(\sigma_j)}$.
Denoting by $\Lambda$ and $\Lambda'$ the subsets of indices selecting the $\ell$ contracted variables},  we get 
\begin{equation}
\label{eq: E A A'}
\begin{aligned}
&\cE_{ A, A', \ell}  = 
 \frac{\mu_\eps^{q- \ell}}{\cZ^\eps} 
 \sum_{{{\Lambda \subset \{1, \dots, { M}^{r\delta}_A\} } \atop {\Lambda' \subset \{1, \dots, M^{r\delta} _{A'} \} }} \atop { |\Lambda|= |\Lambda'| = \ell }} 
 \; \sum_{\chi _\ell: \Lambda \mapsto \Lambda'}
\;  \sum_{p\geq 0 }  \frac{\mu_\eps^{p}}{p!}
 \int dZ_{{ M}^{r\delta}_A} dZ'_{M^{r\delta} _{A'}} d\bar Z_p  \delta_{\chi_\ell} (Z_{{ M}^{r\delta}_A}, Z'_{M^{r\delta} _{A'}})  \\
  & \quad\quad\quad\quad\quad\quad\quad\quad\quad  \times    \Phi_{{\mathbf M}^{r\delta}_{A}} (Z_{M^{r\delta} _{A}})\Phi_{{\mathbf M}^{r\delta}_{A'}} (Z'_{M^{r\delta} _{A'}})  \indc_{\cD^\eps_{{ M}^{r\delta}_A+M^{r\delta} _{A'}-\ell +p}}  \cM^{\otimes ({ M}^{r\delta}_A+M^{r\delta} _{A'}-\ell +p) } ,
\end{aligned}
\end{equation}
where the injective map $\chi_\ell : \Lambda \mapsto \Lambda'$ encodes the repetition of the indices in $Z_{{ M}^{r\delta}_A}, Z_{M^{r\delta} _{A'}}'$
$$
 \delta_{\chi_\ell} = \prod _{j \in \Lambda} \delta_{ z_j - z'_{\chi_\ell(j)}}\,.
$$
A factor $\mu_\eps^{-\ell}$ is gained from these repetitions.

\medskip
\noindent
\underline{Step 1}. 
A graph  structure with {\color{black} ${ M}^{r\delta}_A+M^{r\delta} _{A'}-\ell$} vertices, depicted in Figure \ref{fig-eta}, can be extracted from the constraints 
$\Phi_{{\mathbf M}^{r\delta}_{A}}, \Phi_{{\mathbf M}^{r\delta}_{A'}}, \chi_\ell$ in \eqref{eq: E A A'} :
\begin{itemize}
\item  the dynamical constraints corresponding to $(Z^{(j)}_{M_{j}^{r\delta}})_{j\in A}$ and $(Z^{(j)}_{M_{j}^{r\delta}})_{j\in A'}$,
coded by the functions $\phi^{(\sigma_j)}$ {\color{black} according to \eqref{phi-geometric-est-sec7}}, lead to vertices forming $|A|+|A'|$ connected orange packets;
\item   the constraint $\chi_\ell$ from the $\ell$ repetitions in the variables is represented  by green {\color{black} lines (contractions) in Figure \ref{fig-eta}. The vertices linked by green lines correspond in fact to the same repeated variable, and are therefore identified.}
\end{itemize}
We consider a partition $\eta$ of $A \cup A'$ into $s$ connected components $\eta_1, \dots, \eta_s$ (represented in blue in Figure~\ref{fig-eta}):
in each component~$\eta_i$, all orange packets are connected by green lines. Denote by $\ell_i$ the number of green lines in the component $\eta_i$. By definition
$$\ell_i \geq |\eta_i|- 1, \qquad \sum_{i=1}^s \ell_i = \ell\,,$$
where $|\eta_i|$ is the number of orange packets in $\eta_i$.
We denote by $\delta_{\chi_{\ell_i}}$ the identification of particles restricted to $\eta_i$, so that
 $$
 \delta_{\chi_\ell} = \prod_{i=1}^s \delta_{\chi_{\ell_i}}\;,
 $$
 and recall that
a green line can only join a packet $j \in A$ and a packet $j \in A'$.

 \begin{figure}[h]\centering
\includegraphics[width=3.5in]{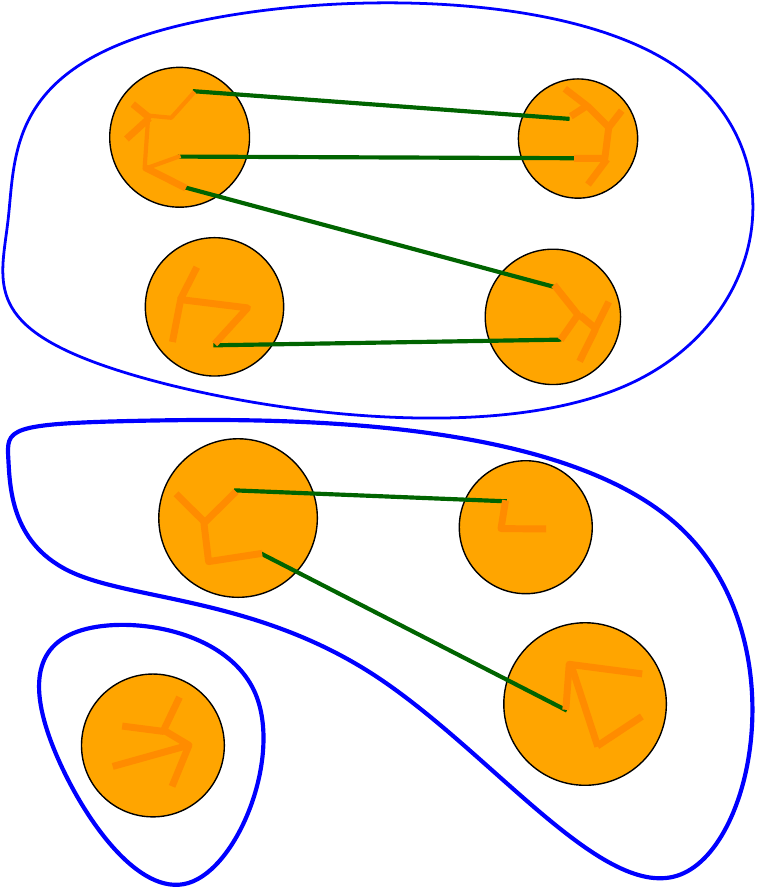} 
\caption{\small 
A partition $\eta$ with $|A|= |A'| =4$ and $s=3$. The first component is $\eta_1=\{1,2,1',2'\}$, the second component is $\eta_2=\{3, 3', 4'\}$ and the third component $\eta_3 = \{4\}$. The number of green lines is $\ell_1 =4$, $\ell_2=2$ and $\ell_3=0$.}
 \label{fig-eta}
\end{figure}

%

Recall that   $Z_{\eta_i} = Z^{(\eta_i)}_{M^{r\delta}_{\eta_i}}$ is  the collection of particles in the connected component $\eta_i$, where $$M^{r\delta}_{\eta_i} = \sum_{j \in A \cap \eta_i}M^{r\delta}_j + \sum_{j \in A' \cap \eta_i}M^{r\delta}_{j} \;.$$
 We should keep in mind now that, as $\ell_i$ particles of $A'$ are identified with particles in $A$, then the total number of particles is actually $M^{r\delta}_{\eta_i}- \ell_i$.

\medskip
\noindent
\underline{Step 2}.
We need now to define new {\color{black} forward} tree graphs to be associated with each component $\eta_i$.
We denote by~$T_{j}$ for~$j\in A $ ($j\in A'$),  the orange forward  tree describing the cluster structure of~$\phi^{(\sigma_j)}$  used in   estimate \eqref{eq: phi-geometric-est-sec7 bis}, coding the geometric constraints on the configuration $Z^{(j)}_{M_{j}^{r\delta}}$ associated with the forward dynamics in terms of minimally connected graphs (we drop here the index $\prec$, but remember that the graphs are equipped with an ordering of edges). We recall that, in such forward dynamics, each configuration of $M_j^{r\delta}$ particles is partitioned in blocks, and that the cardinalities of such blocks are coded in the notation $\M_{j}^{r\delta}$ in \eqref{eq: phi-geometric-est-sec7 bis}; the component $\eta_i$ inherits then a partition in blocks $\M^{r\delta}_{\eta_i}$.
In each connected component $\eta_i$, we extract a minimally connected graph~$T_{\eta_i}$ on the set of vertices  $Z^{(\eta_i)}_{M^{r\delta}_{\eta_i}}$ (equipped with an ordering of edges) by means of the following procedure (see also Figure \ref{fig-T-eta}). 
 We collect first all edges coming from~$T_{j}$ for~$j\in  A \cap \eta_i$ 
 (note that this cannot produce any cycle by definition). This will form the skeleton of the graph and will be denoted  by~$T_{\eta_i}^{1A}$. Then we look in turn at the edges in the remaining orange  forward tree graphs~$T_{j}$ with $j\in A' \cap \eta_i$. Following the ordering, we keep only  edges that do not produce cycles
 after identification of vertices linked by green lines (note that this peeling is unique).
 \begin{figure}[h]\centering
\includegraphics[width=3in]{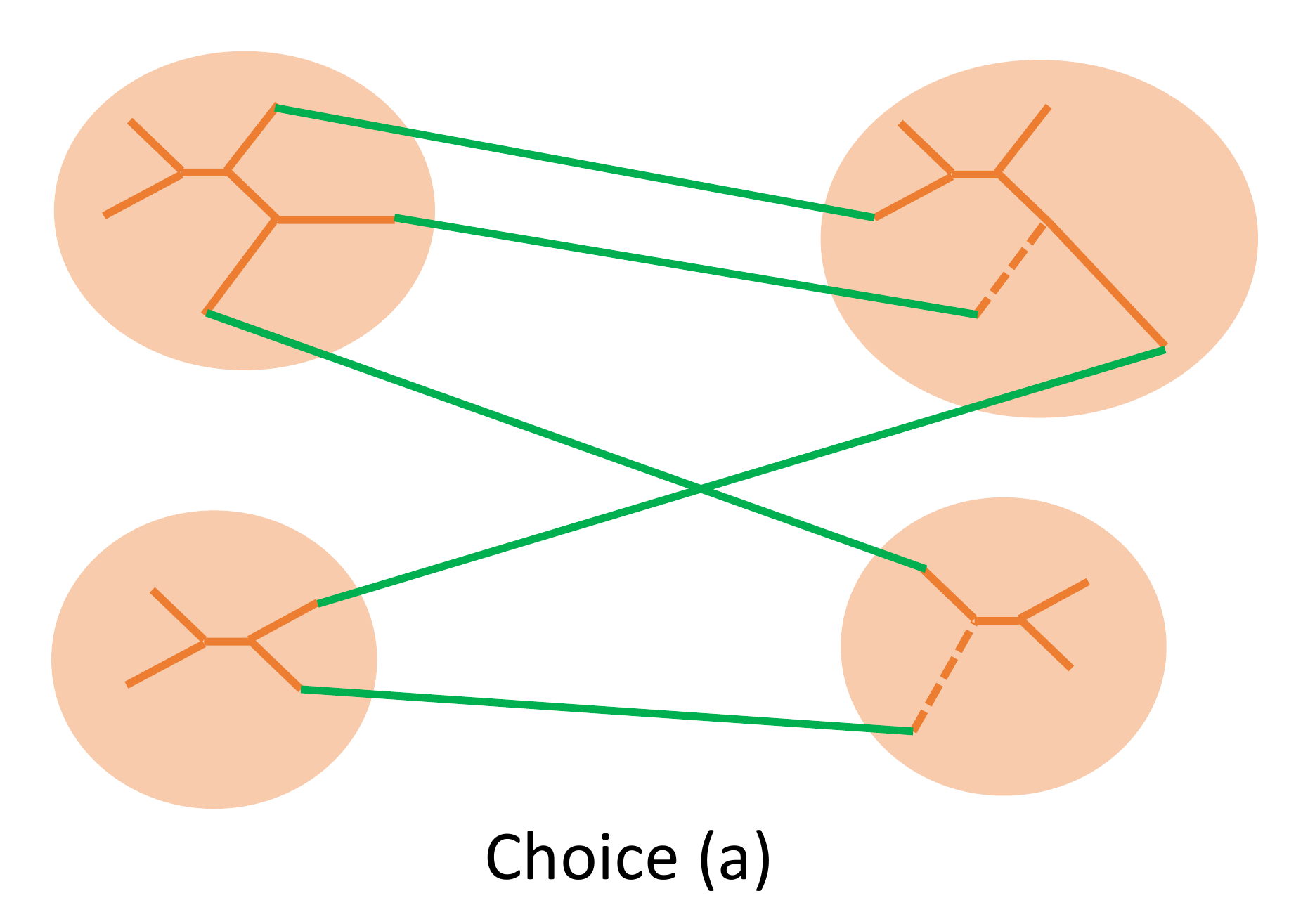} \includegraphics[width=3in]{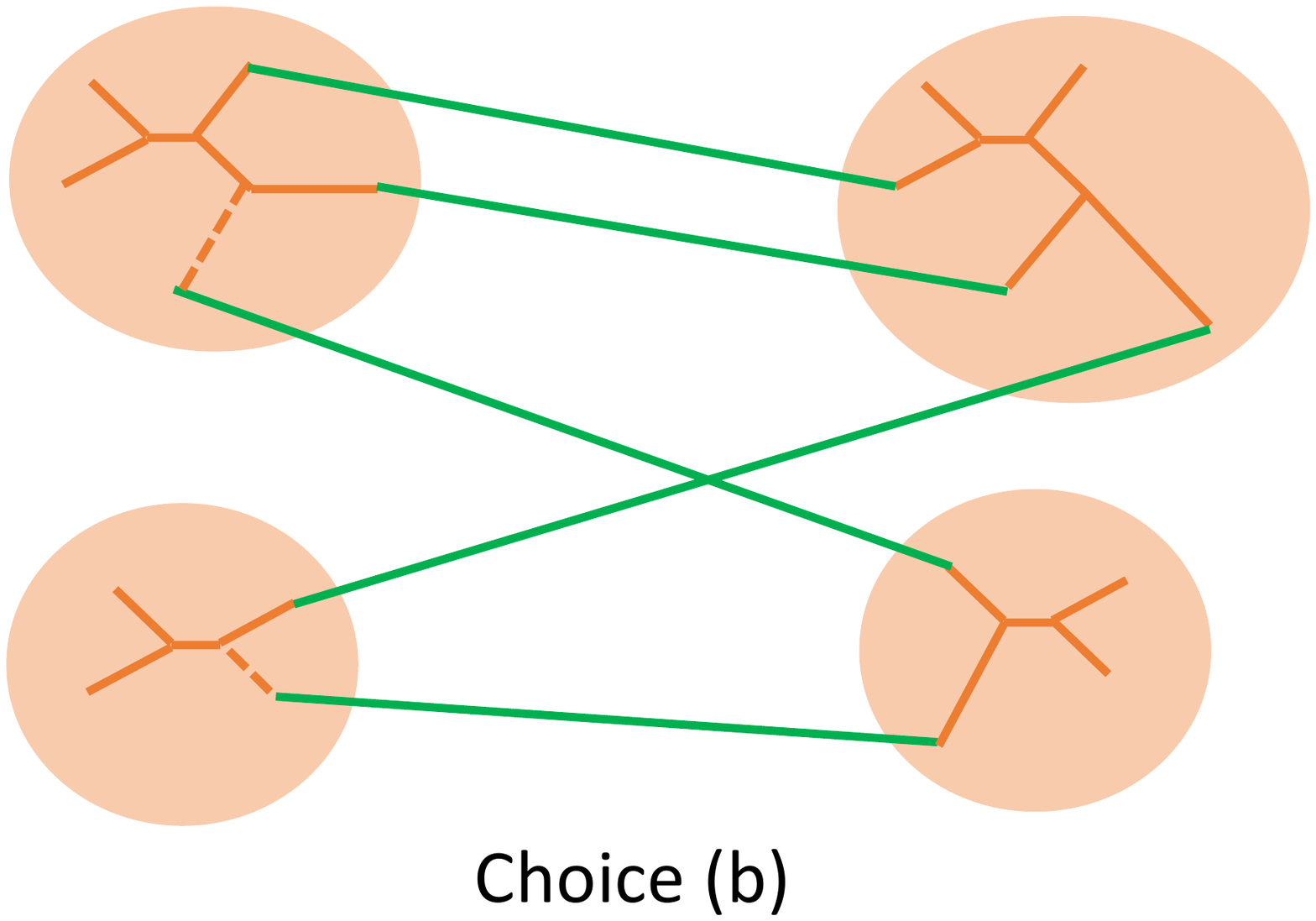} 
\caption{\small 
Two sets $A$ (on the left in both cases (a) and (b)) and $A'$ (on the right in both cases (a) and (b)) are connected by green lines representing identification of particles.
(a)~The skeleton is the set of all orange edges in $\eta_i \cap A$, and   the minimally connected graph~$T_{\eta_i}^{A}$ is obtained by discarding   the orange edges in $A'$  which create cycles in the graph (dotted orange edges). 
(b)~The skeleton  is the set of all orange edges in $\eta_i \cap A'$, and the minimally connected graph~$T_{\eta_i}^{A'}$ is obtained by discarding  the  orange edges in $A$  which create cycles in the graph (dotted orange edges). 
}
 \label{fig-T-eta}
\end{figure}
We end up with a forward  tree~$T_{\eta_i}^{A}$ on the set of vertices  $Z^{(\eta_i)}_{M^{r\delta}_{\eta_i}}$ encoding some of the dynamical constraints of the orange forward trees, which will produce small factors.  We stress that by construction, the admissible tree graphs $T_{\eta_i}^{A}$ depend on $\Lambda,\Lambda',\chi_\ell$ and on the arbitrary choice of constructing the skeleton over $A\cap \eta_i$. The superscript $A$ in $T_{\eta_i}^{A}$ is a shortened notation reminding us of this dependence. 

Thus in \eqref{eq: E A A'}, for each component $\eta_i$ we have a product of test functions controlled 
{\color{black} by the following estimate, which extends \eqref{eq: phi-geometric-est-sec7 bis} to the case of products with repeated variables}:
\begin{equation}
\begin{aligned}
& \mu_\eps^{-\ell_i} \delta_{\chi_{\ell_i}} \big| \prod_{j \in \eta_i}  \phi^{(\sigma_j)} \big| \left( Z_{\M^{r\delta}_{\eta_i}}\right)\\
&\qquad\leq  
C_{q} \, \left(\prod_{j \in \eta_i} |\sigma_j|^{M_j^{r\delta}}\right) |\eta_i|^{M_{\eta_i}^{r\delta}}
\,{\mu_\eps^{M^{r\delta}_{\eta_i}-|\eta_i|-\ell_i} \over  M^{r\delta}_{\eta_i} !  }\; 
\, \delta_{\chi_{\ell_i}}\,
\sum_{\varsigma, \;\lambda,\; \SS,\;\bar  \SS,\;\KK,\;\EE\atop \varsigma', \lambda', \SS',\bar  \SS',\KK',\EE'} \,   
\sum_{ T_{\eta_i}^A} 
\indc_{\{ Z_{\M^{r\delta}_{\eta_i}} \;  \in \; \cR_{T^A_{\eta_i}}^{\rm comp}  \}    } 
\end{aligned}
\label{proof cov clu}
\end{equation}
where $\left( \varsigma, \;\lambda,\; \SS,\;\bar  \SS,\;\KK,\;\EE\right)$ and
$\left(\varsigma', \lambda', \SS',\bar  \SS',\KK',\EE' \right)$ are the whole collections of variables
which are necessary to 
parametrise the orange clusters in $A$ and in $A'$ respectively. Here, $\cR_{T^A_{\eta_i}}^{\rm comp}$ is the corresponding set of compatible configurations for a given ordered tree graph on~$M^{ r\delta}_{\eta_i}-\ell_i$ vertices. Therefore for each orange edge in $T^A_{\eta_i}$, there exist two particles which will be dynamically constrained by encounters, according to the specified forward dynamics (and respecting the time sampling).

\medskip
\noindent
\underline{Step 3}. To control the background particles $\bar Z_p$ in \eqref{eq: E A A'}, we use  a cluster expansion of the exclusion $\indc_{\cD^\eps_{{ M}^{r\delta}_A+M^{r\delta}_{A'}-\ell+p}}$ as in \eqref{eq: cluster expansion}.
We   consider now  $\left( {\color{black} Z^{(\eta_i)}_{M^{r\delta}_{\eta_i}}} \right)_{1\leq i \leq s} $   as $s$ blocks represented each by one vertex, and $(\bar z_j)_{1\leq j\leq p}$ as $p$ separate vertices. 
We then have 
\begin{equation}
\label{eq: var step3}
\begin{aligned}
  \frac{1}{\cZ^\eps}\sum_{p\geq 0 } & \frac{\mu_\eps^{p}}{p!} \int  d\bar Z_p  \indc_{\cD^\eps_{{ M}^{r\delta}_A + M^{r\delta}_{A'} -\ell +p}} 
  (Z_{{ M}^{r\delta}_A}, Z'_{M^{r\delta}_{A'}},\bar Z_p) \cM^{\otimes p} (\bar V_{p} ) \\
 &=   \prod_{i = 1} ^s \indc_{\cD^\eps_{M^{r\delta}_{\eta_i}-\ell_i}} \left( Z^{(\eta_i)}_{M^{r\delta}_{\eta_i}} \right) 
\sum_{\omega \in \cP_s}
\prod_{u=1}^{|\omega|} \left( \sum_{n_u \geq 0} \frac{\mu_\e^{n_u}}{n_u !}
\int \cM^{\otimes n_u} \gp( Z_{\omega_u},  \bar Z_{n_u})d \bar Z_{n_u}\right)\;, \\
\end{aligned}
\end{equation}
where  $Z_{{\omega_u}} = \left( Z^{(\eta_i)}_{M^{r\delta}_{\eta_i}}\right)_{i \in \omega_u}$.

\begin{figure}[h]\centering
\includegraphics[width=4in]{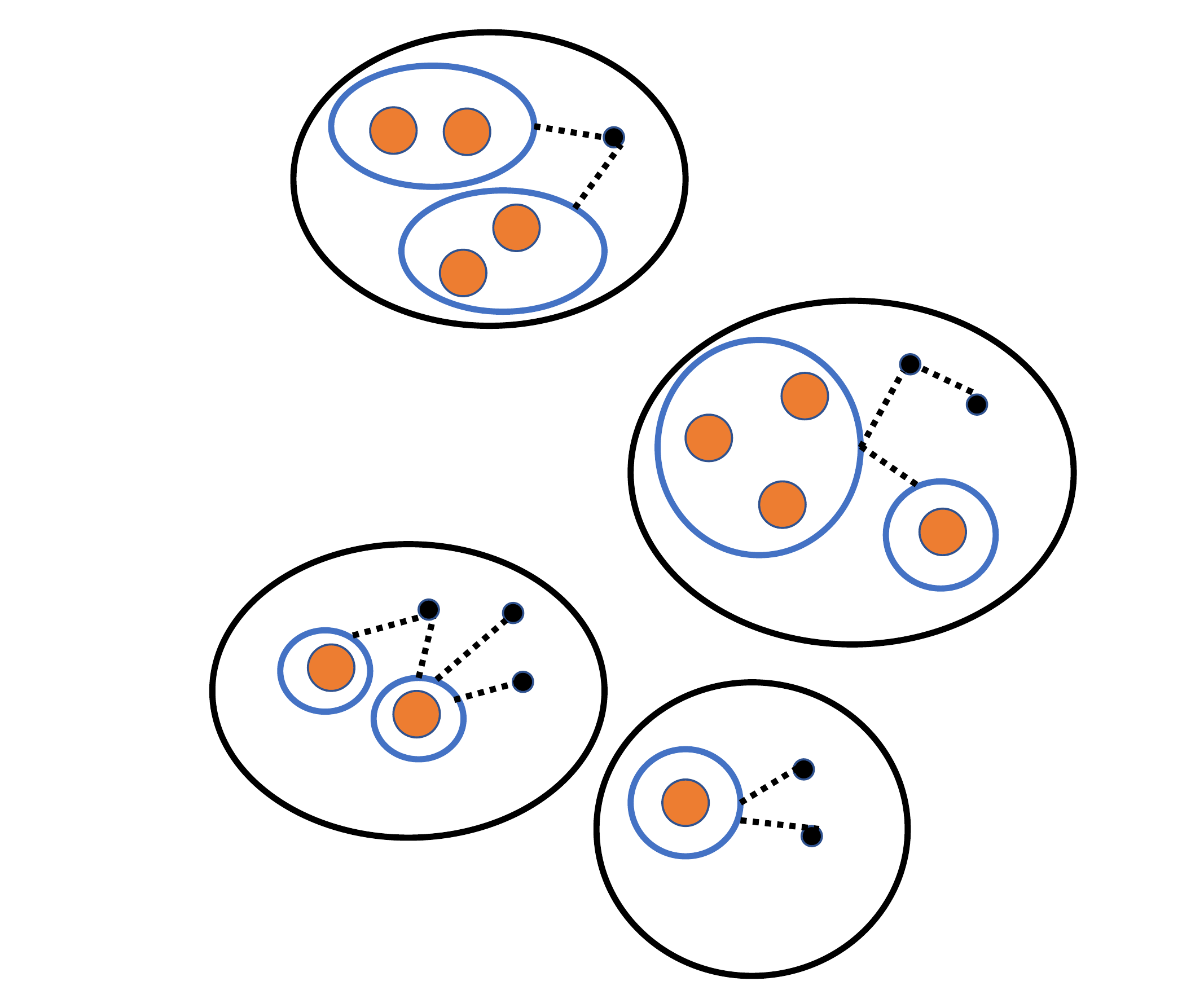} 
\caption{\small 
The $s$ components $\eta_1,\dots ,\eta_s$ (represented in blue) are grouped in {\color{black} $|\omega|$} parts (represented in black) according to the partition $\omega$, and each of these parts $\omega_u$  is provided with an arbitrary number $n_u$ of background particles (black dots). In each part, because of the tree inequality, all vertices are connected by a tree (represented by the dotted black lines).}
 \label{fig-sigma}
\end{figure}

\medskip
\noindent
\underline{Step 4}. By Fubini, we finally get  from \eqref{eq: var prel}, \eqref{eq: E A A'} and \eqref{eq: var step3} that
$$
\begin{aligned} 
   & \bbE_\eps ^{\rm eq} \left[ \mu_\eps ^{q }  \left(\Otimes _{j=1} ^q  \left(   {1\over \mu_\eps^{M^{r\delta}_j} }\sum \phi^{(\sigma_j)}  - \bbE_\eps [\phi^{(\sigma_j)}]  \right) \right)^2\right]  
    =  
    \mu_\eps ^q\sum_{A\subset \{1, \dots, q\} \atop A' \subset \{1',\dots, q'\}}  
    \prod_{ j \in A^c  \cup \left(A'\right)^c  } \bbE_\eps  [-\phi^{(\gs_j)}]  \\
& \qquad \times   
\sum_{ \omega \in    \cP_{A \cup A' }}  \prod_{u = 1} ^{|\omega|}  \left[  \sum_{n_u \geq 0} \frac{\mu_\e^{n_u}}{n_u !} 
\sum_{ \eta \in   \cP_{\omega_u } }
\sum_{\ell_i \geq |\eta_i| - 1 \atop 1 \leq i \leq |\eta|} {\color{black} \sum_{\Lambda_i,\Lambda'_i,\chi_{\ell_i} \atop 1 \leq i \leq |\eta|} }
  \right.\\
    & \qquad 
   \times \left.  \int \cM^{\otimes (M^{r\delta}_{\omega_u}   - \sum_{i=1}^{|\eta|}\ell_i+ n_u)}  \gp(Z_{\eta_1}, Z_{\eta_2},\dots  ,\bar Z_{n_u})  
 \prod_{i=1}^{|\eta|} \Big( \indc_{\cD^\eps_{M^{r\delta}_{\eta_i}-\ell_i}}   \mu_\eps^{-\ell_i} \delta_{\chi_{\ell_i}} 
\prod_{j \in \eta_i }  \phi^{(\sigma_j)} dZ_{\eta_i}\Big)  d \bar Z_{n_u} 
\right] \;. 
\end{aligned}
$$
The set $\omega_u$ corresponds to {\color{black} $M^{r\delta}_{\omega_u}-\sum_{i=1}^{|\eta|}\ell_i$} particles, to which $n_u$ background particles are added, as depicted in Figure \ref{fig-sigma}. Here we denote abusively by $\eta=\{\eta_1,\dots,\eta_{|\eta|}\}$ the generic partition of one $\omega_u$.
There are $ |\eta|$ components in the  partition $\eta$ and, in each component~$\eta_i$, {\color{black} we recall that $\ell_i \geq | \eta_i|- 1$ denotes the number of green edges and $\delta_{\chi_{\ell_i}}$ the identification of  the $\ell_i$ particles in $\Lambda'_i$ with the $\ell_i$ particles in $\Lambda_i$.}

{\color{black}
Using~(\ref{eq:singleton}) and proceeding as in~\eqref{eq:abssinglexp}, we split any partition $\omega$ of $A \cup A'$ in a union of singletons $\{j\}$ for $j \in (A\setminus B)\cup (A' \setminus B')$, and a partition $\tilde \omega$ of  $B\cup B'$ with no singleton (and at most $\frac12 ( |B| +|B'|)$ parts). Compared with the previous situation, we cannot absorb the sum over singletons due to the defect of centering, and we have (noticing that $\ell_j=0$ for singletons)
 \begin{equation}
\begin{aligned} 
& \bbE_\eps^{\rm eq}  \left[ \mu_\eps ^{q }  
\left(\Otimes _{j=1} ^q  \left(   {1\over \mu_\eps^{M^{r\delta}_j} }\sum \phi^{(\sigma_j)}  - \bbE_\eps [\phi^{(\sigma_j)}]  \right) \right)^2\right] \\
& =  \mu_\eps^q\sum_{B\subset \{1, \dots, q\} \atop B' \subset \{1',\dots, q'\}} 
 \left( \prod_{ j \in ^c B \cup ^c B'} \Big (\bbE_\eps^{\rm eq}  [  \phi^{(\sigma_j)}]  - \bbE_\eps  [  \phi^{(\sigma_j)}] \Big)\right)  \sum_{   \omega \in    (\cP_{B \cup B'})^*}  \prod_{u = 1} ^{|   \omega|}  \left[  \sum_{ \eta \in   \cP _{|\omega_u| } }\sum_{\ell_i \geq |\eta_i| - 1 \atop 1 \leq i \leq |\eta|} \sum_{\Lambda_i,\Lambda'_i,\chi_{\ell_i} \atop 1 \leq i \leq  |\eta|} \right.\\
  &\qquad \times  \sum_{n_u \geq 0} \frac{\mu_\e^{n_u}}{n_u !}    \int \cM^{\otimes (M^{r\delta}_{\omega_u}   - \sum_{i=1}^{|\eta|}\ell_i+ n_u)} \gp(Z_{\eta_1}, \dots, Z_{\eta_{|\eta|}},  \bar Z_{n_u})  \\ &\qquad  \times\left. \prod_{i=1}^{|\eta|} \Big( \indc_{\cD^\eps_{M^{r\delta}_{\eta_i}-\ell_i}}   \mu_\eps^{-\ell_i} \delta_{\chi_{\ell_i}} 
\prod_{j \in \eta_i }  \phi^{(\sigma_j)} dZ_{\eta_i}\Big)  d \bar Z_{n_u}\right]
\;.
 \end{aligned}
 \label{main step computation}
\end{equation}

{\color{black}

\medskip
To estimate from above the latter formula,  some of the   constraints on  the clustering structure can be forgotten. 
Indeed we know from Step 2 that, to each component $\eta_i$ and each~$\chi_{\ell_i}$, we can associate a minimally connected graph
{$T_{\eta_i}^B$},  encoding dynamical constraints associated  with orange edges: see Eq.\,\eqref{proof cov clu} and Figure \ref{fig-T-eta}.
The next and final step will be to integrate these dynamical constraints. At this stage, the assumptions from Lemma  
\ref{phi-sigma-lemma} on the cumulant structure will become relevant to describe  precisely  the set $\cR_{T_{\eta_i}^B}^{\rm comp}$ including the time sampling.

\medskip
First of all, we proceed by estimating the integral over the background particles as already done in Section \ref{sec: Expectation of centered Otimes-products}. For each $\omega_u$ and $ \eta \in   \cP_{\omega_u}$, the functions $\varphi$ can be controlled by the tree inequality \eqref{eq:treeineq}, this time applied over the vertices $\left(Z_{\eta_i}\right)_{1 \leq i \leq |\eta|}$ (considered as subconfigurations) and the $n_u$ background particles. Using \eqref{proof cov clu} and the translation invariance, we obtain 
(as in \eqref{eq: facteur L1}-\eqref{eq: facteur L1'''})
that the term in the last two  lines  in \eqref{main step computation} is bounded in absolute value by
 \begin{equation}
\begin{aligned} 
 & C_P C_P^{M^{r\delta}_{\eta}}  (C \eps^d)^{|\eta|- 1}   
 \prod_{i=1}^{|\eta|}\,{\mu_\eps^{M^{r\delta}_{\eta_i}-|\eta_i|-\ell_i} \over   M^{r\delta}_{\eta_i}!  }\\
 & \qquad
\sum_{\varsigma, \;\lambda,\; \SS,\;\bar  \SS,\;\KK,\;\EE\atop \varsigma', \lambda', \SS',\bar  \SS',\KK',\EE'}
 \sum_{T_{\eta_i}^B}
\sup\int \indc_{\cD^\eps_{M^{r\delta}_{\eta_i}-\ell_i} }\left( Z_{\eta_i}\right) 
\cM^{\otimes (M^{r\delta}_{\eta_i} - \ell_i)}
\indc_{\{ Z_{\M^{r\delta}_{\eta_i}} \;  \in \; \cR_{T_{\eta_i}^B}^{\rm comp}  \}    }
\delta_{\chi_{\ell_i}}  dZ_{\eta_i}\;.
\end{aligned}
  \label{main step computation'}
\end{equation}
 As in \eqref{eq: facteur L1},  \eqref{proof cov clu}, the first sum is taken over the whole collections  parametrising the forward dynamics  (and the sums over such parameters can be bounded by $C_P^{M^{r\delta}_{\eta}}$ and combinatorial factors depending only on $q\leq P$); moreover the supremum is taken over one single positional variable and, for brevity, $Z'_{\eta_i}$ is the configuration $Z_{\eta_i}$ deprived of such variable.

The remaining integral is now estimated in a similar way as in the proof of \eqref{expectation-phi-sigma bis} above. We may follow the same strategy devised in Section \ref{geometric-sec}, ordering the orange edges in time in a way to respect the sampling, and identifying a sequence of independent degrees of freedom which can be progressively estimated; see \eqref{eq:relpos}-\eqref{eq: inegalite exponentielle}. However, particles in the skeleton play a special role, as explained in what follows.

Recall that, by Step 2,   the tree $T_{\eta_i}^B$ is constructed asymmetrically 
so that the union of the skeletons $\bigcup_i T_{\eta_i}^{1B}$ records all the dynamical constraints in $( \phi^{(\sigma_j)} )_{ j \in B}$ (Figures \ref{fig-eta}-\ref{fig-T-eta}).
For these edges, we proceed exactly as in Section \ref{geometric-sec} and recover a bound of the form \eqref{eq: integration Bi}-\eqref{eq: time constraints 4.1} taking into account the time sampling. Instead, the orange edges which are outside the skeleton are estimated more crudely, discarding the dynamical constraints associated with the sampling. This leads to the estimate
\begin{equation}
\begin{aligned} 
& \sum_{T_{\eta_i}^B}\sup\int \indc_{\cD^\eps_{M^{r\delta}_{\eta_i}-\ell_i}} \left( Z_{\eta_i}\right) 
\cM^{\otimes (M^{r\delta}_{\eta_i}-\ell_i)}
\indc_{\{ Z_{\M^{r\delta}_{\eta_i}} \;  \in \; \cR_{T_{\eta_i}^B}^{\rm comp}  \}    }
\delta_{\chi_{\ell_i}}  dZ_{\eta_i} \leq C \left( \frac{1}{\mu_\eps}\right)^{M^{ r\delta}_{\eta_i}-\ell_i -1}  \\
& \qquad \times
\prod_{j \in \eta_i\cap B}
 \left(  M^{r \delta}_{j} \right)^{2(M^{r \delta}_{j}-1)} 
 \frac{3^{M^{ r\delta}_{j}-1}}{\left(M^{ r\delta}_{j} -1\right) !} \, 
 (C_P \delta)^{N^r_{\sigma_j}} \, 
  (C_P\tau)^{N_{\sigma_j}^{<r}+|\sigma_j| -1}
   \, (C_P \Theta)^{M_{j}-|\sigma_j|}\\
   & \qquad \times
 \left( M^{r \delta}_{\eta_i\cap {B'}}\right)^{2(M^{r \delta}_{\eta_i\cap {B'}} -\ell_i)}
\frac{(C_P \Theta)^{M^{r \delta}_{\eta_i\cap {B'}}-\ell_i}}{\left( M^{r \delta}_{\eta_i\cap {B'}}-\ell_i \right) !}
 \end{aligned}
   \label{main step computation''}
\end{equation}
Note that one factor $\left(M^{r \delta}_{j} \right)^{M^{r \delta}_{j}-1} $ is compensated by the factorial at the denominator in the second line, and the same can be said for one factor $ \left( M^{r \delta}_{\eta_i\cap {B'}}\right)^{M^{r \delta}_{\eta_i\cap {B'}} -\ell_i}$ in the third line. Hence, enlarging the constants, we get
\begin{equation}
\begin{aligned} 
& \sum_{T_{\eta_i}^B} \sup \int \indc_{\cD^\eps_{M^{r\delta}_{\eta_i}-\ell_i}} \left( Z_{\eta_i}\right) 
\cM^{\otimes (M^{r\delta}_{\eta_i}-\ell_i)}
\indc_{\{ Z_{\M^{r\delta}_{\eta_i}} \;  \in \; \cR_{T_{\eta_i}^B}^{\rm comp}  \}    }
\delta_{\chi_{\ell_i}}  dZ_{\eta_i}  \leq C \left( \frac{1}{\mu_\eps}\right)^{M^{ r\delta}_{\eta_i}-\ell_i -1}  \\
& \qquad \times
 \left(  M^{r \delta}_{\eta_i} \right)^{M^{r \delta}_{\eta_i}} 
 \left( M^{r \delta}_{\eta_i\cap {B'}}\right)^{-\ell_i}
 \left(\prod_{j \in \eta_i\cap B}
  (C_P\delta)^{N^r_{\sigma_j}} \, 
  (C_P\tau)^{N_{\sigma_j}^{<r}+|\sigma_j| -1}
   \,  (C_P\Theta)^{M_{j}-|\sigma_j|}\right)(C_P\Theta)^{M^{r \delta}_{\eta_i\cap {B'}}-\ell_i}
 \end{aligned}
   \label{main step computation''''}
\end{equation}
where the factor $ \left(  M^{r \delta}_{\eta_i} \right)^{M^{r \delta}_{\eta_i}} $ compensates, up to geometric terms,
the factorial in  \eqref{main step computation'}.
On the other hand, the number of possible contractions at $\ell_i$ fixed is
\begin{equation}
\begin{aligned} 
\sum_{\Lambda_i,\Lambda'_i,\chi_{\ell_i} \atop 1 \leq i \leq |\eta|} 1 
= \binom{M^{r\delta}_{\eta_i \cap B} }{\ell_i}
\binom{M^{r\delta}_{\eta_i\cap {B'}}}{\ell_i} \ell_i! \leq 2^{M^{r\delta}_{\eta_i\cap B}}
 \left(M^{r\delta}_{\eta_i\cap B'}\right)^{\ell_i}
\end{aligned}
   \label{main step computation'''''}
\end{equation}
which compensates the factor  $\left( M^{r \delta}_{\eta_i\cap {B'}}\right)^{-\ell_i}$ in \eqref{main step computation''''}.
Therefore by \eqref{main step computation}, \eqref{main step computation'}, \eqref{main step computation''''} and \eqref{main step computation'''''}
we deduce that (recalling $\sum_i |\eta_i| = |B|+|B'|$)
\begin{equation}
\begin{aligned} 
& \bbE_\eps^{\rm eq} \left[ \mu_\eps ^{q }  
\left(\Otimes _{j=1} ^q  \left(   {1\over \mu_\eps^{M^{r\delta}_j} }\sum \phi^{(\sigma_j)}  - \bbE_\eps [\phi^{(\sigma_j)}]  \right) \right)^2\right] \\
&\qquad \leq C_P \mu_\eps^q\sum_{B\subset \{1, \dots, q\}} \sum_{B' \subset \{1',\dots, q'\}} 
\mu_\eps^{- |B|-|B'|} \left( \prod_{ j \in ^c B \cup ^c B'}\left|  \bbE_\eps^{\rm eq}  [ {\mathbf 1}_{^c\Upsilon^\eps_\cN}\phi^{(\sigma_j)}]\right|  \right) 
\sum_{  \omega \in   ( \cP_{  B\cup B'} )^*}  \prod_{u = 1} ^{ |\omega|}  
\sum_{ \eta \in   \cP_{\omega_u} } \\
 &  \qquad\qquad \times   \,
\left( C\e^d \mu_\e\right)^{|\eta|-1} \mu_\e
  \left(\prod_{j \in \eta^B_i}   
  (C_P\delta)^{N^r_{\sigma_j}} \, 
  (C_P\tau)^{N_{\sigma_j}^{<r}+|\sigma_j| -1}
   \,  (C_P\Theta)^{M_{j}-|\sigma_j|}\right)(C_P\Theta)^{M^{r \delta}_{\eta_i\cap {B'}}}\\
    \end{aligned}
    \label{eq:conclusion8}
\end{equation}
for some constant $C_P$ as in the statement of Lemma \ref{phi-sigma-lemma} and some pure  constant $C$.
As $\eta_i\cap B $ can be replaced by $\eta_i\cap B' $ by symmetry, we also deduce that
\begin{equation}
\begin{aligned} 
& \bbE_\eps^{\rm eq} \left[ \mu_\eps ^{q }  
\left(\Otimes _{j=1} ^q  \left(   {1\over \mu_\eps^{M^{r\delta}_j} }\sum \phi^{(\sigma_j)}  - \bbE_\eps [\phi^{(\sigma_j)}]  \right) \right)^2\right] \\
&\qquad \leq C_P \mu_\eps^q\sum_{B\subset \{1, \dots, q\}} \sum_{B' \subset \{1',\dots, q'\}} 
\mu_\eps^{- |B|-|B'|} \left( \prod_{ j \in ^c B \cup ^c B'}\left|  \bbE_\eps^{\rm eq}  [ {\mathbf 1}_{^c\Upsilon^\eps_\cN}\phi^{(\sigma_j)}]\right|  \right) \sum_{  \omega \in    \cP_{  B\cup B'}^*}  
\prod_{u = 1} ^{ |\omega|}  
\sum_{ \eta \in   \cP_{\omega_u} } \\
 &  \qquad\qquad \times   \,
\left( C'\e^d \mu_\e\right)^{|\eta|-1} \mu_\e
  \left(\prod_{j \in \eta_i} 
  (C_P\delta)^{N^r_{\sigma_j}} \, 
  (C_P\tau)^{N_{\sigma_j}^{<r}+|\sigma_j| -1}
   \,  (C_P\Theta)^{2M_{j}+N_{\sigma_j}^{<r}+N^r_{\sigma_j}-|\sigma_j|}\right)^{1/2}\;.
    \end{aligned}
    \label{eq:conclusion8+MG}
\end{equation}

Recall that $C\e^d \mu_\e = C\e$, so that we obtain a rough upper bound
\begin{equation} 
\begin{aligned}
&\prod_{u = 1} ^{ |\omega|}  
\sum_{ \eta \in   \cP_{\omega_u} } 
\left( C'\e^d \mu_\e\right)^{|\eta|-1} \mu_\e
  \left(\prod_{j \in \eta_i} 
  (C_P\delta)^{N^r_{\sigma_j}} \, 
  (C_P\tau)^{N_{\sigma_j}^{<r}+|\sigma_j| -1}
   \,  (C_P\Theta)^{2M_{j}+N_{\sigma_j}^{<r}+N^r_{\sigma_j}-|\sigma_j|}\right)^{1/2} \\
   &\quad  \leq  \mu_\eps^{|\omega|} \left( \prod_{j \in B\cup B' }  (C_P\delta)^{N^r_{\sigma_j }} \, 
  (C_P\tau)^{N_{\sigma_j}^{<r}+|\sigma_j| -1}
   \,  (C_P\Theta)^{2M_{j}+N_{\sigma_j}^{<r}+N^r_{\sigma_j}-|\sigma_j|}\right)^{1/2}\,.
   \end{aligned}
   \label{eq:conclusion8+MG'}
\end{equation}

We are left with the cost of the conditioning in the singletons
\begin{equation}
 \bbE_\eps^{\rm eq}  [ {\mathbf 1}_{^c\Upsilon^\eps_\cN}\phi^{(\sigma_j)}]
 \equiv \bbE^{\rm eq}_\eps [{\mathbf 1}_{ ^c \Upsilon^\eps_\cN}\,  \pi^\eps _{M_{\sigma_j} ^{r\delta}}(\phi^{(\sigma_j) })]
\end{equation}
which, recalling that
$
 \pi^\eps _{M_{\sigma_j} ^{r\delta}}(\phi^{(\sigma_j) })=
\bbE^{\rm eq}_\eps\Big[ \pi^\eps _{M_{\sigma_j} ^{r\delta}}(\phi^{(\sigma_j) })  \Big] +
 \mu_\eps^{-1/2} \zeta^{\eps,\rm eq}_{M_{\sigma_j} ^{r\delta} } 
 \big(\phi ^{ (\sigma_j )  } \big)$,
is bounded as follows:
\begin{equation}
\label{eq:estcondphisj}
\begin{aligned}
& \left|\bbE^{\rm eq}_\eps [{\mathbf 1}_{ ^c \Upsilon^\eps_\cN}\,  \pi^\eps _{M_{\sigma_j}^{r\delta}}(\phi^{(\sigma_j) })] \right|
  = 
\left|\bbP^{\rm eq}_\eps\Big[  {^c \Upsilon^\eps_\cN}   \Big] 
 \bbE^{\rm eq}_\eps\Big[\phi^{(\sigma_j) }  \Big]  +  \mu_\eps^{-1/2}
\bbE_\eps^{\rm eq}\Big[ {\mathbf 1}_{^c \Upsilon^\eps_\cN}  \;  \zeta^{\eps,\rm eq}_{M_{\sigma_j} ^{r\delta} } 
 \big(\phi ^{ (\sigma_j )  }\big)  \Big] \,  \right| \\
 & \qquad\leq  \bbP^{\rm eq}_\eps\Big[  {^c \Upsilon^\eps_\cN}   \Big]  \bbE^{\rm eq}_\eps\Big[\left|\phi^{(\sigma_j) }\right|  \Big]
 +  \mu_\eps^{-1/2} \bbP^{\rm eq}_\eps\Big[  {^c \Upsilon^\eps_\cN}   \Big]^{1/2}
 \bbE_\eps^{\rm eq}\Big[ 
 \left(
 \zeta^{\eps,\rm eq}_{M_{\sigma_j} ^{r\delta} }\big(\phi^{ (\sigma_j )}\big)
 \right)^2   
 \Big]^{1/2} \\
 & \qquad \leq C_P  
 \Big[  \Theta \e^d 
  (C_P\Theta)^{M_j - |\sigma_j|}  (C_P \delta)^{ N_{\sigma_j }^{r}}  (C_P \tau)^{N_{\sigma_j}^{<r}+|\sigma_j | - 1} \\
 & \qquad \qquad + \mu_\eps^{-1/2}\left(\Theta\e^d\right)^{1/2}
 \Big( (C_P \Theta)^{ 2M_{j}  + N_{\sigma_j }^r + N_{\sigma_j }^{<r} - |\sigma_j|  }  (C _P\delta)^{ N_{\sigma_j}^{r }}  (C _P\tau)^{N_{\sigma_j }^{<r}+|\sigma_j | - 1} \Big)^{1/2}\Big]\;.
\end{aligned}
\end{equation}
In the last step we used \eqref{conditioning-est}, together with \eqref{expectation-phi-sigma-eq} to bound the first term, and the estimate \eqref{variance-phi-sigma} at equilibrium in the case of one single factor to bound the second term (this estimate has been proved in \cite{BGSS2} and follows also from the previous computation). Notice that the second term is dominant for $\mu_\e$ large (and $\Theta>1$). We will actually only keep the rough estimate
\begin{equation}
\label{eq: addition eps}
\left|\bbE^{\rm eq}_\eps [{\mathbf 1}_{ ^c \Upsilon^\eps_\cN}\,  \pi^\eps _{M_{\sigma_j}^{r\delta}}(\phi^{(\sigma_j) })] \right|\leq  \mu_\eps^{-1/2}
 \Big( (\Theta \eps^d) (C_P \Theta)^{ 2M_{j}  + N_{\sigma_j }^r + N_{\sigma_j }^{<r} - |\sigma_j|  }  (C _P\delta)^{ N_{\sigma_j}^{r }}  (C _P\tau)^{N_{\sigma_j }^{<r}+|\sigma_j | - 1} \Big)^{1/2} \;.
\end{equation}
When inserting this into \eqref{eq:conclusion8+MG}-\eqref{eq:conclusion8+MG'}, we obtain the following power counting
\begin{equation}
\label{eq: power rough counting}
\mu_\eps^q \mu_\eps^{- |B|-|B'|} \mu_\e^{|\omega|} \mu_\eps^{-(|^c B| + |^c B'|)/2}
=  \mu_\eps^{-(|B| + |B'|)/2}\mu_\e^{|\omega|} \leq 1
\end{equation}
because the partition in $\omega$ has no singleton.
Notice  from \eqref{eq: addition eps} that each contribution in $^c B, ^c B'$ has an additional 
factor $\eps^{d/2}$ so that the leading order terms in the power counting \eqref{eq: power rough counting}
are associated with $|B| = |B'| = q$ and with partitions $\omega$ which are pairing of the sets, as expected from the Gaussian asymptotics.

In conclusion, we arrive to
\begin{equation}
\label{eq:concsec82}
\begin{aligned} 
& \bbE_\eps^{\rm eq} \left[ \mu_\eps ^{q }  
\left(\Otimes _{j=1} ^q  \left(   {1\over \mu_\eps^{M^{
r\delta}_j} }\sum \phi^{(\sigma_j)}  - \bbE_\eps [\phi^{(\sigma_j)}]  \right) \right)^2\right] \\
 & \qquad \leq C_P  \prod_{j=1}^q 
 \left( 
  (C_P\Theta)^{2M_{j}+N_{\sigma_j}^{<r}+N^r_{\sigma_j}-|\sigma_j|}
 (C_P\delta)^{N^r_{\sigma_j}} (C_P\tau)^{N_{\sigma_j}^{<r}+|\sigma_j| -1} \right)\;.
    \end{aligned}
\end{equation}

\medskip

A similar proof leads to the estimate
\begin{equation}
\label{variance-phi-recSec8}
\begin{aligned}
&\bbE_\eps^{\rm eq}\left[ \mu_\eps ^{q }  \left(
\Otimes _{j=1 \atop j\neq i} ^q  \left(   {1\over \mu_\eps^{M^{r\delta}_j} }\sum \phi^{(\sigma_j)}  - \bbE_\eps [\phi^{(\sigma_j)}]\right)
\Otimes \left(   {1\over \mu_\eps^{M^{r\delta}_i} }\sum \phi^{(\sigma_i), {\rm cyc}}  - \bbE_\eps [\phi^{(\sigma_i), {\rm cyc}}]\right)
\right)^2\right]
 \\
& \qquad \leq C_P \eps\delta  | \log \eps | (\Theta |\log \eps |)^{2d+4}   (C_P \Theta)^{ 2M_{i}+N_{\sigma_i}^{<r}+N^r_{\sigma_i}-|\sigma_i|}  (C_P \delta)^{( N_{\sigma_i  }^{r}-1)_+}  (C_P \tau)^{(N_{\sigma_i }^{<r}+|\sigma_i  | - 2)_+}\\
&\qquad\qquad \times \prod_{j \neq i}    (C_P\Theta)^{2M_{j}+N_{\sigma_j}^{<r}+N^r_{\sigma_j}-|\sigma_j|}
 (C_P\delta)^{N^r_{\sigma_j}} (C_P\tau)^{N_{\sigma_j}^{<r}+|\sigma_j| -1}  \, .
\end{aligned}
\end{equation}
To obtain \eqref{variance-phi-recSec8}, we repeat the above argument leading to \eqref{eq:concsec82}. In what follows, we only discuss the main differences. In the derivation of the product bound \eqref{proof cov clu}, we used the elementary estimates  \eqref{eq:combinfacp}, \eqref{eq: phi-geometric-est-sec7 bis} where the cluster structure is given by minimally connected graphs. 
In the case of \eqref{variance-phi-recSec8}, one factor of type $\phi^{(\sigma_i), {\rm cyc}}$ is present which satisfies the different estimate
$$
| \phi^{(\sigma_i), {\rm cyc}}(Z_{\M_i^{r\delta}}) | 
\leq C_P \;  {|\sigma_i|}^{M_i^{r\delta}}
\frac{ \mu_\eps^{M_i^{r\delta}-1}}{M_i^{r\delta}!}\,  \sum_{ \bar \SS, \SS, \KK, \varsigma, \lambda, \EE}   \; \sum_{T_\prec \in \cT^\prec_{M_i^{r \delta}}} 
\indc_{\{ Z_{\M_i^{r \delta}} \;  \in \; \cR_{T_\prec}^{{\rm comp, rec}}  \}    } \;.
$$
Here the set $\cR_{T_\prec}^{{\rm comp, rec}}$ is defined as the set $\cR_{T_\prec}^{{\rm comp}}$, with the additional constraint that the graph encoding all encounters in the forward dynamics should contain at least one edge on $\cI_\delta$ and at least one cycle. The construction in step 2 proceeds then as before, but the set $\cR^{{\rm comp}}$ in \eqref{proof cov clu} is replaced by $\cR^{{\rm comp, rec}}$ if the factor $\phi^{(\sigma_i), {\rm cyc}}$ belongs to the skeleton ($i \in A$). This leads to a formula as \eqref{main step computation} where, depending on $B,B'$, we distinguish several possibilities:
\begin{itemize}
\item $i$ belongs to $B$ and $B'$. The estimates \eqref{main step computation''}-\eqref{main step computation''''} are then improved by applying \eqref{eq: integration Bi'} (instead of \eqref{eq: integration Bi}), which uses the reinforced geometric condition on the cycle to bring an additional small factor $\e\delta |\log\e| (\Theta |\log \eps|) ^{2d+4} $. One gets then a contribution as in the right hand side of \eqref{eq:conclusion8+MG} with such an additional smallness.
\item $i$ belongs to $B$ and $^c B'$ (or viceversa). Similarly, \eqref{eq:conclusion8+MG}  is modified by a small factor $(\e\delta |\log\e| (\Theta |\log \eps|) ^{2d+4})^{1/2}$. An even  smaller  factor  $(\eps ^{2d-1} \Theta \e\delta |\log\e| (\Theta |\log \eps|) ^{2d+4})^{1/2}$ is produced by the estimate of  $\bbE_\eps^{\rm eq}  [ {\mathbf 1}_{^c\Upsilon^\eps_\cN}\phi^{(\sigma_i), {\rm cyc}}]$ (performed as in \eqref{eq:estcondphisj}), thanks to \eqref{expectation-phi-rec-eq} and \eqref{variance-phi-rec} in the case of one single factor.
\item $i$ belongs to $^c B$ and $^c B'$. Then we have two factors $\bbE_\eps^{\rm eq}  [ {\mathbf 1}_{^c\Upsilon^\eps_\cN}\phi^{(\sigma_i), {\rm cyc}}]$ estimated as previously.
\end{itemize}
In all the cases we end up with a gain $\e\delta |\log\e| (\Theta |\log \eps|) ^{2d+4} $, which proves \eqref{variance-phi-recSec8}.

\subsection{Conclusion of the proofs}
\label{sec: Conditioned products}

In this section, we shall derive \eqref{E-HM 1}, \eqref{variance-phi-sigma} and
\eqref{variance-phi-rec} from the analogous results obtained above under  the equilibrium measure. 
Finally we will prove Proposition \ref{Proposition - estimates on g0}.

\bigskip
{\it Proof of \eqref{E-HM 1}.} 
Recalling~\eqref{eq: def zeta m diff} there holds
$$
\begin{aligned} 
 \left|\bbE_\eps  \left[  \Otimes_{i=1}^q \zeta^{\eps}_{M_i ^{r\delta} } \left(\phi^{(\sigma_i) } \right)  \right]\right|
 & =  
 \left|\bbE_\eps  \left[  \Otimes_{i=1}^q 
\Big( 
 \zeta^{\eps,\rm eq}_{M_i ^{r\delta} } \left(\phi^{(\sigma_i) } \right) 
 +  \sqrt{\mu_\eps } \,\bbE^{\rm eq}_\eps [{\mathbf 1}_{ ^c \Upsilon^\eps_\cN}\,  \pi^\eps _{M_i^{r\delta}}(\phi^{(\sigma_i) })]
\Big) 
  \right]\right|\\
& \leq \sum_{A \subset \{1,\dots, q\} } 
  \left(
  \left|
\bbE^{\rm eq}_\eps  \left[ \Otimes_{i \in A} \zeta^{\eps,\rm eq}_{M_i ^{r\delta} } \left(\phi^{(\sigma_i) } \right) \right]\right|
+ \left|\bbE^{\rm eq}_\eps  \left[ {\mathbf 1}_{ ^c \Upsilon^\eps_\cN}\Otimes_{i \in A} \zeta^{\eps,\rm eq}_{M_i ^{r\delta} } \left(\phi^{(\sigma_i) } \right) \right]\right|
\right)\\
& \qquad\qquad\qquad \times 
 \prod_{ j \in A^c} \sqrt{\mu_\eps } \,\left|\bbE^{\rm eq}_\eps [{\mathbf 1}_{ ^c \Upsilon^\eps_\cN}\,  \pi^\eps _{M_j ^{r\delta}}(\phi^{(\sigma_j) })]\right|\;.
 \end{aligned}
$$
The second line can be bounded by \eqref{eq:estcondphisj}, while the first term in the first line is bounded by \eqref{E-HM 1-eq}. Finally the second term in the first line is bounded by
\begin{eqnarray*}
& \bbE^{\rm eq}_\eps  \left[{\mathbf 1}_{ ^c \Upsilon^\eps_\cN}\Otimes_{i \in A} \zeta^{\eps,\rm eq}_{M_i ^{r\delta} } \left(\phi^{(\sigma_i) } \right) \right]  \leq  \bbP^{\rm eq}_\eps\Big[ {^c \Upsilon^\eps_\cN} \Big]^{1/2}
\bbE_\eps^{\rm eq}\left[\left( \Otimes_{i \in A} \zeta^{\eps,\rm eq}_{M_i ^{r\delta}  } 
\big(\phi ^{ (\sigma_i )  }   \big) \right)^2 \right]^{1/2}\\
&\qquad \leq \left( \Theta \eps^d\right)^{1/2} 
C_P  \prod_{j \in A}
 \left(  
  (C_P\Theta)^{2M_{j}+N_{\sigma_j}^{<r}+N^r_{\sigma_j}-|\sigma_j|}
 (C_P\delta)^{N^r_{\sigma_j}} (C_P\tau)^{N_{\sigma_j}^{<r}+|\sigma_j| -1} \right)^{1/2}
\end{eqnarray*}
where we used \eqref{conditioning-est} and the analogue of \eqref{eq:concsec82} in the simpler case of centered fluctuations. Using that $\left( \Theta \eps^d\right)^{1/2}  \ll \e$ for $d>2$ we obtain that
\begin{equation}
\left|\bbE_\eps \left[ \Otimes_{i=1}^q  \zeta^{\eps}_{M_{i}^{r\delta}   } 
\big( \phi^{(\sigma_i) } \big)  \right]  \right|
\leq C_q \eps 
\prod_{i=1}^q
 \left( 
  (C_P\Theta)^{2M_{i}+N_{\sigma_i}^{<r}+N^r_{\sigma_i}-|\sigma_i|}
 (C_P\delta)^{N^r_{\sigma_i}} (C_P\tau)^{N_{\sigma_i}^{<r}+|\sigma_i| -1} \right)^{1/2}\;,
\end{equation}
which concludes the proof. \qed

\bigskip
{\it Proof of \eqref{variance-phi-sigma} and \eqref{variance-phi-rec}.} $ $ 
Both estimates  follow immediately as 
$$
\bbE_\eps\Big[ \Big( \Otimes_{i=1}^q \zeta^\eps_{M_{i}^{r\delta}} 
\big(  \phi^{(\sigma_i) } \big) \Big)^2 \Big] 
\leq 
\bbE_\eps^{\rm eq}\Big[ \Big( \Otimes_{i=1}^q \zeta^\eps_{M_{i}^{r\delta}} 
\big(  \phi^{(\sigma_i) } \big) \Big)^2 \Big]\;.
$$
\qed

\bigskip
{\it Proof of Proposition \rm\ref{Proposition - estimates on g0}.} $ $ 
Proceeding as before,
\begin{eqnarray*}
&& \left| \bbE_\eps \Big( \big(\zeta^\eps (h)\big) ^p \Big)\right|
= \left| \bbE_\eps \Big( \left(\zeta^{\eps, \rm eq} (h)
+ \sqrt{\mu_\eps } \,\bbE^{\rm eq}_\eps [{\mathbf 1}_{^c \Upsilon^\eps_\cN}\,  \pi^\eps(h)]
\right) ^p \Big)\right|\\
&&\qquad  \leq \sum_{k=0}^p \binom{p}{k} \left(
 \left|\bbE_\eps^{\rm eq} \Big( \big(\zeta^{\eps,\rm eq} (h)\big) ^k \Big)\right|
 +  \left|\bbE_\eps^{\rm eq} \Big( {\mathbf 1}_{^c \Upsilon^\eps_\cN}\big(\zeta^{\eps,\rm eq} (h)\big) ^k \Big)\right|\right)
 \mu_{\eps}^{\frac{p-k}{2}}
 \left|\bbE^{\rm eq}_\eps [{\mathbf 1}_{^c \Upsilon^\eps_\cN}\,  \pi^\eps(h)]\right|^{p-k}\\
 && \qquad \leq\sum_{k=0}^p \binom{p}{k} \left( 
 C_k \| h \|_\infty^k + \left( \Theta \eps^d\right)^{1/2} \sqrt{C_{2k}}\| h \|_\infty^k
 \right)\\ &&
\qquad \qquad \times \mu_{\eps}^{\frac{p-k}{2}}
 \left(
 \left( \Theta \eps^d\right)  \frac{\bbE_\eps^{\rm eq}\left[\cN \right]}{\mu_\eps}\| h \|_\infty
 + \mu_\eps^{-1/2} \left( \Theta \eps^d\right)^{1/2}\sqrt{C_2}\| h \|_\infty
 \right)^{p-k} \\ && \qquad
 \leq C_p  \| h \|_\infty^p
\end{eqnarray*}
for some constant $C_p >0$, where in the second inequality we used \eqref{eq: moment ordre 2,4 eq}
(derived in Proposition A.1 from \cite{BGSS2}) and \eqref{conditioning-est}. This proves \eqref{eq: moment ordre 2,4}.

To prove \eqref{eq:IPeq-IP}, we write
\begin{eqnarray*}
\bbE_\eps \left[ \prod_{p=1} ^P \zeta^{\eps}_{\theta_p} ( h^{(p)}) \right]
= \sum_{A \subset \{1,\cdots,P\}}\bbE_\eps \left[ \prod_{p \in A} \zeta^{\eps, \rm eq}_{\theta_p} ( h^{(p)}) \right]
\mu_\eps^{|A^c|/2}\left( \prod_{p \in A^c}\bbE^{\rm eq}_\eps [{\mathbf 1}_{^c \Upsilon^\eps_\cN}\,  \pi^\eps(h^{(p)})]\right)\;,
\end{eqnarray*}
from which we get
\begin{align*}
\left| I_P^{\eps,\rm eq} - I_P^{\eps} \right| 
&= \left| \bbE^{\rm eq}_\eps \left[ \prod_{p=1} ^P \zeta^{\eps,\rm eq}_{\theta_p} ( h^{(p)}) \right]
- \bbE_\eps \left[ \prod_{p=1} ^P \zeta^{\eps}_{\theta_p} ( h^{(p)}) \right]
\right| \\
& \leq 
\left| \bbE^{\rm eq}_\eps \left[ {\mathbf 1}_{^c \Upsilon^\eps_\cN}\prod_{p=1} ^P \zeta^{\eps,\rm eq}_{\theta_p} ( h^{(p)}) \right] \right|\\
& \qquad \qquad 
+ \sum_{A \subset \{1,\cdots,P\} \atop A^c \neq \emptyset} \left| 
\bbE_\eps \left[ \prod_{p \in A} \zeta^{\eps, \rm eq}_{\theta_p} ( h^{(p)}) \right]
\mu_\eps^{|A^c|/2}\left( \prod_{p \in A^c}\bbE^{\rm eq}_\eps [{\mathbf 1}_{^c \Upsilon^\eps_\cN}\,  \pi^\eps(h^{(p)})]\right)
\right|\;.
\end{align*}
Using once again \eqref{conditioning-est} and \eqref{eq:estcondphisj}, together with
H\"{o}lder's inequality to bound the moments of fluctuation fields, 
one deduces the estimate
$$ \left| I_P^{\eps,\rm eq} - I_P^{\eps} \right| \leq 
C_P \left(\Theta\e^{d}\right)^{1/2}\;. $$
This concludes the proof of Proposition \ref{Proposition - estimates on g0}. \qed

\end{document}